\documentclass[reqno,11pt,a4paper]{amsart}
\usepackage{amsmath}
\usepackage{amssymb}
\usepackage{graphicx}
\usepackage{mathrsfs}
\usepackage{bm} 
\usepackage{bbm}
\usepackage{amsmath,amssymb,amsthm,amsfonts,bm}
\usepackage{color}
\usepackage{amsthm}
\usepackage[colorlinks,linkcolor=blue,anchorcolor=red,citecolor=blue]{hyperref}
\usepackage{geometry}
\usepackage{fullpage}
\newtheorem{Thm}{Theorem}[section]

\newtheorem{Lem}[Thm]{Lemma}

\newtheorem{Rem}[Thm]{Remark}

\numberwithin{equation}{section}
\newcommand{\1}{\mathbf{1}}
\newcommand{\R}{\mathbb{R}}

\newcommand{\E}{\mathcal{E}}
\newcommand{\D}{\mathcal{D}}
\renewcommand{\L}{\mathcal{L}}

\newcommand{\g}{\mathbf{g}}
\newcommand{\f}{\mathbf{f}}
\newcommand{\G}{\overline{G}}
\newcommand{\ve}{\varepsilon}

\renewcommand{\th}{\theta}
\newcommand{\wt}{\widetilde}
\newcommand{\ol}{\overline}

\newcommand{\pa}{\partial}
\newcommand{\na}{\nabla}

\newcommand{\al}{\alpha}
\newcommand{\lam}{\lambda}
\newcommand{\wh}{\widehat}
\renewcommand{\S}{\mathbb{S}}
\newcommand{\<}{\langle}
\renewcommand{\>}{\rangle}
\newcommand{\T}{\mathbb{T}}

\newcommand{\vertiii}[1]{{\left\vert\kern-0.25ex\left\vert\kern-0.25ex\left\vert #1 \right\vert\kern-0.25ex\right\vert\kern-0.25ex\right\vert}}

\usepackage{cite}

\allowdisplaybreaks

\begin{document}
	
%	\title[Rarefaction Waves for the VPB system]{Stability of Nonlinear Rarefaction Waves for the Vlasov-Poisson-Boltzmann Equation in Rectangular Duct with Polynomial Perturbation}
	\title[Rarefaction Waves for the VPB system]{Stability of Rarefaction Waves for the Non-cutoff Vlasov-Poisson-Boltzmann System with Physical Boundary}
%	%\author{Dingqun Deng     \and
%	%}
	\author{Dingqun Deng}
	\address{Yanqi Lake Beijing Institute of Mathematical Sciences and Applications, Tsinghua University, Beijing, People's Republic of China}
	\curraddr{}
	\email{dingqun.deng@gmail.com}
	\thanks{}
\date{\today}

%	%    author two information
%	
%	
%	%    \subjclass is required.
%	
%	\date{}
%%	\date{Received: date / Accepted: date}
%	
%	\subjclass[2020]{Primary 76P05; Secondary 35Q20, 82C40.}
%	\dedicatory{}

	
	%    Abstract is required. 76P05? 35E05
	\begin{abstract}
		In this paper, we are concerned with the Vlasov-Poisson-Boltzmann (VPB) system in three-dimensional spatial space without angular cutoff in a rectangular duct with or without physical boundary conditions. Near a local Maxwellian with macroscopic quantities given by rarefaction wave solution of one-dimensional compressible Euler equations, we establish the time-asymptotic stability of planar rarefaction wave solutions for the Cauchy problem to VPB system with periodic or specular-reflection boundary condition. In particular, we successfully introduce physical boundaries, namely, specular-reflection boundary, to the models describing wave patterns of kinetic equations. Moreover, we treat the non-cutoff collision kernel instead of the cutoff one. As a simplified model, we also consider the stability and large time behavior of the rarefaction wave solution for the Boltzmann equation.
%		Moreover, we can obtain the time-decay rate of rarefaction wave solution between the VPB system and Euler equations.   
%		The technical tools are macro-micro decomposition for non-cutoff and polynomial weighted cases and the high-order specular-reflection boundary condition for transverse directions. 
%		In this work, we consider three-dimensional Vlasov-Poisson-Boltzmann system in rectangular duct with specular reflection boundary condition near a local Maxwellian, whose fluid quantities is given by the rarefaction wave solution of one-dimensional compressible Euler equations. We are able to prove the time-asymptotic stability of nonlinear rarefaction wave for the Vlasov-Poisson-Boltzmann system near such local Maxwellian with small perturbation. The technical method are improved macro-micro decomposition applicable for non-cutoff case and the high-order specular boundary condition for transverse directions. 
	\end{abstract}

	\maketitle

	\tableofcontents

	\section{Introduction}
	
	\subsection{Models and Spatial Domain}
	In this paper, we consider the planar rarefaction wave for the Vlasov-Poisson-Boltzmann (VPB) system near a local Maxwellian with polynomial perturbations, which describes the motion of plasma particles of two species (e.g. ions and electrons) in domain $\Omega$. The VPB system without angular cutoff in domain $\Omega$ takes the form 
	\begin{equation}\label{1}\left\{
		\begin{aligned}
			&\pa_tF_+ + v\cdot \na_xF_+ - \na_x\phi\cdot\na_vF_+ = Q(F_+,F_+)+Q(F_-,F_+),\\
			&\pa_tF_- + v\cdot \na_xF_- + \na_x\phi\cdot\na_vF_- = Q(F_+,F_-)+Q(F_-,F_-),\\
			&-\Delta_x\phi = \int_{\R^3}(F_+-F_-)\,dv,\\ 
			& F_\pm(0,x,v)=F_{0,\pm}(x,v),\quad E(0,x)=E_0(x).
		\end{aligned}\right. 
	\end{equation}
	Here, the unknown $F_\pm(t,x,v)\ge 0$ is the distribution functions for the particles of ions $(+)$ and electrons $(-)$, respectively, at position $x\in\Omega$ and velocity $v\in\R^3$ and time $t\ge 0$. 
	The self-consistent electrostatic field takes the form $E(t,x)= -\na_x\phi(t,x)$. 
	The non-cutoff Boltzmann collision operator is given by 
\begin{align*}
	Q(G,F) = \int_{\R^3}\int_{\mathbb{S}^{2}} B(v-v_*,\sigma)\big[G(v'_*)F(v')-G(v_*)F(v)\big]\,d\sigma dv_*.
\end{align*}  
where $(v,v_*)$ and $(v',v'_*)$ are velocity pairs given in terms of the $\sigma$-representation by
\begin{align*}
	v'=\frac{v+v_*}{2}+\frac{|v-v_*|}{2}\sigma,\quad v'_*=\frac{v+v_*}{2}-\frac{|v-v_*|}{2}\sigma,\quad \sigma\in\mathbb{S}^2,
\end{align*}
that satisfy 
conservation laws of momentum and energy:
\begin{equation*}
	v+v_*=v'+v'_*,\quad 
	|v|^2+|v_*|^2=|v'|^2+|v'_*|^2.
\end{equation*}
The Boltzmann collision kernel $B(v-v_*,\sigma)$ depends only on $|v-v_*|$ and the deviation angle $\theta$ through $\cos\theta=\frac{v-v_*}{|v-v_*|}\cdot\sigma$. Without loss of generality, we can assume $B(v-v_*,\sigma)$ is supported on $0<\theta\le\pi/2$, since one can reduce the situation with standard {\it symmetrization}: $\overline{B}(v-v_*,\sigma)={B}(v-v_*,\sigma)+{B}(v-v_*,-\sigma)$, (cf. \cite{Glassey1996,Villani1998}). Moreover, we assume $B(v-v_*,\sigma)$ takes the product form as follows:
\begin{equation*}
	B(v-v_*,\sigma) = |v-v_*|^\gamma b(\cos\theta),
\end{equation*}
where $|v-v_*|^\gamma$ is called the kinetic part with $\gamma>-3$, and $b(\cos\theta)$ is called the angular part satisfying that there are $C_b>0$ and $0<s<1$ such that  
\begin{align*}
	\frac{1}{C_b\theta^{1+2s}}\le \sin\theta b(\cos\theta)\le \frac{C_b}{\theta^{1+2s}}, \quad\forall\,\theta\in (0,\frac{\pi}{2}].
\end{align*} 
%We further assume $\gamma+2s\ge 1$ and $1/2\le s<1$. 
For the angular non-cutoff Boltzmann operator, it is convenient to call it {\em hard potential} when $\gamma+2s\ge 0$ and {\em soft potential} when $-3<\gamma+2s<0$. We will assume in this work that 
\begin{equation*}
	\begin{aligned}
		&\gamma>\max\{-3,-2s-\frac{3}{2}\}\ \ \text{ for Boltzmann case,}\\
		&\gamma\ge 0, \ \frac{1}{2}\le s<1\ \ \text{  for Vlasov-Poisson-Boltzmann case}. 
	\end{aligned}
\end{equation*}
%Note that we consider both the hard potential case and the soft potential case. 

\smallskip 
We want to study the time-asymptotic stability of planar rarefaction wave to three dimensional Vlasov-Poisson-Boltzmann equation \eqref{1} in domain $\Omega$ with specular boundary condition. Thus, we denote 																	
\begin{align*}
%	\label{rhopm}
	M_{[\rho_\pm,u_\pm,\theta_\pm]}:=\frac{\rho_\pm}{(2\pi R\theta_\pm)^{\frac{3}{2}}}\exp\Big(-\frac{|v-u_\pm|^2}{2R\theta_\pm}\Big)
\end{align*}
 and assume the far field condition for the initial data on $x_1$-direction:
\begin{equation}\label{F01}\left\{
	\begin{aligned}
		F_{0,+}(x,v)\to M_{[\rho_\pm,u_\pm,\theta_\pm]} \quad \text{ as }x_1\to\pm\infty, \\
		F_{0,-}(x,v)\to M_{[\rho_\pm,u_\pm,\theta_\pm]} \quad \text{ as }x_1\to\pm\infty, 
	\end{aligned}\right.
\end{equation}
where $(\rho_\pm,u_\pm,\theta_\pm)$ are two constant states satisfying $\rho_\pm>0$, $u_\pm = (u_{1\pm},0,0)^t$, $\theta_\pm>0$ and $R>0$ is the gas constant.
To study solutions $F_\pm(t,x,v)$ to \eqref{1} and \eqref{F01}, we define the sum and difference of $F_+$ and $F_-$ as 
\begin{align*}
	F_1 = \frac{F_++F_-}{2},\ \text{ and }\ F_2 = \frac{F_+-F_-}{2}.
\end{align*}
Correspondingly, the VPB system \eqref{1} can be rewritten as the system of $F_1,F_2$:
\begin{equation}\label{F1}\left\{
	\begin{aligned}
		&\pa_tF_1 + v\cdot \na_xF_1 - \na_x\phi\cdot\na_vF_2 = 2Q(F_1,F_1),\\
		&\pa_tF_2 + v\cdot \na_xF_2 - \na_x\phi\cdot\na_vF_1 = 2Q(F_1,F_2),\\
		&-\Delta_x\phi = 2\int_{\R^3}F_2\,dv,\\ 
		& F_{1}(0,x,v)=F_{0,1}(x,v),\ F_{2}(0,x,v)=F_{0,2}(x,v),\ E(0,x)=E_0(x),
	\end{aligned}\right. 
\end{equation}
with initial data $F_{0,1}$ and $F_{0,2}$ satisfying 
\begin{equation}\label{F0}\left\{
	\begin{aligned}
		F_{0,1}(x,v)&\to M_{[\rho_\pm,u_\pm,\theta_\pm]} \quad \text{ as }x_1\to\pm\infty, \\
		F_{0,2}(x,v)&\to 0 \quad \text{ as }x_1\to\pm\infty.  
	\end{aligned}\right.
\end{equation}
%Taking inner product of the second equation of \eqref{F1} with $1$, we have 
%\begin{align}
%	\label{contin}
%	\pa_t\int_{\R^3}F_2\,dv + \na_x\cdot\int_{\R^3}vF_2\,dv= 0. 
%\end{align}
%We will consider linear operator $\L_2$ given by 
%\begin{align}
%	\label{LL2}
%	\L_2 f = -v\cdot\na_x + Q(\mu,f). 
%\end{align}
The Vlasov-Poisson-Boltzmann system \eqref{F1} describes the motion of plasma particles of two species (e.g. ions and electrons) under the effect of self-consistent electrostatic potential determined by the Poisson equation. In particular, when the effect of the electrical field is neglected and the two species condition reduces to one species, namely $\phi=F_2=0$, the VPB system reduces the Boltzmann equation. Thus, we also consider the Boltzmann equation as a simplified model:
\begin{equation}\label{FB}\left\{
	\begin{aligned}
		&\pa_tF + v\cdot \na_xF = 2Q(F,F),\\
		& F(0,x,v)=F_{0}(x,v),
	\end{aligned}\right. 
\end{equation}
with initial data $F_{0}(x,v)$ satisfying 
\begin{equation}\label{F0B}
		F_{0}(x,v)\to M_{[\rho_\pm,u_\pm,\theta_\pm]} \quad \text{ as }x_1\to\pm\infty.
\end{equation}							
Next we specify the spatial domains considered in this work. 
\subsubsection{Case of Torus}
In this case, we consider a simple model with periodic domain on transverse direction. That is, we let 
\begin{align*}
	\Omega = \R\times\T^2,
\end{align*}
where $\T^2 = [-\pi,\pi]^2$ is the periodic domain.

\subsubsection{Case of Rectangular Duct}
We will also consider the generalized rectangular duct as the following. Let 
\begin{align*}
%	\label{Omega}
	\Omega = \cup_{i=1}^N\R\times (a_{i,2},b_{i,2})\times (a_{i,3},b_{i,3}), 
\end{align*}
where $a_{i,j},b_{i,j}\in\R$ satisfy $a_{i,j}<b_{i,j}$ for $1\le i\le N$ and $j=2,3$. Then $\pa\Omega$ can be divided into two parts: $\Gamma_2$ and $\Gamma_3$, where $\Gamma_j$ $(j=2,3)$ is orthogonal to $x_j$-axis. Further, we don't distinguish $\Gamma_j$ and the interior of $\Gamma_j$ since the boundary of $\Gamma_j$ has {\it zero} spherical measure. Then the unit outward normal vector exists almost everywhere on $\Gamma_j$, which is $e_j$ or $-e_j$, with $e_j$ being the standard unit vector. Note that $\Omega$ could be non-convex in general. 
With $n=n(x)$ being the outward normal direction at $x\in\pa\Omega$, we decompose $\pa\Omega$ as 
\begin{align*}
	\gamma_- &= \{(x,v)\in\pa\Omega\times\R^3 : n(x)\cdot v<0\},\quad\text{(the incoming set),}\\
	\gamma_+ &= \{(x,v)\in\pa\Omega\times\R^3 : n(x)\cdot v>0\},\quad\text{(the outgoing set),}\\
	\gamma_0 &= \{(x,v)\in\pa\Omega\times\R^3 : n(x)\cdot v=0\},\quad\text{(the grazing set).}
\end{align*}
We will denote $(\tau_1,\tau_2)$ to be the other two standard unit vectors $\{e_i\}_{i\neq j}$ such that $(n,\tau_1,\tau_2)$ forms a unit normal basis in $\R^3$. 
For the generalized rectangular duct $\Omega$, we consider the {\it specular-reflection} boundary condition:
\begin{align}
	\label{specular}
	F_\pm(t,x,R_xv) = F_\pm(t,x,v)\ \ \text{ on } \gamma_-,\text{ where } R_xv = v - 2n(x)(n(x)\cdot v). 
\end{align}
For the electrostatic field, we assume the Neumann boundary condition
\begin{align}\label{Neumann}
	\pa_{n}\phi=0, \ \text{ on }x\in \pa\Omega.  
\end{align}

%The specular-reflection boundary condition reads
%\begin{equation}\label{specularB}
%	F(t,x,R_xv) = F(t,x,v)\ \ \text{ on } \gamma_-,\text{ where } R_xv = v - 2n(x)(n(x)\cdot v). 
%\end{equation}

%In particular, the Poisson equation for potential $\phi$ is a pure Neumann boundary problem and we require the zero-mean condition on $\phi$:
%\begin{align*}
%	\int_\Omega\phi(t,x)\,dx=0, \ \text{ for } t\ge 0
%\end{align*}
% to ensure the uniqueness of solutions. Also, we need zero-mean condition
%\begin{align}\label{19}
%	\int_\Omega\int_{\R^3}(F_+-F_-)\,dvdx=0
%\end{align}
%to ensure its existence, which follows from conservation laws. 
%Indeed, taking inner product of the second equation of \eqref{F1} with $1$ over $\Omega\times\R^3$, one has
%\begin{align*}
%	\pa_t\int_\Omega\int_{\R^3}(F_+-F_-)\,dvdx=0,\quad \text{ for } t\ge 0,
%\end{align*}
%where we apply Lemma \ref{Lemspecular} to find that $(v\cdot\na_xF_2,1)_{L^2_{x,v}}=0$. Assuming \eqref{19} is valid initially at $t=0$, we deduce that \eqref{19} holds for $t\ge 0$. 

\smallskip

%Guo2002,Duan2013,Li2017,Li2021,Duan2010  VPB

%For the general theory of Boltzmann equation, one may refer to \cite{} for long-time asymptotic behaviors near a global Maxwellian. 

We would like to discuss the nonlinear wave patterns related to the system of fluid dynamics, i.e., Euler equations and Navier-Stokes equations (cf. \cite{Huang2010,Huang2010a,Xin2010,Guo2021a}).
There are three basic wave patterns of hyperbolic conservation laws, that is, two nonlinear waves, shock and rarefaction waves, in the genuinely nonlinear characteristic fields and a linear wave, contact discontinuity, in the linearly degenerate field. These dilation invariant solutions, and their linear superposition in the increasing order of characteristic speed are called Riemann solutions.
According to the nature, a rarefaction wave tends to constant states locally in space as the time goes to infinity. Thus, it would expect that there's a corresponding Boltzmann wave that is time-asymptotically in local thermal equilibrium and tends to the fluid rarefaction wave (for Euler system).

\smallskip
%which is related to the systems fluid dynamics, i.e., Euler equations and Navier-Stokes equations (cf. \cite{Liu2004a,Yu2004,Liu2006a,Huang2010,Huang2010a,Yu2010,Xin2010,Guo2021a}).
%According to the expansive nature, a rarefaction wave tends to constant states locally in space as the time goes to infinity. Thus, it would expect that there's a corresponding Boltzmann wave that is time-asymptotically in local thermal equilibrium and tends to the fluid rarefaction wave (for Euler system).
%In fact, there are three basic wave patterns to the system of hyperbolic conservation laws, that is, two nonlinear waves, shock and rarefaction waves, in the genuinely nonlinear characteristic fields and a linear wave, contact discontinuity, in the linearly degenerate field. These dilation invariant solutions, and their linear superposition in the increasing order of characteristic speed are called Riemann solutions.

In the present paper, we are interested in the nonlinear wave patterns of the Boltzmann equation. In the 1980s, Caflisch and Nicolaenko \cite{Caflisch1982} constructed the shock profile solutions of the Boltzmann equation with the cutoff assumption for hard potentials. Later on, Liu and Yu \cite{Liu2004a} proved the stability of viscous shock wave by using the energy method with micro-macro decomposition with the shock profile given by \cite{Caflisch1982,Liu2013}. In \cite{Yu2010,Yu2004}, Yu established the large time behavior of shock profiles solutions of the Boltzmann equation for the hard sphere model. For the stability of rarefaction waves solution, we refer to \cite{Liu2006a} for Boltzmann equation, \cite{Duan2020a} for Vlasov-Poisson-Landau system and \cite{Duan2021c} for Landau equation with small Knudsen number. \cite{Li2017} discussed the nonlinear stability of spatial one-dimensional shock and rarefaction waves for the cutoff Vlasov-Poisson-Boltzmann system. The stability of contact wave for Boltzmann equation are discussed in \cite{Huang2006} with the zero mass condition and \cite{Huang2008} without the zero mass condition. For the contact wave of Vlasov-Poisson-Boltzmann system, we refer to \cite{Huang2016}. 

\smallskip
The above-mentioned works are all in the framework of Grad's angular cutoff assumption and spatially one-dimensional space. For the three-dimensional spatial space, Duan and Liu \cite{Duan2017b} established the stability of rarefaction wave for VPB system for a disparate mass binary mixture near the quasi-neutral Euler system. 
We also refer planar rarefaction wave to \cite{Wang2019} for the cutoff Boltzmann equation, \cite{Wang2019a} for the cutoff Vlasov-Poisson-Boltzmann system and \cite{Duan2021} for Vlasov-Poisson-Landau system. Moreover, we mention Duan and Liu \cite{Duan2021b}, which gives the global stability of the solution to Boltzmann equation for shear flow.

\smallskip
For the time-asymptotic behavior of rarefaction wave, we refer to \cite{Kawashima2008} for the Navier-Stokes equations in the half space. In \cite{Matsumura1986} and \cite{Matsumura1992}, the authors found the large time decay for smoothing approximated rarefaction waves for the rarefaction wave solution of Euler equations. The time asymptotic decay rate of rarefaction waves in zeroth order for Burgers' equations can be found in \cite{Ito1996}. 

\smallskip
In this work, we are going the consider the stability of rarefaction waves solution to non-cutoff Vlasov-Poisson-Boltzmann system. 
% with large-time decay rate to the rarefaction waves to the compressible Euler equations.
%We are able to obtain the large-time decay rate for fluid components of the Vlasov-Poisson-Boltzmann system for the first time. 
We will consider the rectangular duct with specular-reflection boundary condition as our spatial domain, which has an actual physical boundary. 
It's an interesting and difficult topic when the physical boundary conditions are involved in the theory of wave patterns for fluid dynamics. This should be a new boundary model for kinetic equations such as Boltzmann equation as far as we know, and we believe that similar boundary can be applied to other wave patterns and other kinetic equation. The rectangular duct considered in this work has a good reflection property on the boundary, which allows performing high-order derivatives energy estimates.  
	
	\smallskip
	For the boundary theory, we refer to \cite{Guo2009} as classic work for the Boltzmann equation in bounded domains. Later, \cite{Liu2016} generalize the result to soft potential and \cite{Cao2019} give the global strong solution for the Vlasov-Poisson-Boltzmann system. Recently, \cite{Guo2020,Deng2021e} generalized the result to Landau equation and Vlasov-Poisson-Boltzmann/Landau system for specular-reflection boundary condition. Moreover, we mention \cite{Duan2020} for the existence of the Boltzmann equation in a finite channel with inflow and specular-reflection boundary conditions. 

\subsection{Reformulation of the Problem}
%Corresponding to system \eqref{1}, we consider the following reformulation. Let $f=[f_+,f_-]$ be given by 
%\begin{align*}
%	F_\pm=\mu+f_\pm.
%\end{align*}
%where $\mu$ is given by 
%\begin{equation}\label{globalmu}
%	\mu = M_{[1,0,\frac{3}{2}]} = (2\pi)^{-\frac{3}{2}}\exp\{-\frac{|v|^2}{2}\},
%\end{equation}
%Then system \eqref{1} becomes
%\begin{equation}
%	\left\{
%	\begin{aligned}
%		\label{fpm}
%		&\pa_tf_\pm + v\cdot\nabla_xf_\pm\pm \nabla_x\phi\cdot v\mu \mp\nabla_x\phi\cdot\nabla_vf_\pm  \\&\qquad\qquad= Q(f_\pm+f_\mp,\mu)+ Q(2\mu+f_\pm+f_\mp,f_\pm),\\
%		&-\Delta_x\phi=\int_{\mathbb R^3}{(f_+-f_-)dv},\\
%		&f_\pm(0) = f_{0,\pm},\quad E(0) = E_0,
%	\end{aligned}\right.
%\end{equation}
%where $f_{0,\pm}=F_{0,\pm}-\mu$. 
%Note that $Q(\mu,\mu)=0$. 
%We also denote operator 
%\begin{equation*}
%	%\label{L1}
%	L_\mu f=2Q(\mu,f)+2Q(f,\mu),\quad L_{\mu,2} = 2Q(\mu,f). 
%\end{equation*}
%Together with the transport operator and Poisson term, we define the linear operator 
%\begin{align}
%	\label{LL}
%	\L_\mu f = -v\cdot\na_xf  + L_\mu f, \quad \L_{\mu,2} f = -v\cdot\na_x \mp\nabla_x\phi\cdot v\mu + L_{\mu,2} f. 
%\end{align}

%As far as we know, the non-cutoff case and boundary condition aren't analyzed until now. We first apply the macro-micro decomposition from \cite{Liu2004a} to rewrite equation \eqref{F1} and then use the high-order specular boundary technique from \cite{Deng2021} to deal with the boundary condition. 
%
%
%Next we present the micro-marco decomposition near local Maxwellian

We first consider the macro-micro decomposition as in \cite{Liu2004,Liu2004a}. 
Define collision invariant functions $\xi_i(v)$ by 
\begin{align*}
	\xi_0(v)=1,\quad \xi_i(v)=v_i\ (i=1,2,3),\quad \xi_4(v)=\frac{1}{2}|v|^2.
\end{align*} 
Then 
%according to \cite{Guo2002a}, 
one has 
\begin{align*}
	\int_{\R^3}\xi_i(v)Q(F_1,F_1)\,dv=0, \text{ for }0\le i\le 4.
\end{align*}
For any solution $F_1$ to equation \eqref{F1}, there are five macroscopic (fluid) quantities: the mass density $\rho(t,x)$, momentum $\rho(t,x)u(t,x)$ and energy density $e(t,x)+\frac{1}{2}|u(t,x)|^2$:
\begin{equation}\label{rhouth}\left\{
	\begin{aligned}
		&\rho(t,x) \equiv \int_{\R^3}F_1(t,x,v)\,dv,\\
		&\rho(t,x)u_i(t,x) \equiv \int_{\R^3}\xi_i(v)F_1(t,x,v)\,dv,\text{ for }i=1,2,3,\\
		&\rho(t,x)\big(e(t,x)+\frac{1}{2}|u(t,x)|^2\big) \equiv \int_{\R^3}\xi_4(v)F_1(t,x,v)\,dv.
	\end{aligned}\right. 
\end{equation}
We define the local Maxwellian associated with solution $F_1(t,x,v)$ to the equation \eqref{F1} in terms of macroscopic quantities by 
\begin{align*}
%	\label{M}
	M:= M_{[\rho,u,\theta]}(t,x,v) = \frac{\rho(t,x)}{(2\pi R \theta(t,x))^{\frac{3}{2}}}\exp\Big(-\frac{|v-u(t,x)|^2}{2R\theta(t,x)}\Big).
\end{align*}
Here $\th(t,x)$ is the temperature which is related to the internal energy $e$ by $e=\frac{3}{2}R\th$ and $u(t,x)=(u_1,u_2,u_3)(t,x)$ is the fluid velocity. Define the $L^2$ inner product over $\R^3_v$ by $L^2_v$. Then the macroscopic space $\ker L_M$ ($L_M$ is given in \eqref{LM}) is spanned by orthonormal basis
\begin{equation}\label{mathfrackN}\left\{
	\begin{aligned}
		&\chi_0(v) = \frac{1}{\sqrt{\rho}}M,\\
		&\chi_i(v) = \frac{v_i-u_i}{\sqrt{R\th\rho}}M,\text{ for }i=1,2,3,\\
		&\chi_4(v) = \frac{1}{\sqrt{6\rho}}\big(\frac{|v-u|^2}{R\th}-3\big)M,\\
		&(\chi_i,\frac{\chi_j}{M})_{L^2_v} =\delta_{ij},\text{ for }i,j=0,1,2,3,4.
	\end{aligned}\right. 
\end{equation}
The macroscopic projection $P_0$ from $L^2_v$ to $\ker L_M$ and the microscopic projection $P_1$ from $L^2_v$ to $(\ker L_M)^{\perp}$ are defined by 
\begin{align}\label{P0}
	P_0f = \sum_{i=0}^4(f,\frac{\chi_i}{M})_{L^2_v}\,\chi_i,\quad P_1f=f-P_0f.
\end{align}
A function $f$ is called microscopic if 
\begin{align}
	\label{micro}
	(f,\xi_i)_{L^2_v}=0, \ \forall\,0\le i\le 4. 
\end{align}
Note that $P_1f$ is microscopic. By using the famous collision invariant (cf. \cite{Ukai}), $Q(f,f)$ is also microscopic. 
Besides projection $P_0$ for local Maxwellian $M$, we further denote the projection $P_{\mu}$ and $P_{\mu,2}$ for global Maxwellian 
$\mu$
by letting 
%\begin{align}
%	\label{P2}
%	P_2f = \int_{\R^3}f(u)\,du\,\mu, 
%\end{align}
%and
%\begin{align}
%\label{P2}
%P_2f = \int_{\R^3}f(u)\,du\,\mu,  
%%+v\cdot\int_{\R^3}uf(u)\,du+(|v|^2-3)\int_{\R^3}\frac{|u|^2-3}{6}f(u)\,du\Big)\mu. 
%\end{align}
\begin{align*}
%	\label{P3}
	P_{\mu}f = \Big(\int_{\R^3}\mu^{\frac{1}{2}}f(u)\,du+v\cdot\int_{\R^3}u\mu^{\frac{1}{2}}f(u)\,du+(|v|^2-3)\int_{\R^3}\frac{|u|^2-3}{6}\mu^{\frac{1}{2}}f(u)\,du\Big)\mu^{\frac{1}{2}}, 
\end{align*}
and 
\begin{align}
	\label{Pmu2}
	P_{\mu,2}f = \int_{\R^3}\mu^{\frac{1}{2}}f(u)\,du\,\mu^{\frac{1}{2}}. 
\end{align}
%	,\quad b=\int_{\R^3}uF_2(u)\,du,\quad c=\int_{\R^3}\frac{|u|^2-3}{6}F_2(u)\,du. 
%Moreover, for $f=[f_+,f_-]$, we define the projections $\P=[\P_+,\P_-]$ with respect to operator $L$ given in \eqref{L1} as 
%\begin{equation}\label{P}
%	\PP f = \big(a_\pm+v\cdot b+(|v|^2-3)c\big)\mu
%%	 \Big(\int_{\R^3}f_\pm(u)du+v\cdot\int_{\R^3}\frac{u}{2}(f_++f_-)(u)du+(|v|^2-3)\int_{\R^3}\frac{|u|^2-3}{12}(f_++f_-)(u)du\Big)\mu,
%\end{equation}
%where function $a_\pm,b,c$ are given by 
%\begin{equation}
%	\begin{aligned}\label{abc}
%		a_\pm &= \int_{\R^3} f_\pm dv,\quad 
%		b_j= \frac{1}{2}\int_{\R^3} v_j (f_++f_-) dv ,\\
%		c&=\int_{\R^3} \frac {|v|^2-3}{12}(f_++f_-) dv. 
%	\end{aligned}
%\end{equation}
%Note that $L\P f=0$, i.e. the kernel of $L$ is the span of $\{[1,0]\mu,[0,1]\mu,[1,1]v\mu,[1,1]|v|^2\mu\}$. 
%%In order to apply the semigroup method with semigroup $e^{t\L}$, we define projection $\Pi=[\Pi_+,\Pi_-]$ on $\Omega\times\R^3$ as  
%%	\begin{multline}\label{Pi}
%%		\Pi_\pm f = \Big(\int_{\Omega}\int_{\R^3} f_\pm(u) dudx+v\cdot\int_{\Omega}\int_{\R^3} \frac{u}{2} (f_++f_-)(u) dudx\\
%%		+(|v|^2-3)\int_{\Omega}\int_{\R^3} \frac {|u|^2-3}{12}(f_++f_-)(u) dudx\Big)\mu. 
%%	\end{multline}
%%%and 
%%%\begin{align}
%%%	\label{Pi2}
%%%	\Pi_2 f= \int_{\Omega}\int_{\R^3} f_\pm(u) dudx\,\mu
%%%\end{align}
%%	Notice that $\Pi\{\I-\P\}f=0$ and hence, $\{\I-\P\}f=\{\I-\Pi\}\{\I-\P\}f$. Here $\I=[\I_+,\I_-]$ is the identity mapping. 
With projections $P_0$ and $P_1$, we can decompose the solution $F_1(t,x,v)$ to \eqref{F1} into macroscopic (fluid) part and microscopic (non-fluid) part; cf. \cite{Liu2004}: 
\begin{equation*}
%	\label{FMG}
	F_1(t,x,v) = M(t,x,v)+G(t,x,v),\quad P_0F_1=M,\quad P_1F_1=G.
\end{equation*}
Then the equation \eqref{F1}$_1$ becomes 
\begin{equation}\label{2}
	\pa_t(M+G)+v\cdot\na_x(M+G)-\na_x\phi\cdot\na_v F_2= 2Q(M+G,M+G).
\end{equation}
Taking the inner product of \eqref{2} with collision invariant $\xi_i(v)$ over $\R^3_v$, we have the fluid-type system:
\begin{equation}\label{fluid}\left\{
	\begin{aligned}
		&\pa_t\rho + \na_x\cdot(\rho u) = 0,\\
		&\pa_t(\rho u) + \na_x\cdot(\rho u\otimes u) +\na_x p +\na_x\phi\int_{\R^3}F_2\,dv = -\int_{\R^3}v\otimes v\cdot\na_x G\,dv,\\
		&\pa_t\big(\rho(e+\frac{|u|^2}{2})\big) + \na_x\cdot\big(\rho u(e+\frac{|u|^2}{2})+pu\big)+\na_x\phi\cdot\int_{\R^3}vF_2\,dv = -\frac{1}{2}\int_{\R^3}|v|^2 v\cdot\na_x G\,dv,
	\end{aligned}\right. 
\end{equation}
where $p=\frac{2}{3}\rho e=R\rho\th$ is the pressure and $e=\th$ for mono-atomic gas. We assume the gas constant $R=\frac{2}{3}$ for convenience. For the microscopic part, noticing $Q(M,M)=0$, we apply $P_1$ to \eqref{2} to obtain 
\begin{align}\label{G1}
	\pa_tG + P_1v\cdot\na_x(M+G) -P_1\na_x\phi\cdot\na_vF_2 = L_MG + 2Q(G,G),
\end{align}
where \begin{equation}
	\label{LM}L_MG=2Q(M,G)+2Q(G,M)
%	{\red check this 2?}
\end{equation} is the linearized collision operator and the null space $\ker L_M$ of $L_M$ is spanned by $\{\chi_j\}_{j=0}^4$. 
By direct calculation, we have 
\begin{align}\label{P1M}
	P_1v\cdot\na_xM &= P_1v\cdot \Big(\frac{|v-u|^2\na_x\th}{2R\th^2}+\frac{(v-u)\cdot\na_x u}{R\th}\Big)M. 
\end{align}
Since $L_M$ is injective on $(\ker L_M)^{\perp}$, we can consider its inverse on $(\ker L_M)^\perp$ and \eqref{G1} yields 
\begin{align}\label{G}
	G 
%	&= L_M^{-1}(P_1v\cdot\na_xM ) + L_M^{-1} (\pa_tG +P_1v\cdot\na_xG-P_1\na_x\phi\cdot\na_vF_2- Q(G,G))\\
	&= L_M^{-1}\Big(P_1v\cdot \Big(\frac{|v-u|^2\na_x\th}{2R\th^2}+\frac{(v-u)\cdot\na_x u}{R\th}\Big)M\Big) + L_M^{-1}\Theta,
\end{align}
where  
\begin{align}\label{Theta2}
	\Theta = \pa_tG +P_1v\cdot\na_xG-P_1\na_x\phi\cdot\na_vF_2- 2Q(G,G).
\end{align}
We will use the microscopic perturbation function $\overline{G}(t,x,v)$ defined by 
\begin{align}\label{olG}
	\overline{G}(t,x,v) = L_M^{-1}P_1v_1\Big\{	\frac{|v-u|^2\bar\th_{x_1}}{2R\th^2}+\frac{(v-u)\cdot \bar{u}_{x_1}}{R\th}\Big\}M.
\end{align}
Then we denote the perturbation around $(\bar\rho,\bar u,\bar\theta)$ and $\ol G$ by 
\begin{equation}\label{g}
	\left\{\begin{aligned}
		&\wt\rho(t,x) = \rho(t,x)-\bar\rho(t,x),\\
		&\wt u(t,x) = u(t,x)-\bar u(t,x),\\
		&\wt\th(t,x) = \th(t,x)-\bar\th(t,x),\\
		&\g(t,x) = \mu^{-\frac{1}{2}}(G(t,x)-\overline G(t,x)),\\
		&\f = \mu^{-\frac{1}{2}}F_2(t,x).
%		&\g_\pm = \g\pm F_2. 
	\end{aligned}\right.
\end{equation}
%We further denote 
%\begin{equation}\label{f}
%	\left\{\begin{aligned}??
%		&F_2 = \f_0\frac{M}{\rho} + \f_1,\ \text{ with }\f_0=(F_2,M)_{L^2_v}\frac{M}{\rho},\ \f_1=F_2-\f_0\frac{M}{\rho},\\
%		&F_\pm = F_1\pm F_2 = M+\ol G+\sqrt\mu\g_\pm,\\
%		&
%	\end{aligned}\right.
%\end{equation}
%Note that the definition of $\g$ is a scalar function, not as a vector of $\g=[\g_+,\g_-]$. 
%{\red need to calculation $\pa_t\|F_2\|$??}
%where the linearized operator $L_{M,2}$ is given by 
%\begin{align}\label{LM2}
%	L_{M,2}f=2Q(M,f).
%\end{align}
%The null space $\ker L_{M,2}$ is spanned by the single macroscopic variable $\{\chi_0\}$. Then 
%Similar to \eqref{P0}, we can define the macroscopic projection $P_{M,2}$ from $L^2_v$ to $\ker L_{M,2}$ by 
%\begin{align}
%	P_{M,2}f = (f,\frac{\chi_0}{M})_{L^2_v}\,\chi_0. 
%\end{align}
%Also, we denote the linearized operator $L_{\mu,2}$ and projection $P_{\mu,2}$ corresponding to $\mu$ by 
%\begin{align}
%	\label{Lmu22}
%	L_{\mu,2}f=2Q(\mu,f),
%\end{align}
%and 
%\begin{align}
%	P_{\mu,2}f = (f,\mu)_{L^2_v}\,\mu. 
%\end{align}
Corresponding to projection \eqref{Pmu2}, we denote the macroscopic components of $\f$ by 
\begin{align*}
	a=\int_{\R^3}\sqrt{\mu(u)}\f(u)\,du. 
\end{align*}
We also denote linearized operator $\L$, $\L_2$ and nonlinear collision term $\Gamma(f,g)$ respectively by 
\begin{align*}
%	\label{L}
	\L f &= 2\mu^{-\frac{1}{2}}Q(\mu,\mu^{1/2}f)+ 2\mu^{-\frac{1}{2}}Q(\mu^{1/2}f,\mu),
	\\	\L_2 f &= 2\mu^{-\frac{1}{2}}Q(\mu,\mu^{1/2}f),
\end{align*}
and 
\begin{equation*}
%	\label{GammaDef}
	\Gamma(f,g) = 2\mu^{-\frac{1}{2}}Q(\mu^{1/2}f,\mu^{1/2}g). 
\end{equation*}
For $F_2$, we apply $F_1=M+G$ and $F_2=\sqrt\mu\f$ in \eqref{F1}$_2$ to obtain that 
\begin{align}
\label{F2}
\pa_t\f+ v\cdot\na_x\f-\frac{\na_x\phi\cdot\na_vF_1}{\sqrt\mu} =\L_2\f+\Gamma\big(\frac{M-\mu}{\sqrt\mu},\f\big) + \Gamma\big(\frac{G}{\sqrt\mu},\f\big), 
\end{align}
For the microscopic parts $\g$, we have from \eqref{G1} that 
\begin{multline*}
%	\label{g1}
	\pa_tG + P_1v\cdot\na_xG-P_1\na_x\phi\cdot\na_vF_2=2Q(\g,M)+2Q(M,\g)+2Q(G,G)\\
	-P_1v\cdot\Big\{\frac{|v-u|^2\na_x\wt\th}{2R\th^2}+\frac{(v-u)\cdot \na_x\wt{u}}{R\th}\Big\}M,
\end{multline*}
and hence, 
\begin{multline}
	\label{g2}
	\pa_t \g+v\cdot\na_x\g-\frac{P_1\na_x\phi\cdot\na_vF_2}{\sqrt{\mu}} =\L\g + \Gamma\big(\frac{M-\mu}{\sqrt\mu},\g\big)+ \Gamma\big(\g,\frac{M-\mu}{\sqrt\mu}\big) + \Gamma\big(\frac{G}{\sqrt\mu},\frac{G}{\sqrt\mu}\big) \\+ \frac{P_0(v\sqrt\mu\cdot\na_x\g)}{\sqrt\mu} - \frac{P_1(v\cdot\na_x\ol G)}{\sqrt\mu} - \frac{\pa_t\ol G}{\sqrt\mu} - \frac{1}{\sqrt\mu}P_1v\cdot\Big\{\frac{|v-u|^2\na_x\wt\th}{2R\th^2}+\frac{(v-u)\cdot\na_x\wt u}{R\th}\Big\}M.
\end{multline}
%Taking the addition and difference of \eqref{g2} and of \eqref{F2}, we have 
%\begin{multline}
%	\label{g3}
%	\pa_t\g_\pm+v\cdot\na_x\g_\pm\mp \na_x\phi\cdot\na_vF_\pm=2Q(F_1,\g_\pm)+2Q(\g,M)+2Q(G,\ol G)
%	-P_0\na_x\phi\cdot\na_vF_2\\+P_0v\cdot\na_x\g-P_1v\cdot\na_x\ol G-\pa_t\ol G
%	-P_1v\cdot \Big\{\frac{|v-u|^2\na_x\wt\th}{2R\th^2}+\frac{(v-u)\cdot \na_x\wt{u}}{R\th}\Big\}M. 
%\end{multline}
%Moreover, taking inner product of the second equation of \eqref{F1} with $1$, we have 
%\begin{align}
%	\label{F11}
%	\pa_t\int_{\R^3}F_2\,dv + \na_x\cdot\int_{\R^3}vF_2\,dv =0. 
%\end{align}

%???decomposition?
%For $F_\pm$, we can rewrite \eqref{1} as 
%\begin{align}\label{Fpm}
%	\pa_tF_\pm + v\cdot\na_xF_\pm \mp\na_x\phi\cdot\na_vF_\pm = Q(F_\pm+F_\mp,F_\pm). 
%\end{align}
%{\red define $L_\mu$}

Next we give the definition for nonlinear asymptotic rarefaction wave profile to the equation \eqref{1} as in \cite{Matsumura1986,Liu2006a}. For planar rarefaction waves, given two end states $(\rho_\pm,u_\pm,\th_\pm)$,  we consider the one dimensional compressible Euler system:
\begin{equation}\label{Euler}\left\{
	\begin{aligned}
		&\pa_t\rho + \pa_{x_1}(\rho u) = 0,\\
		&\pa_t(\rho u_1) + \pa_{x_1}(\rho u_1^2+ p) = 0,\\
		&\pa_t(\rho u_i) + \pa_{x_1}(\rho u_1u_i) = 0,\quad \text{for }i=2,3,\\
		&\pa_t\big(\rho(e+\frac{|u|^2}{2})\big) + \pa_{x_1}\big(\rho u(e+\frac{|u|^2}{2})+pu\big) = 0,
	\end{aligned}\right. 
\end{equation}
with the Riemann initial data 
\begin{equation}\label{Euler0}
	(\rho,u,\th)|_{t=0}=	(\rho^r_0,u^r_0,\th^r_0)(x) = \left\{\begin{aligned}
		(\rho_-,u_-,\th_-),\quad x_1<0,\\
		(\rho_+,u_+,\th_+),\quad x_1>0.
	\end{aligned}\right.
\end{equation}
Combining Euler system \eqref{Euler} and \eqref{Euler0} with the state equation: 
\begin{align*}
	p = \frac{2}{3}\rho \th = k\rho^{\frac{5}{3}}\exp(S),
\end{align*}
where $k=\frac{1}{2\pi e}$ is a constant and $S$ is the macroscopic energy given in \eqref{S1}, the Euler system \eqref{Euler} for $(\rho,u_1,S)$ has three distinct eigenvalues 
\begin{align*}
	\lambda_i(\rho,u_1,S) = u_1 + (-1)^{\frac{i+1}{2}}\sqrt{p_\rho(\rho,S)},\ i=1,3,\quad\lambda_2(\rho,u_1,S) = u_1,
\end{align*}
with corresponding right eigenvectors 
\begin{align*}
	r_i(\rho,u_1,S) = \big((-1)^{\frac{i+1}{2}}\rho,\sqrt{p_\rho(\rho,S)},0\big)^t,\ i=1,3,\quad r_2(\rho,u_1,S) = (p_S,0,-p_\rho)^t,
\end{align*}
such that 
\begin{align*}
	r_i(\rho,u_1,S)\cdot\na_{(\rho,u_1,S)}\lambda_i(\rho,u_1,S)\neq 0,\ i=1,3,\ \text{ and }\	r_2(\rho,u_1,S)\cdot\na_{(\rho,u_1,S)}\lambda_2(\rho,u_1,S)= 0.
\end{align*}
Thus the two $i$-Riemann invariants $\mathcal{R}^{(j)}_i$ $(i=1,3,j=1,2)$ can be defined by (cf. \cite{RichardCourant1999,Smoller2012}) 
\begin{align}\label{invariant}
	\mathcal{R}^{(1)}_i = u_1+(-1)^{\frac{i+1}{2}}\int^\rho\frac{\sqrt{p_z(z,S)}}{z}\,dz,\quad \mathcal{R}^{(2)}_i = S,
\end{align}
such that 
\begin{align*}
	r_i(\rho,u_1,S)\cdot \na_{(\rho,u_1,S)}\mathcal{R}^{(j)}_i(\rho,u_1,S) = 0, \ \ \forall\,i=1,3,\,j=1,2.
\end{align*}
Given the right state $(\rho_+,u_{1+},\th_+)$ with $\rho_+>0$ and $\th_+>0$, the $i$-rarefaction wave curve $(i=1,3)$ in the phase space $(\rho,u_1,\th)$ with $\rho>0$ and $\th>0$ can be defined by (cf. \cite{Lax1957})
\begin{align*}
	R_i(\rho_+,u_{1+},\th_+) := \big\{(\rho,u_1,\th):\pa_{x_1}\lambda_{i}>0,\ \mathcal{R}^{(j)}_i(\rho,u_1,S)=\mathcal{R}^{(j)}_i(\rho_+,u_{1+},\th_+),\ j=1,2\big\}.
\end{align*}
Without loss of generality, we only consider the $3$-rarefaction wave to the Euler system \eqref{Euler} with \eqref{Euler0} in this paper, and the stability of $1$-rarefaction wave can be considered similarly. The $3$-rarefaction wave to the Euler system \eqref{Euler} and \eqref{Euler0} can be expressed explicitly by the Riemann solution to the inviscid Burgers equation; cf. \cite{Liu1977,Liu1985}:
\begin{equation}\label{Burgers}
	\left\{\begin{aligned}
		&w_t + ww_{x_1} = 0,\\
		&w(0,x_1) = \left\{\begin{aligned}
			&w_-,\quad x_1<0,\\
			&w_+,\quad x_1>0.
		\end{aligned}\right.
	\end{aligned}\right.
\end{equation}
If $w_-<w_+$, then the Riemann problem \eqref{Burgers} admits a rarefaction wave solution 
$w^r(t,x_1)=w^r(\frac{x_1}{t})$ given by 
\begin{equation*}
	w^r(\frac{x_1}{t}) = \left\{\begin{aligned}
		&w_-,\quad \frac{{x_1}}{t}\le w_-,\\
		&\frac{{x_1}}{t},\quad w_-\le  \frac{{x_1}}{t}\le w_+,\\
		&w_+,\quad  \frac{{x_1}}{t}\ge w_+.
	\end{aligned}\right.
\end{equation*}
The $3$-rarefaction wave $(\rho^r,u^r,\th^r)$ to the Euler system \eqref{Euler} with \eqref{Euler0} can be defined explicitly by 
\begin{equation}\label{rarer}
	\left\{\begin{aligned}
		&w_\pm=\lam_3(\rho_\pm,u_\pm,\th_\pm),\quad w^r(\frac{x_1}{t})=\lam_3(\rho^r,u^r_1,\th^r)(\frac{{x_1}}{t}),\\
		&\mathcal{R}^{(j)}_3(\rho^r,u^r_1,\th^r)(\frac{{x_1}}{t}) = \mathcal{R}^{(j)}_3(\rho_\pm,u_\pm,\th_\pm),\ j=1,2,\ u^r_2=u^r_3=0,
	\end{aligned}\right.
\end{equation}
where $\mathcal{R}^{(j)}_3$ $(j=1,2)$ are the $3$-rarefaction invariants defined in \eqref{invariant}.
This solution is not smooth enough for our analysis and hence, we will construct an approximate smoothing rarefaction wave to the $3$-rarefaction wave given in \eqref{rarer}. Motivated by \cite{Matsumura1986,Matsumura1992}, the approximate rarefaction wave can be constructed by the Burgers equation 
\begin{equation}\label{Burger2}
	\left\{\begin{aligned}
		&\bar{w}_t+\bar{w}\bar{w}_{x_1}=0,\\
		&\bar{w}(0,{x_1}) = \bar{w}_0({x_1}) = \frac{w_++w_-}{2}+\frac{w_+-w_-}{2}k_0\int^{x_1}_0(1+y^2)^{-1}\,dy,
	\end{aligned}\right.
\end{equation}
where 
%$\ve>0$ is a small parameter to be determined and 
$k_0$ is a positive constant such that $k_0\int^{\infty}_0(1+y^2)^{-1}\,dy=1$. 
% for each $q\ge 2$. 
The solution $\bar{w}(t,x)$ of problem \eqref{Burger2} can be given explicitly by 
\begin{align}\label{barw}
	\bar{w}(t,{x_1}) = \bar{w}_0(x_0(t,{x_1})),
\end{align}
with $x_0(t,{x_1})$ given by relation 
\begin{align*}
	x_1= x_0(t,{x_1}) + \bar{w}_0(x_0(t,{x_1}))t.
\end{align*}
Correspondingly, the approximate rarefaction wave $(\bar{\rho},\bar{u},\bar{\th})$ for the $3$-rarefaction wave $(\rho^r,u^r,\th^r)$ to Euler system \eqref{Euler} and \eqref{Euler0} can be defined by 
\begin{equation}\label{rarea}
	\left\{\begin{aligned}
		&w_\pm=\lam_3(\rho_\pm,u_\pm,\th_\pm),\quad \bar w(t+1,x_1)=\lam_3(\bar{\rho},\bar{u}_1,\bar{\th})(t,{x_1}),\\
		&\mathcal{R}^{(j)}_3(\bar{\rho},\bar{u}_1,\bar{\th})(t,{x_1}) = \mathcal{R}^{(j)}_3(\rho_\pm,u_\pm,\th_\pm),\ j=1,2,\ \bar{u}_2=\bar{u}_3=0,
	\end{aligned}\right.
\end{equation}
where $\bar{w}(t,{x_1})$ is the solution to \eqref{Burger2} given by \eqref{barw}. Then the approximate planar $3$-rarefaction wave $(\bar{\rho},\bar{u},\bar{\th})(t,{x_1})$ satisfies Euler system 
\begin{equation}\label{Euler2}\left\{
	\begin{aligned}
		&\bar\rho_t + (\bar\rho \bar u_1)_{x_1} = 0,\\
		&(\bar \rho \bar u_1)_t + (\bar\rho\bar u_1^2+\bar p)_{x_1} = 0,\\
		&(\bar\rho\bar u_i)_t + (\bar\rho\bar u_1\bar u_i)_{x_1} = 0,\quad \text{for }i=2,3,\\
		&\big(\bar\rho\bar\th\big)_t + \big(\bar\rho\bar u_1\bar\th \big)_{x_1}+\bar p\bar u_{1x_1} = 0,
	\end{aligned}\right. 
\end{equation}
where $\bar{p}=R\bar\rho\bar\th$. The properties of the approximate rarefaction wave will be given in Lemma \ref{Lem21}.

Therefore, in this work, we will study the stability of the 1D local Maxwellian 
\begin{align}\label{localM}
	\overline{M} = M_{[\bar\rho,\bar u,\bar\th]}(v) = \frac{\bar\rho(t,{x_1})}{(2\pi R \bar\theta(t,{x_1}))^{\frac{3}{2}}}\exp\Big(-\frac{|v-\bar u(t,{x_1})|^2}{2R\bar\theta(t,{x_1})}\Big). 
\end{align}
%where $(\bar\rho,\bar u,\bar\th)$ is given in \eqref{rarea}.
 Then we deduce the stability of local Maxwellian $M_{[\rho^r,u^r,\th^r]}$ by using approximation between $M_{[\rho^r,u^r,\th^r]}$ and $M_{[\bar\rho,\bar u,\bar\th]}$.

%\smallskip

\subsection{Notations and Main Results}
Now we present some notations throughout the paper. 
Let $\<v\>=\sqrt{1+|v|^2}$ and $I$ be the identity mapping. % and $\1_{S}$ be the indicator function on a set $S$. 
%	$(\cdot|\cdot)$ denotes the inner product in $\C$.
The notation $a\approx b$ (resp. $a\gtrsim b$, $a\lesssim b$) for positive real function $a$, $b$ means there exists $C>0$ not depending on possible free parameters such that $C^{-1}a\le b\le Ca$ (resp. $a\ge C^{-1}b$, $a\le Cb$) on their domains.
If each component of $\al'$ is not greater than that of $\al$'s, we denote by $\al'\le\al$.  $\al'<\al$ means $\al'\le\al$ and $\al'\neq \al$. We will write $C>0$ (large) and $\lambda>0$ (small) to be a generic constant, which may be different in different lines.
Let
$$\pa^\al_\beta =  \pa^{\al_0}_t\pa^{\al_1}_{x_1}\pa^{\al_2}_{x_2}\pa^{\al_3}_{x_3}\pa^{\beta_1}_{v_1}\pa^{\beta_2}_{v_2}\pa^{\beta_3}_{v_3},$$
where $\al=(\al_0,\al_1,\al_2,\al_3)$ and $\beta=(\beta_1,\beta_2,\beta_3)$ are multi-indices. 
%Note that we don't consider time derivative estimate in this paper. 
%Note that derivative $\pa^\al$ contains time derivative. For linear operators $L_M$ $L_{M,2}$ $L_{\mu,2}$??$\L_\mu$,  and  given in \eqref{LL} and \eqref{LM}, we define the semigroup generated by it as 
%%\begin{align*}
%	$e^{t\L_\mu}$ and $e^{tL_M}$ respectively. Correspondingly, we define inner products and norms:
%	\begin{align}\label{DLM}
%		(f, g)_{L_M} &:= \int_{0}^{\infty}(P_1e^{\tau L_M}f,P_1e^{\tau L_M}g)_{L^2_xL^2_{10}}\,d\tau,\quad
%		\|f\|^2_{L_M} = (f,f)_{L_M}, 
%	\end{align}
%	and
%	\begin{align}\label{DLmu}
%		(f, g)_{\L_\mu} &:= \int_{0}^{\infty}(e^{\tau\L_\mu}f,e^{\tau\L_\mu}g)_{L^2_xL^2_{10}}\,d\tau,\quad 
%		\|f\|^2_{\L_\mu} = (f,f)_{\L_\mu}, 
%	\end{align}
%%\end{align*}
%For any $k,l,n\in\R$, 
For any $k\in\R$, denote the Sobolev space $L^2_k$ and $L^2_xL^2_k$, respectively, as 
\begin{align*}
	|f|^2_{L^2_k} = \int_{\R^3}|\<v\>^kf|^2\,dv,\quad \|f\|_{L^2_xL^2_k}^2 = \int_{\Omega}|\<v\>^kf|^2_{L^2_v}\,dx.
\end{align*}
To describe the Boltzmann collision operator, as in \cite{Gressman2011}, we denote
\begin{equation*}
	|f|^2_{L^2_D}:=|\<v\>^{\frac{\gamma+2s}{2}}f|^2_{L^2_v}+ \int_{\R^3}dv\,\<v\>^{\gamma+2s+1}\int_{\R^3}dv'\,\frac{(f'-f)^2}{d(v,v')^{3+2s}}\1_{d(v,v')\le 1},
\end{equation*}
and 
\begin{equation*}
%	|f|^2_{L^2_{D,w}}=|wf|^2_{L^2_D},\quad 
\|f\|^2_{L^2_xL^2_D} := \int_{\Omega}|f|^2_{L^2_D}dx.
\end{equation*}
The fractional differentiation effects are measured using the anisotropic metric on the {\it lifted} paraboloid
$d(v,v'):=\{|v-v'|^2+\frac{1}{4}(|v|^2-|v'|^2)^2\}^{1/2}$.
%For any weight function $w$, we write 
%\begin{equation*}
%		|f|^2_{L^2_{D,w}}=|wf|^2_{L^2_D},\quad 
%	\|f\|^2_{L^2_xL^2_{D,w}} := \int_{\Omega}|f|^2_{L^2_{D,w}}dx.
%\end{equation*}
Then by \cite[Eq. (2.15)]{Gressman2011}, we have 
\begin{align}\label{123}
	|\<v\>^{\frac{\gamma}{2}}\<D_v\>^sf|^2_{L^2_v}+|\<v\>^{\frac{\gamma+2s}{2}}f|^2_{L^2_v}\lesssim |f|^2_{L^2_D}\lesssim |\<v\>^{\frac{\gamma+2s}{2}}\<D_v\>^sf|^2_{L^2_v}
\end{align}

We consider the weight function $w(\al)$ as 
%\begin{equation}\label{w}
%	w(\al) =\left\{
%	\begin{aligned}
%		&\<v\>^{??},\text{ for Boltzmann equation,}\\
%		&\<v\>^{??},\text{ for Boltzmann equation,}
%	\end{aligned}
%\right.
%\end{equation}
	\begin{align}\label{w}
%		e^{\frac{q\<v\>}{(1+t)^\vt}},\qquad
	w(\alpha)=
		&\<v\>^{k-|\alpha|+2}, 
%		e^{\frac{q\<v\>}{(1+t)^\vt}}, 
\end{align}
respectively.  
%where $q\ge 0$
%, $0<\vt<\frac{1}{4}$. 
%where 
%\begin{equation*}
%	\begin{aligned}&p = 4s,\quad q=4s, \ \text{ for hard potential,}\\
%			&p = 4s-\gamma-\frac{2\gamma(1-s)}{s}+1,\quad q=4s-\frac{2\gamma}{s}+1, \ \text{ for soft potential,}
%	\end{aligned}
%\end{equation*}
We will also write $w(|\al|)=w(\al)$ for brevity. 
In order to prove the stability of the local Maxwellian $M_{[\bar\rho,\bar u,\bar\th]}$ defined in \eqref{localM}, we define the ``instant energy functional" $\E_k(t)$ and ``dissipation rate" $\D_k(t)$ respectively by 
\begin{align}\label{E}
	\E_k(t) &\approx \sum_{|\al|\le 3}\Big(\|\pa^\al(\wt\rho,\wt u,\wt\th)\|_{L^2_x}^2 + \|w(\al)\pa^\al\g\|^2_{L^2_{x,v}}+ \|w(\al)\pa^\al\f\|^2_{L^2_{x,v}} +\|\pa^\al\na_x\phi\|_{L^2_x}^2\Big). 
\end{align}
and 
\begin{align}
	\label{D}
	\D_k(t) &\notag= \sum_{1\le|\al|\le 3}\|\pa^\al(\wt\rho,\wt u,\wt\th)(t)\|^2_{L^2_x} + \sum_{|\al|\le 3}\|w(\al)\pa^\al \g\|^2_{L^2_xL^2_D}\\&\qquad+ \sum_{|\al|\le 3}\|w(\al)\pa^\al\f\|^2_{L^2_{x}L^2_D} +\sum_{|\al|\le 3}\|\pa^\al\na_x\phi\|_{L^2_x}^2,
\end{align}
where $(\wt{\rho},\wt u,\wt\th)$ and $\g$, $\f$ are given by \eqref{g}. 
Here index $k$ indicates the power on weight $w$ in \eqref{w}. 
%Moreover, we denote the high-order ``instant energy functional" $\E^h_k(t)$ and ``dissipation rate" $\D^h_k(t)$ respectively by 
%\begin{align}\label{Eh}
%	\E^h_k(t) &\approx \sum_{1\le|\al|\le 2}\Big(\|\pa^\al(\wt\rho,\wt u,\wt\th)\|_{L^2_x}^2 +\|w(\al)\pa^\al\g\|^2_{L^2_{x,v}}+\|w(\al)\pa^\al\f\|^2_{L^2_{x,v}} +\|\pa^\al\na_x\phi\|_{L^2_x}^2\Big). 
%\end{align}
%and 
%\begin{align}
%	\label{Dh}
%	\D^h_k(t) &= \sum_{|\al|= 2}\|\pa^\al(\wt\rho,\wt u,\wt\th)(t)\|^2_{L^2_x} + \sum_{1\le|\al|\le {2}}\Big(\|w(\al)\pa^\al \g\|^2_{L^2_xL^2_D}+ \|w(\al)\pa^\al\f\|^2_{L^2_{x}L^2_D}+\|\pa^\al\na_x\phi\|_{L^2_x}^2\Big). 
%\end{align}
%For fluid components, we emphasize that $\D_k(t)$ contains only high-order derivatives terms of $\na_x(\wt u,\wt\th)$. The term $\wt\rho$ and pure time derivative on $(\wt u,\wt\th)$ is not included in $\D_k(t)$. 
Letting $F_2=\phi=0$, the system \eqref{F1} and \eqref{F0} reduce to pure Boltzmann equation \eqref{FB} and \eqref{F0B}. In this case, we also use the notations for energy functionals $\E_k(t)$ and $\D_k(t)$ with $F_2=\phi=0$ for brevity.

\smallskip 
In our analysis, we will assume the {\it a priori} assumption that 
%Moreover, throughout the paper, we will use the {\it a priori} assumption 
\begin{align}\label{priori2}
	\sup_{0\le t\le T}{\E_k}(t)\le \ve, 
\end{align}
for any $T>0$ and small $\ve>0$. 
For simplicity, we choose $(\rho_\pm,u_\pm,\th_\pm)$ in \eqref{F0} to be close enough to the state $(1,0,\frac{3}{2})$ and, as in \cite{Liu2006a}, we will use a global Maxwellian $\mu$. 
Using the approximate rarefaction wave (Lemma \ref{Lem21}), there exists a constant $\eta_0=\eta_0(\ve)>0$ small enough such that $\lim_{\ve\to0}\eta_0(\ve)=0$ and for all $(t,x)$, 
\begin{equation}\label{eta}
	|\rho(t,x)-1|+|u(t,x)|+|\th(t,x)-\frac{3}{2}|<\eta_0,\quad \frac{3}{2}<\th(t,x)<3.
\end{equation}
Also, we denote the wave strength 
\begin{align}
	\label{delta}\delta=|(\rho_+-\rho_-,u_+-u_-,\th_+-\th_-)|. 
\end{align}
%Moreover, we will use the weight function \begin{equation}
%	e^{\frac{A_{\al,\beta}\phi}{\<v\>^2}}
%\end{equation} in our analysis. Here $A_{\al,\beta}$ is a constant satisfying 
%\begin{align}\label{Aalbeta}
%	\na_v(w^2(\al))=\frac{A_{\al,\beta}v}{\<v\>^2}w^2(\al). 
%\end{align}

With the above preparation, we are ready to state our main results. 
\begin{Thm}\label{Main1}
%	Let $K=2$ be the total order of derivatives. 
	Let $k\ge \max\{0,\gamma+2s\}$. Consider Boltzmann equation \eqref{FB} for $\gamma>\{-3,-2s-\frac{3}{2}\}$ and Vlasov-Poisson-Boltzmann system \eqref{F1} for $\gamma\ge 0$, $\frac{1}{2}\le s<1$. 
	Assume that \eqref{eta} holds for some small $\eta_0>0$ and the Riemann solution of Euler system \eqref{Euler} with \eqref{Euler0} consists of a $3$-rarefaction wave $(\rho^r,u^r,\th^r)$ given by \eqref{rarer}. There exists $\delta_0,\ve_0>0$ small enough such that if the wave strength $\delta\le \delta_0$, the initial data $F_0$ satisfies $F_{0,\pm}(x,v)\ge 0$ and 
	\begin{align}
		\label{small2}
		\E_k(0)\le \ve_0,
	\end{align}
%	for some $A>0$, 
	then the Boltzmann equation \eqref{FB} or VPB system \eqref{F1} in $\Omega$ (with boundary condition \eqref{specular} and \eqref{Neumann} for the case of rectangular duct) admits a unique global solution $F_\pm(t,x,v)\ge 0$ satisfying 
	\begin{align}\label{es1}
		\sup_{0\le t\le T}\overline\E_k(t) + \int^T_0\D_k(t)\,dt \lesssim \ve_0 + \delta^\frac{1}{6},
	\end{align}
	for $T>0$, where $\overline{\E_k}$ is given by \eqref{olE} satisfying 
	\begin{align*}
		\E_k(t) -\delta^{\frac{2}{3}} \lesssim \ol\E_k(t) \lesssim \E_k(t) +\delta^{\frac{2}{3}}.
	\end{align*} 
Here $\E_k(t)$ and $\D_k(t)$ are given by \eqref{E} and \eqref{D} respectively. 
	Consequently, 
	\begin{align}\label{es2}
		\sum_{|\al|\le 3}\|w(\al)\mu^{-1/2}\pa^\al(F_\pm(t,x,v)-M_{[\rho^r,u^r,\th^r]})\|_{L^2_xL^2_v}\lesssim \ve_0 + \delta^\frac{1}{6}.
	\end{align}
	Moreover, assuming the initial data satisfies 
	\begin{align*}
%		\label{small222}
		\E_k(0)+\E_{k-\gamma/2}(0) \le A, 
	\end{align*} 
	for some $A>0$, we have the time-asymptotic behavior 
	\begin{align}\label{es3}
		\lim_{t\to\infty}\|\<v\>^{k}\mu^{-1/2}\big(F_\pm-M_{[\rho^r,u^r,\th^r]}(\frac{x_1}{1+t})\big)\|^2_{L^\infty_xL^2_v}=0.
%		\lesssim 
%		(1+A)(1+t)^{-1}.
	\end{align}
	
\end{Thm}
%\begin{Rem}
%Theorem \ref{Main1} gives the time-asymptotic stability of rarefaction wave
%% with optimal time decay 
% for non-cutoff three-dimensional VPB system with or without specular reflection boundary conditions in generalized rectangular duct. 
%% The time decay rate is optimal in the sense that \eqref{es3} has the same rate as the approximation between rarefaction wave $(\rho^r,u^r,\th^r)$ and its smoothing approximated version $(\bar\rho,\bar u,\bar\th)$, which follows from \cite{Ito1996}. This is the first result of the time decay rate of macroscopic quantities as far as we know. 
%%We mention that Duan and Yu \cite{Duan2020a} gave the time decay rate for $\na_x\phi$ and $\f$ for the Vlasov-Poisson-Landau system. 
%\end{Rem}
\begin{Rem}
In our analysis, we will mainly focus on the case of rectangular duct. The proof for $\R\times\T^2$ is similar. 
%	Using the similar arguments proving Theorem \ref{Main1}, one can obtain the same result in instead of $\Omega$ given in \eqref{Omega}. 
In fact, the only different calculations are on the boundary values when applying integration by parts. For the case of torus, one can always apply the integration-by-parts technique since $\T^2$ is periodic.  
\end{Rem}

For the case of rectangular duct with specular-reflection boundary conditions, we present the high-order compatible specular boundary argument as in \cite{Deng2021} to investigate the boundary conditions of $(\rho,u,\th)$, $\g$ and $\f$ on $\pa\Omega$. With the specular-reflection effect, the boundary terms arising from operator $v\cdot\na_x$ vanish perfectly by using reflect operator $R_xv = v - 2n(x)(n(x)\cdot v)$. Moreover, we can obtain boundary values
\begin{align*}
	\pa_{x_ix_ix_i}\big(\rho(x),\th(x),u_j(x)\big)=\pa_{x_i}\big(\rho(x),\th(x),u_j(x)\big)=
	\pa_{x_ix_i}u_i(x) = u_i(x) =0,
\end{align*}
on $\Gamma_i$ $(i=2,3)$. 
 Then we can perform the integration-by-parts technique with respect to spatial variables. 
%For example, with changing of variable $v\mapsto R_xv$ and noticing $F(R_xv)=F(v)$ on the boundary $\pa\Omega$, we have 
%\begin{align*}
%	\int_{\pa\Omega}\int_{\R^3}v\cdot n\big|{ F}(v)\big|^2\,dvdS(x) &= \int_{\pa\Omega}\int_{\R^3}R_xv\cdot n\big|{ F}(R_xv)\big|^2\,dvdS(x)\\
%	&=-\int_{\pa\Omega}\int_{\R^3}v\cdot n\big|{ F}(v)\big|^2\,dvdS(x)=0, 
%\end{align*}since $R_xv\cdot n=-v\cdot n$ on $\pa\Omega$. 
%This nice property can be applied widely in the boundary analysis for Boltzmann equation. 

\smallskip In contrast to the cutoff Boltzmann equation, the non-cutoff case needs to calculate the estimate on $L^{-1}_M$ by using coercive estimate on $L_M$, which possesses a regularizing effect of order $s$ with respect to the velocity variable. With this, we can calculate the estimate on $L^{-1}_MQ(G,G)$ and $L^{-1}_M\na_vF_2$, which can be controlled by using dissipation norm $|\cdot|_{L^2_D}$.

\smallskip

The rest of the paper is organized as follows. In Section \ref{Sec2}, we list the basic properties for rarefaction wave of Euler equations and the properties of Boltzmann collision operators. In Section \ref{Sec3}, we derive the high-order specular boundary condition and estimate on $L^{-1}_M$ and collision operator. In Section \ref{Sec4}, we calculate the energy estimates for the macroscopic components. In Section \ref{Sec5}, we derive the energy estimates of the microscopic components. In Section \ref{Sec6}, we conclude Theorem \ref{Main1} by using the {\it a priori} arguments.

\section{Preliminaries}\label{Sec2}

In this section, 
we give some properties for the rarefaction wave for Euler equations given in \eqref{rarer} and \eqref{rarea}, and then list some basic properties on Burnett functions and $\ol G$. Also, we derive the high-order compatible specular boundary condition for the case of rectangular duct. 
%present the local-in-time existence of solution to Boltzmann equation \eqref{F1} with \eqref{specular}, 

%We begin with the local-in-time existence, whose proof will be given in ??

\smallskip

We first give some properties of the smoothing approximated $3$-rarefaction wave. 
% constructed in \eqref{rarer} and \eqref{rarea}. 
\begin{Lem}[\cite{Ito1996}]\label{Lem20}
If $w_-<w_+$, the smooth rarefaction wave $\bar w$ has the following properties. 
	
	\noindent (i) $w_-<w(t,x_1)<w_+$, $w_{x_1}(t,x_1)>0$, for $(t,x_1)\in\R_+\times\R$. 
	
	\noindent (ii) For any $1\le q\le \infty$, there exists $C_q>0$ such that 
	\begin{equation*}
		\|w(t,x_1)-w^r(\frac{x_1}{1+t})\|_{L^q_{x_1}}\le C_q(w_+-w_-)^{\frac{1}{q}},
	\end{equation*}
\begin{equation*}
	\|\pa_{x_1}w(t,x_1)\|_{L^q_{x_1}}\le C_q\min\big\{w_+-w_-,(w_+-w_-)^{\frac{1}{q}}(1+t)^{-1+\frac{1}{q}}\big\},
\end{equation*}
\begin{equation*}
	\|\pa^j_{x_1}w(t,x_1)\|_{L^q_{x_1}}\le C_q\min\big\{w_+-w_-,(1+t)^{-1}\big\},\ j\ge 2.
\end{equation*}
%	\begin{align*}
%		\|w(t,x_1)-w^r(\frac{x_1}{1+t})\|_{L^q_{x_1}}&\le C_q(w_+-w_-)^{\frac{1}{q}},\\
%		\|\pa_{x_1}w(t,x_1)\|_{L^q_{x_1}}&\le C_q\min\big\{w_+-w_-,(w_+-w_-)^{\frac{1}{q}}(1+t)^{-1+\frac{1}{q}}\big\},\\
%		\|\pa^j_{x_1}w(t,x_1)\|_{L^q_{x_1}}&\le C_q\min\big\{w_+-w_-,(1+t)^{-1}\big\},\ j\ge 2.
%	\end{align*}

\noindent (iii) Time asymptotically, for $1<q\le \infty$, 
\begin{align*}
	\big\|w(t,x_1)-w^r\big(\frac{x_1}{t+1}\big)\big\|_{L^q_{x_1}}\le C_q(1+t)^{-\frac{1}{2}+\frac{1}{2q}}. 
\end{align*}
\end{Lem}
The proof of Lemma \ref{Lem20} (i) and (ii) can be found in \cite[Lemma 2.1]{Matsumura1986}. The time asymptotic decay can be found in \cite[Lemma 2, pp. 318]{Ito1996}. Using Lemma \ref{Lem20} and applying the method in \cite[Section 3, pp. 345-346]{Liu2006a}, we can obtain the decay properties for $(\bar\rho,\bar u,\bar\th)$ as the followings. 
%It can be proved by using the similar arguments in \cite{Matsumura1986} and the proof is omitted for brevity; see also \cite{Li2017}.  
\begin{Lem}\label{Lem21}
	The approximate $3$-rarefaction wave $(\bar\rho,\bar u,\bar\th)(t,x)$ defined in \eqref{rarea} satisfies the following properties: 
	
	\medskip
	\noindent	
	(i) $\bar u_{1x_1}>0$ for $x\in\R$ and $t\ge 0$.\\
	(ii) The following estimates hold for all $t>0$ and $q\in[1,\infty]$:
	\begin{equation*}
		\|(\bar\rho,\bar u,\bar\th)(t,\cdot)-(\rho^r,u^r,\th^r)(\frac{\cdot}{1+t})\|_{L^q_{x_1}}\le C\delta,
	\end{equation*}
	\begin{equation*}
		\|\pa_{x_1}(\bar\rho,\bar u,\bar\th)(t,\cdot)\|_{L^q_{x_1}}\le C_q\delta^{1/q}(1+t)^{-1+1/q},
	\end{equation*}
	\begin{equation*}
		\|\pa^j_{x_1}(\bar\rho,\bar u,\bar\th)(t,\cdot)\|_{L^q_{x_1}}\le C_q\min\{\delta,\,(1+t)^{-1}\},\quad j\ge 2.
	\end{equation*}
	(iii) Time asymptotically, for $q\in(1,\infty]$, 
	\begin{align*}
%		\label{time2}
		\big\|(\bar\rho,\bar u,\bar\th)(t,\cdot)-(\rho^r,u^r,\th^r)\big(\frac{\cdot}{1+t}\big)\big\|_{L^q_{x_1}}\le C_q(1+t)^{-\frac{1}{2}+\frac{1}{2q}}. 
	\end{align*}
\end{Lem}

We will frequently apply the following equivalence of norms and Gagliardo–Nirenberg interpolation inequality. 
\begin{Lem}[\cite{Deng2020a}, Corollary 2.5]
	For any $k,l\in\R$ and any $f\in H^l_k$,  we have 
	\begin{align}\label{vdv}
		|\<v\>^k\<D_v\>^lf|_{L^2_v}\approx |\<D_v\>^l\<v\>^kf|_{L^2_v}.
	\end{align}
\end{Lem}
%Next we give the form of Gagliardo–Nirenberg interpolation inequality in $\Omega$. 
\begin{Lem}\label{Lemembedd}
	For any $f\in H^2_x(\Omega)$, we have 
	\begin{align}
		\label{Gag}
		\|f\|_{L^\infty_{x}}&\lesssim\|f\|_{H^2_{x}}^{17/18}\|\na_xf\|_{L^2_{x}}^{1/18}, \\
		\label{Gag1}
		\|f\|_{L^\infty_{x}}&\lesssim\|f\|_{H^2_x}^{1/2}\|\na_xf\|_{H^2_x}^{1/2}. 
	\end{align}
	For $f\in H^1_x(\Omega)$, we have 
	\begin{equation}
		\label{GagL3}
		\|f\|_{L^3_x} \lesssim\|f\|_{L^2_{x}}^{1/2}\|f\|_{H^1_{x}}^{1/2},
		%	^{\frac{1}{2}}\|\na_xf\|_{L^2_x}^{\frac{1}{2}}. 
	\end{equation}
	\begin{equation}
		\label{GagL3a} 
		\|f\|_{L^3_x} \lesssim\|\pa_{x_1}f\|_{L^2_{x}}^{1/6}\|f\|_{L^2_{x_1}H^1_{x_2,x_3}}^{5/6},  
	\end{equation}
	and 
	\begin{align}
		\label{GagL6}
		\|f\|_{L^6_x} \lesssim\|f\|_{H^1_x}.
%		\|\na_xf\|^{1/3}_{L^2_{x}}\|f\|^{2/3}_{H^1_x}.
	\end{align}
\end{Lem}
\begin{proof}
	Denote $\Omega_1$ to be $\T^2$ or $\cup_{i=1}^N (a_{i,2},b_{i,2})\times (a_{i,3},b_{i,3})$. 
	Firstly, we use extension theorem \cite[Thoerem VI.5, pp. 181]{Stein1971} to extend function $f$ in domain $\Omega_1$ with Lipschitz boundary to a function $Ef$ in $\R^2$ such that $Ef=f$ in $\Omega$ and 
	\begin{align*}
		\|Ef\|_{H^k(\R^2)}\lesssim \|f\|_{H^k(\Omega_1)},
	\end{align*}
	for any $k\ge 0$. 
By Gagliardo–Nirenberg interpolation inequality on $\R^3$, 
%	 on $x_1\in\R$,
(cf. \cite[Theorem 12.83]{Leoni2017} and \cite[Page 125]{Nirenberg1959}), we obtain 
	%\begin{align*}
	%	\|Ef\|_{L^\infty_x(\R^3)}&\lesssim \|\na_xEf\|_{L^2_x(\R^3)}^{1/2}\|\na_x^2Ef\|_{L^2_x(\R^3)}^{1/2}\\
	%	&\lesssim\|f\|_{H^1_x(\Omega)}^{1/2}\|f\|_{H^2_x(\Omega)}^{1/2}, 
	%\end{align*}
	%	 on $x_1\in\R$,
	% (cf. \cite[Theorem 12.83]{Leoni2017} and \cite[Page 125]{Nirenberg1959}), 
%	we have 
	\begin{align*}
		\|Ef\|_{L^3_x(\R^3)}&\lesssim \|Ef\|_{L^2_x(\R^3)}^{1/2}\|\na_xEf\|_{L^2_x(\R^3)}^{1/2}\\
		&\lesssim\|f\|_{L^2_x(\Omega)}^{1/2}\|f\|_{H^1_x(\Omega)}^{1/2},  
	\end{align*}
	which gives \eqref{GagL3}. 
	%and 
	%\begin{align*}
	%\|Ef\|_{L^6_x(\R^3)}\lesssim \|\na_xEf\|_{L^2_x(\R^3)}
	%\lesssim\|f\|_{H^1_x(\Omega)}. 
	%\end{align*}
	For \eqref{GagL3a}, we apply Gagliardo–Nirenberg interpolation inequality on $x_1\in\R$ to obtain 
	\begin{align*}
		\|f\|_{L^3_{x_1}}\lesssim \|\na_{x_1}f\|_{L^2_{x_1}}^{1/6}\|f\|_{L^2_{x_1}}^{5/6}. 
	\end{align*}
	For $(x_2,x_3)$ direction, 
	%	\in \cup_{i=1}^N(a_{i,2},b_{i,2})\times (a_{i,3},b_{i,3})$ as in \eqref{Omega}, 
	we can apply Sobolev inequality to obtain  
	\begin{align*}
		\|f\|_{L^3_x}&\lesssim \big\|\|\na_{x_1}f\|_{L^2_{x_1}}^{1/6}\|f\|_{L^2_{x_1}}^{5/6}\big\|_{L^3_{x_2,x_3}}\\
		&\lesssim \|\na_{x_1}f\|_{L^2_{x}}^{1/6}\big\|\|f\|_{L^2_{x_1}}\big\|_{L^{\frac{10}{3}}_{x_2,x_3}}^{5/6}\\
		&\lesssim \|\na_{x_1}f\|_{L^2_{x}}^{1/6}\|f\|_{L^2_{x_1}H^1_{x_2,x_3}}^{5/6}. 
	\end{align*}
	For \eqref{GagL6}, we apply Gagliardo–Nirenberg inequality on $\R^3$ to obtain 
	\begin{align*}
		\|f\|_{L^6_x(\Omega)}\lesssim \|Ef\|_{L^6_x(\R^3)}\lesssim \|\na_xEf\|_{L^2_x(\R^3)}\lesssim \|f\|_{H^1_x(\Omega)}. 
	\end{align*}
For \eqref{Gag}, by Gagliardo–Nirenberg interpolation inequality on $\R^3$ and \eqref{GagL3a}, \eqref{GagL6}, we have 
\begin{align*}
	\|Ef\|_{L^\infty_{x}(\R^3_x)}\lesssim \|Ef\|_{L^3_{x}(\R^3_x)}^{1/3}\|\na_xEf\|_{L^6_{x}(\R^3_x)}^{2/3}
	\lesssim \|f\|_{L^3_{x}(\Omega)}^{1/3}\|f\|_{H^2_{x}(\Omega)}^{2/3}
	\lesssim \|\na_{x_1}f\|_{L^2_{x}(\Omega)}^{1/18}\|f\|_{H^2_{x}(\Omega)}^{17/18}.  
\end{align*}
For \eqref{Gag1}, we apply Gagliardo–Nirenberg interpolation inequality on $\R$ to obtain 
\begin{align*}
	\|f\|_{L^\infty_{x_1}}\lesssim \|\pa_{x_1}f\|_{L^2_{x_1}}^{1/2}\|f\|_{L^2_{x_1}}^{1/2}. 
\end{align*}
By Fourier transform $\wh f=\int_{\R^2}f(x')e^{-2\pi ix'\cdot\xi}\,dx'$ on $\R^2$, we have 
\begin{multline}\label{25}
	\|Ef\|_{L^\infty_{x_2,x_3}(\R^2)}\lesssim\|\wh{Ef}\|_{L^1_{x_2,x_3}(\R^2)}
	\lesssim \|\<\xi\>^2\wh{Ef}\|_{L^1_{x_2,x_3}(\R^2)}\|\<\xi\>^{-2}\|_{L^1_{x_2,x_3}(\R^2)}\\
	\lesssim \|Ef\|_{H^2_{x_2,x_3}(\R^2)}\lesssim \|f\|_{H^2_{x_2,x_3}(\Omega_1)}.
\end{multline}
Combining the above estimates, we have 
\begin{align*}
	\|f\|_{L^\infty_x}\lesssim \big\|\|f\|_{L^\infty_{x_1}}\big\|_{L^\infty_{x_2,x_3}}\lesssim \|\pa_{x_1}f\|^{1/2}_{L^2_{x_1}H^2_{x_2,x_3}}\|f\|^{1/2}_{L^2_{x_1}H^2_{x_2,x_3}}, 
\end{align*}
which implies \eqref{Gag1}. 
	Then we conclude the Lemma. 
\end{proof}

Next we reformulate $\ol{G}$. Recall the Burnett functions (cf. \cite{Guo2006,Ukai,Duan2020a}) defined by 
\begin{align*}
%	\label{211}
	\wh{A}_j(v) = \frac{|v|^2-5}{2}v_j\ \text{ and }\ \wh{B}_{ij}(v)=v_iv_j-\frac{1}{3}\delta_{ij}|v|^2\ \text{ for }i,j=1,2,3,
\end{align*}
where $\delta_{ij}$ is the Kronecker delta. 
Noticing that $\wh{A}_j(\frac{v-u}{\sqrt{R\th}})M$ and $\wh{B}_{ij}(\frac{v-u}{\sqrt{R\th}})M$ are orthogonal to the null space $\ker L_M$ of the linearized operator $L_M$, we can define $A_j(\frac{v-u}{\sqrt{R\th}})$ and $B_{ij}(\frac{v-u}{\sqrt{R\th}})$ such that $P_0A_j(\frac{v-u}{\sqrt{R\th}})=P_0B_{ij}(\frac{v-u}{\sqrt{R\th}})=0$ and 
\begin{align}\label{21}
	A_j\big(\frac{v-u}{\sqrt{R\th}}\big) = L^{-1}_M\big(\wh{A}_j(\frac{v-u}{\sqrt{R\th}})M\big)\ \text{ and }\
	B_{ij}\big(\frac{v-u}{\sqrt{R\th}}\big) = L^{-1}_M\big(\wh{B}_{ij}(\frac{v-u}{\sqrt{R\th}})M\big).
\end{align}	
By definition \eqref{olG}, \eqref{21} and direct computation on $P_1=I-P_0$, we have (cf. \cite[Lemma 6.7]{Duan2020a}) 
\begin{align}\label{olG3}
	P_1v_1M\Big\{\frac{|v-u|^2\bar\th_{x_1}}{2R\th^2}+\frac{(v-u)\cdot \bar{u}_{x_1}}{R\th}\Big\} = \frac{\sqrt{R}}{\sqrt{\th}}\bar{\th}_{x_1}\wh A_1\big(\frac{v-u}{\sqrt{R\th}}\big)M + \bar u_{1x_1}\wh B_{11}\big(\frac{v-u}{\sqrt{R\th}}\big)M. 
\end{align}	
and 
\begin{align}\label{olG2}
	\ol G = \frac{\sqrt{R}}{\sqrt{\th}}\bar{\th}_{x_1} A_1\big(\frac{v-u}{\sqrt{R\th}}\big) + \bar u_{1x_1}B_{11}\big(\frac{v-u}{\sqrt{R\th}}\big). 
\end{align}
We list some elementary properties of the Burnett functions; cf. \cite{Bardos1991}, \cite[Proposition 4.3 and Lemma 4.4]{Bardos1993}, \cite[Page 647]{Guo2006} and \cite[Proposition 3.2.1]{Ukai}. One may refer to \cite[Proposition 4.3 and Lemma 4.4]{Bardos1993} for the proof. The last statement of Lemma \ref{LemBur} comes from \cite[Eq. (A.2), pp. 3636]{Duan2015}. 
%Here we use the version from \cite[Proposition 3.2.1]{Ukai}.
\begin{Lem}\label{LemBur}
	The Burnett functions have the following properties:
	\begin{itemize}
		\item $-(\wh A_i,A_i)_{L^2_v}$ is positive and independent of $i$;
		\item $(\wh A_i,A_j)_{L^2_v}=0$ for any $i\neq j$;
		\item $(\wh A_i,B_{jk})_{L^2_v}=0$ for $i,j,k$;
		\item $(\wh B_{ij},B_{kl})_{L^2_v}=(\wh B_{kl},B_{ij})_{L^2_v}=(\wh B_{ji},B_{kl})_{L^2_v}$, which is independent of $i,j$ for fixed $k,l$;
		\item $-(\wh B_{ij},B_{ij})_{L^2_v}$ is positive and independent of $i,j$ for $i\neq j$;
		\item $-(\wh B_{ii},B_{jj})_{L^2_v}$ is positive and independent of $i$ for $i\neq j$;
		\item $-(\wh B_{ii},B_{ii})_{L^2_v}$ is positive and independent of $i$;
		\item $(\wh B_{ij},B_{kl})_{L^2_v}=0$ unless either $(i,j)=(k,l)$ or $(l,k)$, or $i=j$ and $k=l$;
		\item $(\wh B_{ii},B_{ii})_{L^2_v}=3(\wh B_{ij},B_{ij})_{L^2_v}$ for $i\neq j$. 
	\end{itemize}
\end{Lem}
In terms of Burnett functions, as in \cite[Eq. (4.10), pp. 710]{Bardos1993}, the viscosity coefficient $\mu(\th)$ and the heat conductivity coefficient $\kappa(\th)$ can be represented by 
\begin{equation}\label{mu}
	\mu(\th) = -R\th\int_{\R^3}B_{ij}\big(\frac{v-u}{\sqrt{R\th}}\big)\wh B_{ij}\big(\frac{v-u}{\sqrt{R\th}}\big)\,dv >0, \quad i\neq j,
\end{equation}
and 
\begin{equation}\label{kappa}
	\kappa(\th) = -R^2\th\int_{\R^3}A_j\big(\frac{v-u}{\sqrt{R\th}}\big)\wh A_j\big(\frac{v-u}{\sqrt{R\th}}\big)\,dv >0.
\end{equation}

\medskip
In the following we shall recall the estimates on Boltzmann collision operator.

\begin{Lem}\label{LemGamma}
	Assume $\gamma>\max\{-3,-2s-3/2\}$. Denote $w_l=\<v\>^l$
		 and $\wt\L$ be the linearized operator $\L$ or $\L_2$. 
	%		For any $\eta>0$ small enough, there exists $C_\eta>0$ such that f
	Then for any multi-index $\al$, we have 
	\begin{equation}\label{L1}
		-(\pa^\alpha \wt\L g, w_l^2\pa^\al g)_{L^2_v}\gtrsim |w_l\pa^\al g|^2_{L^2_D}
%		 - C\sum_{|\beta'|<|\beta|}|w_l\pa^\al_{\beta'}g|^2_{L^2_D} 
- C|\pa^\al g|^2_{L^2_{B_C}},
	\end{equation}
	where $L^2_{B_C}$ is the $L^2_v$ space restricted on ball $B_C=\{|v|\le C\}$. 
%	If $|\beta|=0$, we have 
%	\begin{equation}\label{L2}
%		-(\pa^\alpha \L g, w_l^2\pa^\al g)_{L^2_v}\gtrsim |w_l\pa^\al g|^2_{L^2_D} - C|\pa^\al g|^2_{L^2_{B_C}}.
%	\end{equation}
	The null space $\ker \L$ and $\ker\L_2$ of $\L$ and $\L_2$ is spanned by $\big\{\sqrt\mu,v_i\sqrt\mu$ $(i=1,2,3),|v|^2\sqrt\mu\big\}$ and $\big\{\sqrt\mu\big\}$ respectively. For $g\in(\ker \L)^\perp$ and $h\in(\ker\L_2)^\perp$, there exists $\lam>0$ such that 
	\begin{equation}
		\begin{aligned}\label{Lg}
			-(\L g,g)_{L^2_v}\ge \lam|g|^2_{L^2_D},\\
			-(\L_2h,h)_{L^2_v}\ge \lam|h|^2_{L^2_D}. 
		\end{aligned}
	\end{equation}
	%The null space $\N_2$ of $\L_2$ is spanned by $\sqrt\mu$ by similar argument as $\L$. Then by \cite[Proposition 2.1, pp. 925]{Alexandre2012}, for $g\in\N_2^\perp$, there exists $\lam>0$ such that 
	%\begin{align}\label{L2g}
	%	-(\L_2 g,g)_{L^2_v}\ge \lam|g|^2_{L^2_D}.
	%\end{align}
	For the estimate on $\Gamma$, if $\gamma+2s\ge 0$, we have 
	\begin{align}\label{Gamma}
		|(\Gamma(f,g), w_l^2 h)_{L^2_v}| \lesssim \big(|f|_{L^2_v}|w_lg|_{L^2_D}+|w_lf|_{L^2_v}|g|_{L^2_D}\big)|w_l h|_{L^2_D}. 
	\end{align}
	If $\gamma>\max\{-3,-2s-3/2\}$, then 
	\begin{align}\label{Gamma1}
		|(\Gamma(f,g), w_l^2 h)_{L^2_v}| \lesssim \big(|\<v\>^{\frac{\gamma+2s}{2}}f|_{L^2_v}|w_lg|_{L^2_D}+|\<v\>^{\frac{\gamma+2s}{2}}w_lf|_{L^2_v}|g|_{L^2_D}\big)|w_l h|_{L^2_D}. 
	\end{align}
Consequently, denoting $\nu(v)=\min\{1,\<v\>^{(\gamma+2s)/2}\}$, we have 
\begin{align}\label{2134}
	|w_l\Gamma(f,g)|_{H^{-s}_{-\gamma/2}}\lesssim |\nu f|_{L^2_v}|w_lg|_{L^2_D}+|\nu w_lf|_{L^2_v}|g|_{L^2_D}. 
\end{align}
%{\red goal??}
\end{Lem}
\begin{proof}
%	The proof of \eqref{L1} can be found in \cite[Lemma 2.6 and Lemma 2.7]{Duan2013a}. 
		When $\wt\L=\L$, the proofs of \eqref{L1} can be found in \cite[Lemma 2.6 and pp. 784]{Gressman2011}.
		When $\wt\L=\L_2$, \eqref{L1} can seen from \cite[Propostion 4.8, pp. 983]{Alexandre2012}.
	The estimate \eqref{Lg} can be seen from \cite[Proposition 2.1, pp. 925]{Alexandre2012}. 
	Note that $|\cdot|_{L^2_D}$ is equivalent to the triple norm $\vertiii{\cdot}$ in \cite{Alexandre2012}. 
	%%	The proof of \eqref{Gamma} is given in \cite[Eq. (6.6), pp. 817]{Gressman2011}.
%	The proof of \eqref{Gamma} is given in \cite[Eq. (6.6), pp. 817]{Gressman2011}. 
	We refer to \cite[Lemma 2.2, pp. 1612]{Deng2021b} for the proof of \eqref{Gamma}. 
	Here we give the proof on \eqref{Gamma1}. 
	%	Assume $\gamma>\max\{-3,-2s-3/2\}$ in the following. 
	We would like to perform the calculation in \cite{Gressman2011} and hence write the following functionals.
	Let $\{\chi_k\}_{k=-\infty}^{k=+\infty}$ be a partition of unity on $(0,\infty)$ such that $|\chi_k|\le 1$ and supp$(\chi_k)\subset [2^{-k-1},2^{-k}]$. For each $k$, we define 
	$
	B_k = B(v-v_*,\sigma)\chi_k(|v-v'|),
	$
	and
	\begin{align*}
		T^{k,l}_+(f,g,h)=\int_{\R^3}dv\int_{\R^3}dv_*\int_{\mathbb{S}^{2}}d\sigma\,B_k(v-v_*,\sigma)f_*gh'w_l^2(v')\mu^{1/2}(v_*'),\\
		T^{k,l}_-(f,g,h)=\int_{\R^3}dv\int_{\R^3}dv_*\int_{\mathbb{S}^{2}}d\sigma\,B_k(v-v_*,\sigma)f_*ghw_l^2(v)\mu^{1/2}(v_*).
	\end{align*}
	Given the support of $\chi_k$, we have 
	\begin{align}\label{355}
		\int_{\S^2}d\sigma\,B_k\lesssim |v-v_*|^\gamma\int^{2^{-k}|v-v_*|^{-1}}_{2^{-k-1}|v-v_*|^{-1}}d\th\,\th^{-1-2s}\lesssim 2^{2sk}|v-v_*|^{\gamma+2s}, 
	\end{align}
	and 
	\begin{align}\label{355a}
		\int_{\R^3}dv_*\,\sqrt{\mu_*}|v-v_*|^{2(\gamma+2s)}\lesssim \<v\>^{2(\gamma+2s)}. 
	\end{align}
%Then by Cauchy-Schwarz inequality, we have 
%\begin{align}\label{232}
%	\big|T^{k,l}_-(f,g,h)\big|
%	&\lesssim 2^{2sk}\int_{\R^3}dv\int_{\R^3}dv_*|v-v_*|^{\gamma+2s}f_*ghw_l^2(v)\partial_{\beta}(\mu^{1/2}(v_*))\\
%	&\lesssim 2^{2sk}\Big(\int_{\R^3}dv\int_{\R^3}dv_*\<v\>^{\gamma+2s}|f_*g|^2w_l^2(v)\partial_{\beta}(\mu^{1/2}(v_*))\Big)^{1/2}\\
%	&\quad\times \Big(\int_{\R^3}dv\int_{\R^3}dv_*\<v\>^{-(\gamma+2s)}|v-v_*|^{2(\gamma+2s)}|h|^2w_l^2(v)\partial_{\beta}(\mu^{1/2}(v_*))\Big)^{1/2}
%\end{align}
	Applying the same proof in \cite[Proposition 3.1]{Gressman2011}, we have that for any integer $k$, any $m\ge 0$ and $l\in\R$, 
	\begin{align}\label{232}
		\big|T^{k,l}_-(f,g,h)\big|\lesssim 2^{2sk}|\<v\>^{-m}g|_{L^2_v}|\<v\>^{\frac{\gamma+2s}{2}}w_lh|_{L^2_v}|\<v\>^{\frac{\gamma+2s}{2}}w_lf|_{L^2_v}.
		%		,\\
		%		\Big|T^{k,l}_-(f,g,h)\Big|\lesssim 2^{2sk}|\<v\>^{-m}g|_{L^2_v}|\<v\>^{\frac{\gamma+2s}{2}}f|_{H^2_v}|\<v\>^{\frac{\gamma+2s}{2}}h|_{L^2_v}.
		%%		\\
		%		\Big|T^{k,l}_-(f,\phi,g)\Big|+\Big|T^{k,l}_-(f,g,\phi)\Big|\lesssim C_\phi2^{2sk}|\<v\>^{-m}f|_{L^2_v} |\<v\>^{-m}g|_{L^2_v},
	\end{align}
	Next we give the estimate on $T^{k,l}_+(g,h,f)$.
	% by using similar argument in \cite[Proposition 3.3]{Gressman2011}. 
	By Cauchy-Schwarz inequality, 
	\begin{align}\label{233}\notag
		|T^{k,l}_+(g,h,f)|&\lesssim \Big(\int_{\R^3}dv\int_{\R^3}dv_*\int_{\S^{2}}d\sigma\,\frac{B_k(v-v_*,\sigma)}{|v-v_*|^{\gamma+2s}}|g_*|^2|h|^2\<v'\>^{\gamma+2s}w_l^2(v')\sqrt{\mu'_*}\Big)^{\frac{1}{2}}\\
		&\qquad\times \Big(\int_{\R^3}dv\int_{\R^3}dv_*\int_{\S^{2}}d\sigma\,\frac{B_k(v-v_*,\sigma)}{|v-v_*|^{-\gamma-2s}}|f'|^2\<v'\>^{-\gamma-2s}w_l^2(v')\sqrt{\mu'_*}\Big)^{\frac{1}{2}}.
	\end{align}
	After a pre-post change of variables, the second factor on the right hand side of \eqref{233} is bounded above by a uniform constant times $2^{sk}|\<v\>^{\frac{\gamma+2s}{2}}w_lf|_{L^2_v}$ by using \eqref{355} and \eqref{355a}. 
	This is the only place we implicitly use $\gamma+2s>-3/2$. For the first factor on the right hand side of \eqref{233}, we need the following calculations. If $|v'|^2\le \frac{1}{2}(|v|^2+|v_*|^2)$, then from the collisional conservation laws, we have $\mu'_*\le \sqrt{\mu\mu_*}$ and hence 
	\begin{align*}
		\<v'\>^{\gamma+2s}w_l^2(v')\sqrt{\mu'_*}\lesssim \<v\>^{-m}\<v_*\>^{-m},
	\end{align*}
	for any $m>0$. On the region $|v'|^2\ge \frac{1}{2}(|v|^2+|v_*|^2)$, it follows from collisional geometry that $|v'|^2\approx |v|^2+|v_*|^2$. Then 
	\begin{align*}
		\<v'\>^{\gamma+2s}w_l^2(v')\sqrt{\mu'_*}\lesssim \big(w_l^2(v_*)+w_l^2(v)\big)\<v\>^{\gamma+2s}\<v_*\>^{\gamma+2s}. 
	\end{align*}
	Substituting the above two estimates into the first factor on the right hand side of \eqref{233}, we have 
	\begin{align}\label{234}
		|T^{k,l}_+(g,h,f)|\lesssim 2^{2sk}\Big(|\<v\>^{\frac{\gamma+2s}{2}}g|_{L^2_v}|\<v\>^{\frac{\gamma+2s}{2}}w_lh|_{L^2_v}+|\<v\>^{\frac{\gamma+2s}{2}}w_lg|_{L^2_v}|\<v\>^{\frac{\gamma+2s}{2}}h|_{L^2_v}\Big)|\<v\>^{\frac{\gamma+2s}{2}}w_lf|_{L^2_v}. 
	\end{align}
	Combining estimate \eqref{232}, \eqref{234} and \cite[Propositon 3.1, 3.4, 3.6, 3.7]{Gressman2011} and using the same arguments as Section 6.1 in \cite{Gressman2011}, we can obtain 
	\begin{align*}
		|(\Gamma(g,h),w^2_lf)|\lesssim \Big(|\<v\>^{\frac{\gamma+2s}{2}}g|_{L^2_v}|w_lh|_{L^2_D}+|\<v\>^{\frac{\gamma+2s}{2}}w_lg|_{L^2_v}|h|_{L^2_D}\Big)|w_lf|_{L^2_D}. 
	\end{align*}
	This implies \eqref{Gamma1}. Noticing $|\cdot|_{L^2_D}\lesssim |\cdot|_{H^s_{\gamma/2}}$, by duality, we can obtain from \eqref{Gamma} and \eqref{Gamma1} that 
	\begin{align*}
		|w_l\Gamma(f,g)|_{H^{-s}_{-\gamma/2}}\lesssim |\nu f|_{L^2_v}|w_lg|_{L^2_D}+|\nu w_lf|_{L^2_v}|g|_{L^2_D}. 
	\end{align*}
	 This completes the proof of Lemma \ref{LemGamma}. 	
\end{proof}.

\section{Basic Estimates}\label{Sec3}
In this section, we shall apply the lemmas in Section \ref{Sec2} to derive some basic estimates on $(\bar\rho,\bar u,\bar\th)$, Boltzmann collision operator $Q$.% and Poisson term $\na_x\phi$ for later use. 
%With the above Lemma, we give some estimates on $Q$ that will be needed in the next Section. Also, for energy functional with weight, we temporarily use notation
%\begin{align}
%	\label{E2}
%	\E_k(t) = \sum_{|\al|\le {5}}\|\pa^\al (\wt{\rho},\wt u,\wt\th)(t)\|^2_{L^2_x} + \sum_{|\al|\le {5}}\|w(\al)\pa^\al \g\|^2_{L^2_xL^2_v},
%\end{align}
%and 
%\begin{align}
%	\label{D2}
%	\D_k(t) = \sum_{1\le|\al|\le {5}}\|\pa^\al (\wt{\rho},\wt u,\wt\th)(t)\|^2_{L^2_x} + \sum_{|\al|\le {5}}\|w(\al)\pa^\al \g\|^2_{L^2_xL^2_D}
%\end{align}
%in this section. However, we will apply the following Lemmas in the case of $l=0$. 

%\subsection{Time Decay for $(\bar\rho,\bar u,\bar\th)$}
We begin with the large time decay on $(\bar\rho,\bar u,\bar\th)$. 
\begin{Lem}\label{LL31}
	For any $|\al|\ge 1$, we have 
	\begin{equation}\label{pat1}
		\begin{aligned}
			\|\pa_t(\bar\rho,\bar u,\bar\th)\|_{L^2_x}^2&\lesssim  \delta(1+t)^{-1},\\
			\|\pa^\al \pa_t(\bar\rho,\bar u,\bar\th)\|_{L^2_x}^2&\lesssim  \delta^{\frac{2}{3}}(1+t)^{-\frac{4}{3}}.
		\end{aligned} 
	\end{equation}
\end{Lem}
\begin{proof}
	For the derivatives with respect to $t$ on $(\bar\rho,\bar{u},\bar\th)$, we apply \eqref{Euler2} to deduce that 
	\begin{equation*}
%		\label{bar}
		\left\{
		\begin{aligned}
			&\bar\rho_t + \bar\rho_{x_1}\bar u+\bar\rho\bar u_{x_1} = 0,\\
			&\bar u_{1t}+\bar u_1\bar u_{1x_1}+\frac{R\bar\th}{\rho}\bar\rho_{x_1}+R\bar\th_{x_1} = 0,\\
			&\bar\th_t+\bar u_1\bar\th_{x_1}+{R\bar\th}{}\bar u_{1x_1} = 0. 
		\end{aligned}\right.
	\end{equation*}
	Thus, by Lemma \ref{Lem21}, we have 
	\begin{align*}
		\|\pa_t(\bar\rho,\bar u,\bar\th)\|_{L^2_x}& 
%		\|\na_x(\bar\rho,\bar u,\bar\th)\|_{L^2_x} + \|\na_x(\bar\rho,\bar u,\bar\th)\|_{L^6_x}\|(\bar\rho,\bar u,\bar\th)\|_{L^3_x}\\
	\lesssim \|\na_x(\bar\rho,\bar u,\bar\th)\|_{L^2_x}
%		 + \|\na^2_x(\bar\rho,\bar u,\bar\th)\|^{\frac{3}{2}}_{L^2_x}\|\na_x(\bar\rho,\bar u,\bar\th)\|^{\frac{1}{2}}_{L^2_x} 
\lesssim\delta^{\frac{1}{2}}(1+t)^{-\frac{1}{2}}. 
	\end{align*}
Similarly, for $\al=(\al_0,\al_1,\al_2,\al_3)$ such that $\al_0=0$ and $|\al|\ge 1$, we have 
	\begin{align*}
	\|\pa_t\pa^{\al}(\bar\rho,\bar u,\bar\th)\|_{L^2_x}&\lesssim
	\|\na_x\pa^{\al}(\bar\rho,\bar u,\bar\th)\|_{L^2_x} + \sum_{\al'\le\al}\|\pa^{\al'}\na_x(\bar\rho,\bar u,\bar\th)\|_{L^2_x}\sum_{0\neq\al'\le\al}\|\pa^{\al'}(\bar\rho,\bar u,\bar\th)\|_{L^\infty_x}\\
	&\lesssim \delta^{\frac{1}{3}}(1+t)^{-\frac{2}{3}}. 
\end{align*}
The proof for $|\al_0|\ge 1$ is by induction on $\pa_t$ and we can obtain the second estimate of \eqref{pat1}. This conclude Lemma \ref{LL31}. 
\end{proof}

\subsection{Compatible Specular Boundary Condition}
In this subsection, we derive the specular reflection boundary condition for high order derivatives and the corresponding boundary values for macroscopic components $(\rho,u,\th)$. 
\begin{Lem}
	\label{Lemspecular}
	Let $(F_\pm,\phi)$ be the solution to \eqref{1} (with \eqref{specular} and \eqref{Neumann} for the case of rectangular duct). Fix $i=2,3$, $x\in\Gamma_i$ and $v\in\R^3$ such that $v\cdot n(x)\neq 0$. Then for any multi-index $\al$ such that $|\al|\le 3$, we have 
	\begin{equation}\label{24}
		\pa^\al F_\pm(x,v)=(-1)^{|\al_i|}\pa^\al F_\pm(x,R_xv). 
	\end{equation}
	%	and		
	%	\begin{equation}\label{25}\begin{aligned}
	%			\pa_{\tau_j}F_\pm(x,R_xv) &= \pa_{\tau_j}F_\pm(x,v),\\
	%			\pa_{\tau_j\tau_k}F_\pm(x,R_xv) &= \pa_{\tau_j\tau_k}F_\pm(x,v),
	%			\\
	%			\pa_{\tau_j\tau_k\tau_k}F_\pm(x,R_xv) &= \pa_{\tau_j\tau_k\tau_k}F_\pm(x,v),
	%		\end{aligned}
	%	\end{equation}for $j,k=1,2$, 
	%	where $(\tau_1,\tau_2)$ are unit tangent vector on $\Gamma_i$.
	%	For the normal derivatives, we have that on $\{(x,v): v\cdot n(x)\neq 0,\ x\in\Gamma_i\}$,  
	%	\begin{equation}\label{26}
	%		\begin{aligned}
	%			\pa_{n}F_\pm(x,R_xv) &= -\pa_{n}F_\pm(x,v),\\
	%			\pa_{\tau_j}\pa_{n}F_\pm(x,R_xv) &= -\pa_{\tau_j}\pa_{n}F_\pm(x,v),
	%			\\
	%			\pa_{\tau_j\tau_k}\pa_{n}F_\pm(x,R_xv) &= -\pa_{\tau_j\tau_k}\pa_{n}F_\pm(x,v),
	%			%				\pa_{\tau_j\tau_j\tau_k}\pa_{n}F_\pm(x,R_xv) &= -\pa_{\tau_j\tau_j\tau_k}\pa_{n}F_\pm(x,v),\\
	%		\end{aligned}
	%	\end{equation}
	%	and
	%	\begin{equation}\label{27}\begin{aligned}
	%			\pa^2_{n}F_\pm(x,R_xv) &= \pa^2_{n}F_\pm(x,v),
	%			\\
	%			\pa_{\tau_j}\pa^2_{n}F_\pm(x,R_xv) &= \pa_{\tau_j}\pa^2_{n}F_\pm(x,v),
	%			%				\pa_{\tau_j\tau_k}\pa^2_{n}F_\pm(x,R_xv) &= \pa_{\tau_j\tau_k}\pa^2_{n}F_\pm(x,v),
	%		\end{aligned}
	%	\end{equation}
	%and 
	%\begin{equation}\label{28}
	%	\pa^3_{n}F_\pm(x,R_xv) = \pa^3_{n}F_\pm(x,v),
	%\end{equation}
	%for $j,k=1,2$. 
	%%for $j,k=1,2$. 
\end{Lem}
Note that $\al_i$ represent the normal direction on $\Gamma_i$ and $\al_j$ $(j\neq i)$ are the tangent direction on $\Gamma_i$. That is, $\pa_n = \pa_{x_i}$ or $-\pa_{x_i}$. 

\begin{proof}
	%	Fix $x\in\Gamma_i$.
	We first prove that for $0\le m\le 3$:
	\begin{align}\label{24aa}
		\pa^m_{x_i}F(x,v) = (-1)^m\pa_{x_i}^mF(x,R_xv).
	\end{align} 
	If $m=0$, noticing $R_x$ sends $\gamma_-$ to $\gamma_+$ and $R_xR_xv=v$, we can obtain \eqref{24aa} from \eqref{specular}. 
	%Since $\Gamma_i$ is a plane orthogonal to $n(x)$, applying tangent derivatives $\pa_{\tau_j}$ $(j=1,2)$ on \eqref{24}, we obtain \eqref{25}. 
	Next we claim that 
	\begin{equation}\label{2.10}
		%		\na_x\phi\cdot\na_vF(R_xv) =\na_x\phi\cdot\na_vF(v),\quad 
		\na_x\phi\cdot\na_vF_\pm(R_xv) =\na_x\phi\cdot\na_vF_\pm(v), \quad Q(f,g)(R_xv)=Q(f(R_xv),g(R_xv)). 
		%		\ \text{ on } n(x)\cdot v\neq 0.
	\end{equation}
	In fact, 
	noticing that for $x\in\Gamma_i$, $R_xv$ sends $v_i$ to $-v_i$ and preserves the other components, we have $\pa_{v_j}F_\pm(R_xv)=\pa_{v_j}F_\pm(v)$ on $\Gamma_i$ for $j=1,2,3$ such that $j\neq i$.
	Thus, we have 
	\begin{align*}
		\na_x\phi\cdot\na_vF_\pm(R_xv) = \sum_{j\neq i}\pa_{x_j}\phi\,\pa_{v_j}F_\pm(R_xv)
		= \na_x\phi\cdot\na_vF_\pm(v).
	\end{align*} 
	For the collision term, 
%	as in \cite[Lemma 3.1]{Deng2021},
	%
	%
	%	Indeed, by \eqref{24}, it suffices to show that 
	%	\begin{equation}\label{229}
	%		Q(g_1,g_2)(R_xv)=Q(g_1(R_xv),g_2(R_xv)),
	%	\end{equation} for any $g_1,g_2$.
	we apply the Carleman representation (cf. \cite{Alexandre2000} and \cite[Appendix]{Global2019}) to find that 
	\begin{align*}
		Q(f,g)(R_xv) &= \int_{\R^3_h}\int_{E_{0,h}}\tilde{b}(\al,h)\1_{|\al|\ge|h|}\frac{|\al+h|^{\gamma+1+2s}}{|h|^{3+2s}}\\&\qquad\qquad\qquad\times\big(f(R_xv+\al)g(R_xv-h)-f(R_xv+\al-h)g(R_xv)\big)\,d\al dh\\ &=Q(f(R_xv),g(R_xv)),
	\end{align*}
	where we apply rotation $R_x^{-1}$ on $(\al,h)$ and $E_{0,h}$ is the hyper-plane orthogonal to $h$ containing the origin. This completes the claim. 
	
	\smallskip
	For the normal derivatives, we need to apply equation \eqref{1} to obtain the compatible boundary condition. In fact, notice that  
	\begin{equation*}
		v\cdot\na_xF = v\cdot n(x)\pa_{n}F + v\cdot \tau_1(x)\pa_{\tau_1}F +v\cdot \tau_2(x)\pa_{\tau_2}F, 
	\end{equation*}  
	where $\{\tau_j\}_{j=1,2} = \{e_k\}_{k\neq i}$ are the tangent direction on $\Gamma_i$. 
	Then we can rewrite \eqref{1} as 
	\begin{align}\label{3.10}
		v\cdot n\pa_n F_\pm =-v\cdot \tau_1\pa_{\tau_1}F_\pm -v\cdot \tau_2\pa_{\tau_2}F_\pm\pm\na_x\phi\cdot\na_vF_\pm -\pa_t{F_\pm} +Q(F_\pm,F_\pm)+Q(F_\mp,F_\pm).  
	\end{align}	
	Using \eqref{2.10}, we know that the right hand side of \eqref{3.10} preserves its value under change of variable $v\mapsto R_xv$. Then we have  
	\begin{align*}
		R_xv\cdot n\pa_n F(R_xv) = v\cdot n\pa_nF(v).
	\end{align*}
	Noticing $R_xv\cdot n=-v\cdot n$, we can obtain \eqref{24aa} for the case $m=1$.

	\smallskip
	For the second order normal derivatives, we apply $\pa_n$ to \eqref{3.10} again to obtain 
	\begin{multline}\label{24a}
		v\cdot n\pa_n\pa_n F_\pm =-v\cdot \tau_1\pa_{\tau_1}\pa_nF_\pm -v\cdot \tau_2\pa_{\tau_2}\pa_nF_\pm \pm\pa_n\na_x\phi\cdot\na_vF_\pm \pm\na_x\phi\cdot\pa_n\na_vF_\pm\\  -\pa_t{\pa_nF} +Q(\pa_nF_\pm,F_\pm)+Q(F_\pm,\pa_nF_\pm)+Q(\pa_nF_\mp,F_\pm)+Q(F_\mp,\pa_nF_\pm). 
	\end{multline}
	%Using Neumann boundary condition \eqref{Neumann}, we know that $\pa_{x_ix_j}\phi=0$ on $\Gamma_i$ for $j\neq i$. 
	Then by \eqref{Neumann}, \eqref{24aa} with $m=0,1$ and \eqref{2.10}, we have  
	\begin{align*}
		%	\pa_ng_\pm &=
		%		Q(\pa_nF,F)(R_xv)+Q(F,\pa_nF)(R_xv) &= Q(\pa_nF(R_xv),F(R_xv))+Q(F(R_xv),\pa_nF(R_xv))\\ &= -Q(\pa_nF,F)(v)-Q(F,\pa_nF)(v). 
		&\notag\quad\,\pm\pa_n\na_x\phi\cdot\na_vF_\pm(R_xv)\pm\na_x\phi\cdot\pa_n\na_vF_\pm
		(R_xv)\\
		&\notag=\pm\pa_n\pa_{x_i}\phi\,\pa_{v_i}F_\pm (R_xv)\pm\sum_{j\neq i}\pa_{x_j}\phi\,\pa_n\pa_{v_j}F_\pm (R_xv)\\
		&= \pm\pa_n\pa_{x_i}\phi\,\pa_{v_i}F_\pm (v)\pm\sum_{j\neq i}\pa_{x_j}\phi\,\pa_n\pa_{v_j}F_\pm (v). 
	\end{align*}
Note that by taking tangent derivative on \eqref{Neumann}, we have $\pa_n\pa_{x_j}\phi=0$ on $\Gamma_i$ for $j\neq i$. 
	In view of \eqref{2.10}, one see that the right hand side of \eqref{24a} change the sign after taking change of variable $v\mapsto R_xv$.  
	%	\begin{align*}
	%		Q(\pa_nF,F)(R_xv) = -Q(\pa_nF,F)(v)\ \text{ and }\ Q(F,\pa_nF)(R_xv) = -Q(F,\pa_nF)(v).
	%	\end{align*}
	Thus, 
	%	by \eqref{24a}, 
	\begin{align*}
		R_xv\cdot n\pa_n\pa_n F(R_xv) = -v\cdot n\pa_n\pa_n F(v). 
	\end{align*}
	Noticing $R_xv\cdot n=-v\cdot n$, we can obtain \eqref{24aa} for the case $m=2$.
	The third normal derivative estimate can be derived similarly by taking normal derivative on \eqref{24a} again, then one can obtain \eqref{24aa} for $m=3$ and we omit the details for brevity. 
	%	 and then \eqref{27} by taking tangent derivatives.
	Since $\Gamma_i$ is a plane orthogonal to $n(x)$, applying tangent derivatives $\pa_{\tau_j}$ $(j=1,2)$ and time derivative $\pa_t$ on \eqref{24aa}, we obtain \eqref{24}. 
	This completes the proof of Lemma \ref{Lemspecular}.  
\end{proof}

As a corollary, 
%similar to \cite[Lemma 2.8]{Dengrarefaction},
we have the following compatible conditions for macroscopic and microscopic quantities. 
\begin{Lem}
	\label{LemMacro}
	For $i=2,3$ and any $x\in\Gamma_i$, we have 
	\begin{equation}
		\begin{aligned}\label{boundaryrhouth}
			&
%			\pa_{x_ix_ix_ix_ix_i}\big(\rho(x),\th(x),u_j(x)\big)=
			\pa_{x_ix_ix_i}\big(\rho(x),\th(x),u_j(x)\big)=
			\pa_{x_i}\big(\rho(x),\th(x),u_j(x)\big)=
%			0,\\
%			&
%			\pa_{x_ix_ix_ix_i}u_i(x)=
			\pa_{x_ix_i}u_i(x) = u_i(x) =0,
		\end{aligned}
	\end{equation}
	for $j=1,2,3$, $j\neq i$. 
	Consequently, for 
%	$\al=(0,\al_1,\al_2,\al_3)$ such that 
	$|\al|\le 3$, we have 
	\begin{align}\label{233a}
		\pa^{\al}M(R_xv) = (-1)^{|\al_i|}\pa^\al M(v),
		%		,\quad \pa_{x_i}M(R_xv) = -\pa_{x_i}M(v),\quad \pa_{x_ix_i}M(R_xv) =\pa_{x_ix_i}M(v),
	\end{align}
	and 
	\begin{align}\label{29}
		\pa^{\al}\g(R_xv) = (-1)^{|\al_i|}\pa^\al \g(v). 
		%		\g(R_xv) = \g(v),\quad \pa_{x_i}\g(R_xv) = -\pa_{x_i}\g(v),\quad \pa_{x_ix_i}\g(R_xv) =\pa_{x_ix_i}\g(v).
	\end{align}
	Moreover, we have 
	\begin{align}\label{boundphi}
		\pa_{x_ix_ix_i}\phi(x) =\pa_{x_ix_ix_ix_ix_i}\phi(x) = 0. 
	\end{align}
\end{Lem}
\begin{proof}
	Noticing $\pa_{x_i}$ is the normal derivative on $\Gamma_i$, we have from \eqref{24} that for $v\cdot n\neq 0$, 
	\begin{align}\label{210}
		\pa^{\al}F_\pm(R_xv) = (-1)^{|\al_i|}\pa^\al F_\pm(v). 
		%		F(R_xv) = F(v),\ 
		%		\pa_{x_i}F(R_xv) = -\pa_{x_i}F(v)\ 
		%		\text{ and } \pa_{x_ix_i}F(R_xv)=\pa_{x_ix_i}F(v).
	\end{align}		
	According to \eqref{rhouth} and using changing of variable $v\mapsto R_xv$, we have
	\begin{align}\label{235}
		\pa_{x_i}\rho(x) = \int_{\R^3}\pa_{x_i}F_1(v)\,dv = \int_{\R^3}\pa_{x_i}F_1(R_xv)\,dv = -\int_{\R^3}\pa_{x_i}F_1(v)\,dv =0. 
	\end{align}
	For $j=1,2,3$ with $j\neq i$, we have $(R_xv)_j=v_j$ and hence 
	\begin{align*}
		\pa_{x_i}u_j(x) 
%		= \rho^{-1}\int_{\R^3}v_j\pa_{x_i}F(v)\,dv 
		= \rho^{-1}\int_{\R^3}(R_xv)_j\pa_{x_i}F_1(R_xv)\,dv = -\rho^{-1}\int_{\R^3}v_j\pa_{x_i}F_1(v)\,dv = 0. 
	\end{align*}
	Similarly, using \eqref{210} and $(R_xv)_i=v_i$, we have 
	\begin{equation*}
		u_i(x) 
%		= \rho^{-1}\int_{\R^3}v_iF_1(v)\,dv 
		= \rho^{-1}\int_{\R^3}(R_xv)_iF_1(R_xv)\,dv = -\rho^{-1}\int_{\R^3}v_iF_1(v)\,dv = 0, 
	\end{equation*}
	and 
	\begin{equation*}
		\pa_{x_ix_i}u_i(x) = \rho^{-1}\int_{\R^3}(R_xv)_i\pa_{x_ix_i}F_1(R_xv)\,dv = -\rho^{-1}\int_{\R^3}v_i\pa_{x_ix_i}F_1(v)\,dv = 0. 
	\end{equation*}
	Consequently, for $\th = e$, we have $\pa_{x_i}\big(\rho(e+\frac{|u|^2}{2})\big) = \rho\pa_{x_i}\th$ on $\Gamma_i$ and hence 
	\begin{align*}
		\pa_{x_i}\th 
%		= \rho^{-1}\int_{\R^3}\frac{|v|^2}{2}\pa_{x_i}F(v)\,dv 
		= \rho^{-1}\int_{\R^3}\frac{|R_xv|^2}{2}\pa_{x_i}F_1(R_xv)\,dv = -\rho^{-1}\int_{\R^3}\frac{|v|^2}{2}\pa_{x_i}F_1(v)\,dv = 0. 
	\end{align*}
The boundary values for third-order derivative in \eqref{boundaryrhouth} can be obtain similarly. 
%Then we obtain \eqref{boundaryrhouth}. 
%	The boundary values for the higher-order normal derivative in \eqref{boundaryrhouth} can be derived similarly. 
	
	\smallskip
	Using boundary values \eqref{boundaryrhouth} and that $R_xv$ sends $v_i$ to $-v_i$ on $\Gamma_i$, we have 
	$
	M(R_xv) = M(v), 
	$
	\begin{equation*}
		\pa_{x_i}M(R_xv) = \frac{(R_xv)_i\pa_{x_i}u_i}{2R\th}M(R_xv) = - \pa_{x_i}M(v), 
	\end{equation*}
	and 
	\begin{equation*}
		\pa_{x_ix_i}M(R_xv) = \frac{|(R_xv)_i|^2|\pa_{x_i}u_i|^2}{(2R\th)^2}M(R_xv) =\pa_{x_ix_i}M(v). 
	\end{equation*}
Then we obtain \eqref{233a} by taking tangent derivatives. 
	The third-order normal derivative estimates can be derived similarly and we obtain \eqref{233a}. 
	%\begin{equation*}
	%\pa_{x_ix_ix_i}M(R_xv) = \frac{|(R_xv)_i|^2|\pa_{x_i}u_i|^2}{(2R\th)^2}M(R_xv)?? = -\pa_{x_ix_ix_i}M(v). 
	%\end{equation*}
	Applying $F_1=M+G$ and \eqref{210}, we know that 
	\begin{align}\label{229a}
		\pa^{\al}G(R_xv) = (-1)^{|\al_i|}\pa^\al G(v).
		%		G(R_xv) = G(v),\quad \pa_{x_i}G(R_xv) = -\pa_{x_i}G(v),\quad \pa_{x_ix_i}G(R_xv) =\pa_{x_ix_i}G(v).
	\end{align}
	Next we calculate $\ol G=L_M^{-1}P_1h$, where $h=v_1M\big\{\frac{|v-u|^2\bar\th_{x_1}}{2R\th^2}+\frac{(v-u)\cdot \bar{u}_{x_1}}{R\th}\big\}$. 
	Using \eqref{boundaryrhouth} and $\bar{u}_2=\bar{u}_3=0$, we have
	$h(R_xv)=h(v)$,
	\begin{equation*}
		\pa_{x_i}h(R_xv) = \frac{(R_xv)_i\pa_{x_i}u_i}{2R\th}h(R_xv)  +v_1M\frac{((R_xv)_i\pa_{x_i}u_i)\bar\th_{x_1}}{R\th^2} = -\pa_{x_i}h(v), 
	\end{equation*}
	and 
	\begin{equation*}
		\pa_{x_ix_i}h(R_xv) = \frac{|(R_xv)_i\pa_{x_i}u_i|^2}{(2R\th)^2}h(R_xv) +2v_1M\frac{((R_xv)_i\pa_{x_i}u_i)\bar\th_{x_1}}{R\th^2}\frac{(R_xv)_i\pa_{x_i}u_i}{2R\th}
		= \pa_{x_ix_i}h(v). 
	\end{equation*}
%These estimates yield 
	The third-order normal derivative estimates can be derived similarly and one can deduce 
	\begin{align}\label{313}
		\pa^{\al}h(R_xv) = (-1)^{|\al_i|}\pa^\al h(v). 
	\end{align} 
	%\begin{equation*}
	%\pa_{x_ix_ix_i}h(R_xv) = \frac{|(R_xv)_i\pa_{x_i}u_i|^2}{(2R\th)^2}h(R_xv) +2v_1M\frac{((R_xv)_i\pa_{x_i}u_i)\bar\th_{x_1}}{R\th^2}\frac{(R_xv)_i\pa_{x_i}u_i}{2R\th}
	%= \pa_{x_ix_ix_i}h(v)??. 
	%\end{equation*}
	Using the change of variable $v\mapsto R_xv$, similar to \eqref{235}, 
	%	 as \eqref{235},
	one can obtain  
	\begin{align*}
		(h,\frac{\chi_i}{M})_{L^2_v} &= (\pa_{x_i}h,\frac{\chi_j}{M})_{L^2_v} = (\pa_{x_ix_i}h,\frac{\chi_i}{M})_{L^2_v} 
%		= (\pa_{x_ix_ix_i}h,\frac{\chi_j}{M})_{L^2_v}
%		\\& = (\pa_{x_ix_ix_ix_i}h,\frac{\chi_i}{M})_{L^2_v}= (\pa_{x_ix_ix_ix_ix_i}h,\frac{\chi_j}{M})_{L^2_v} = 
		=0,
	\end{align*} for $1\le j\le 4$ such that $j\neq i$. Then we have 
	\begin{align}\label{P0h}
		\pa^{\al}P_0h(R_xv) = (-1)^{|\al_i|}\pa^\al P_0 h(v). 
		%		P_0h(R_xv) = P_0h(v),\quad \pa_{x_i}P_0h(R_xv) = -\pa_{x_i}P_0h(v),\quad \pa_{x_ix_i}P_0h(R_xv) = \pa_{x_ix_i}P_0h(v). 
	\end{align}
	Thus, using $P_1=I-P_0$ and $L_M\ol G = P_1h$, we have from \eqref{313} and \eqref{P0h} that 
	\begin{align*}
		\pa^{\al}L_M\G(R_xv) = (-1)^{|\al_i|}\pa^\al L_M\G(v). 
		%		L_M\G(R_xv) = L_M\G(v),\quad \pa_{x_i}L_M\G(R_xv) = -\pa_{x_i}L_M\G(v),\quad \pa_{x_ix_i}L_M\G(R_xv) = \pa_{x_ix_i}L_M\G(v). 
	\end{align*}
	Applying the second identity of \eqref{2.10}, we know that 
	\begin{align*}
		L_M(\pa^{\al}\G(R_xv)) = L_M(\pa^{\al}\G)(R_xv) = (-1)^{|\al_i|}L_M(\pa^{\al}\G)(v).
	\end{align*}
	Applying $L_M^{-1}$, we have 
	\begin{align*}
		\pa^{\al}\G(R_xv)=(-1)^{|\al_i|}\pa^{\al}\G(v). 
	\end{align*} 
	%Similar arguments can be applied to $\pa_{x_i}L_M\G$ , $\pa_{x_ix_i}L_M\G$ and we can obtain $\pa_{x_i}\G(R_xv)=-\pa_{x_i}\G(v)$ and $\pa_{x_ix_i}\G(R_xv)=\pa_{x_ix_i}\G(v)$. 
	Noticing $\sqrt\mu\g=G-\G$ and applying \eqref{229a}, one can obtain \eqref{29}. 
	For \eqref{boundphi}, applying $\pa_{x_i}$ to the third equation of \eqref{1}, we know that 
	\begin{align}\label{345}
		\pa_{x_i}\Delta_x\phi = \int_{\R^3}\pa_{x_i}(F_+-F_-)\,dv = 0,
	\end{align}
	by change of variable $v\mapsto R_xv$. 
	It follows from \eqref{Neumann} that $\pa_{x_i}\phi=0$. Taking tangent derivatives, we have $\pa_{x_ix_jx_j}\phi=0$ for $j\neq i$. Together with \eqref{345}, we have $\pa_{x_ix_ix_i}\phi=0$. 
%	The fifth order boundary value can be deduced similarly by applying $\pa_{x_ix_i}$ on $\eqref{345}$ and using \eqref{24}. 
	This completes the proof of Lemma \ref{LemMacro}.		
\end{proof}

\subsection{Estimate on $L^{-1}_M$}
In this section, we calculate the estimate on $L^{-1}_M$.

\begin{Lem}\label{LemLM}
	Assume $\gamma>-3$ and $\gamma+2s>-\frac{3}{2}$. Let $l\ge 0$ and $w_l=\<v\>^l$. Then 
	\begin{align}\label{LM1a}
		|w_lM^{-1/2}L_M^{-1}f|_{H^s_{\gamma/2}}+|w_{l}\mu^{-1/2}L_M^{-1}f|_{H^s_{\gamma/2}}\lesssim |w_lM^{-1/2}f|_{H^{-s}_{-\gamma/2}}\lesssim |w_{l}\mu^{-\frac{1}{2}}f|_{H^{-s}_{-\gamma/2}}. 
	\end{align}
Moreover, 
for any $i,j=1,2,3$ and $l\ge0$, we have 
\begin{align}
	\label{decayAB}
		\big|w_l\mu^{-1/2}A_j\big(\frac{v-u}{\sqrt{R\th}}\big)\big|_{H^s_{\gamma/2}}^2+\big|w_l\mu^{-1/2}B_{ij}\big(\frac{v-u}{\sqrt{R\th}}\big)\big|_{H^s_{\gamma/2}}^2\le C<\infty. 
\end{align}
%	\begin{align}\label{LM1b}
%		|w_lM^{-1/2}L_M^{-1}f|_{H^s_{\gamma/2}}^2\lesssim |w_l\mu^{-1/2}f|_{H^{-s}_{-\gamma/2}}.
%	\end{align}
\end{Lem}
\begin{proof}
We consider operator $M^{-\frac{1}{2}}L_M\big(M^{\frac{1}{2}}h\big)$, whose null space is given by $$\text{Span}\{M^{\frac{1}{2}},v_iM^{1/2}(i=1,2,3),|v|^2M^{\frac{1}{2}}\}.$$
%	
%Define 
%$\mathbf{L}_{\sqrt{M}}h=M^{-1/2}Q(M,M^{1/2}h) + M^{-1/2}Q(M^{1/2}h,M)$ and its null space by  
Recalling \eqref{mathfrackN}, we denote $P_3$ to be the projection from $L^2_v$ to 
%$\ker\mathbf{L}_{\sqrt{M}}$
this null space as 
\begin{align*}
	P_3 f = \sum^4_{i=0}\big(f,\frac{\chi_i}{M^{\frac{1}{2}}}\big)_{L^2_v}\,\frac{\chi_i}{M^{\frac{1}{2}}} = M^{-\frac{1}{2}}P_0(M^{\frac{1}{2}}f). 
\end{align*}
Denote $h$ such that $h=\{I-P_3\}h$. 
Then it follows from \cite[Theorem 1.1]{Mouhot2006a} that 
\begin{align*}
	-\big(M^{-\frac{1}{2}}L_M\big(M^{\frac{1}{2}}h\big),h\big)_{L^2_v}\gtrsim |\<v\>^{\frac{\gamma}{2}}\{I-P_3\}h|_{L^2_v}^2. 
\end{align*}
Following similar arguments as in \cite[Lemma 2.14, Proposition 2.16]{Alexandre2012}, we can obtain the coercive estimates in the form of local Maxwellian:
\begin{align*}
	-\big(M^{-\frac{1}{2}}L_M\big(M^{\frac{1}{2}}h\big),h\big)_{L^2_v}\gtrsim |h|_{L^2_D}^2 - C|h|_{L^2_{\gamma/2}}^2. 
\end{align*}
Combining the above two estimates, we have 
% and noticing $\big(M^{-\frac{1}{2}}L\big(M^{\frac{1}{2}}h\big),h\big)_{L^2_v}=\big(M^{-\frac{1}{2}}L\big(M^{\frac{1}{2}}\{I-P_3\}h\big),\{I-P_3\}h\big)_{L^2_v}$, we have 
\begin{align}\label{317}
	-\big(M^{-\frac{1}{2}}L_M\big(M^{\frac{1}{2}}h\big),h\big)_{L^2_v}\gtrsim |\{I-P_3\}h|_{L^2_D}^2. 
\end{align}
%Write $g=M^{1/2}h$, then 
Using similar calculations \cite[Lemma 2.6]{Gressman2011}, one can obtain the estimate for $h=\{I-P_3\}h$:
\begin{align}
	\label{317a}
	-\big(w^2_lM^{-\frac{1}{2}}L\big(M^{\frac{1}{2}}h\big),h\big)_{L^2_v}\gtrsim |w_lh|_{L^2_D}^2 - C|h|_{L^2_{\gamma/2}}^2. 
%	,\\
%	-\big(w^2_{l-|\beta|}\pa_\beta\big(M^{-\frac{1}{2}}L_M\big(M^{\frac{1}{2}}h\big)\big),\pa_\beta h\big)_{L^2_v}\gtrsim |w_{l-|\beta|}\pa_\beta h|_{L^2_D}^2 -\eta\sum_{|\beta_1|\le|\beta|}|w_{l-|\beta_1|}\pa_{\beta_1}h|_{L^2_D} - C_\eta|h|_{L^2_{\gamma/2}}^2,  \label{317b}
\end{align}
%for any $\eta>0$. 
For any $K>0$, taking linear combination $\eqref{317}+\kappa\times\eqref{317a}$ with small enough $\kappa>0$ and denoting $g=P_1M^{\frac{1}{2}}h=M^{\frac{1}{2}}\{I-P_3\}h$, we have 
\begin{align}\label{439}
	-\kappa(w^2_{l}M^{-\frac{1}{2}}L_Mg,M^{-\frac{1}{2}}g\big)_{L^2_v}-(L_Mg,M^{-1}g)_{L^2_v}
%	\\
%	-\kappa(w^2_{l}\big(M^{-\frac{1}{2}}L_MM^{\frac{1}{2}}h\big), h\big)_{L^2_v}-(M^{-\frac{1}{2}}L_MM^{\frac{1}{2}}h,h)_{L^2_v}\\
%	 = -\big(M^{-\frac{1}{2}}L\big(M^{\frac{1}{2}}\{I-P_3\}h\big),\{I-P_3\}h\big)_{L^2_v}\\
	\gtrsim |w_{l} h|_{L^2_D}^2 = |w_{l} M^{-\frac{1}{2}}g|_{L^2_D}^2.
\end{align}
Then \eqref{439} and \eqref{123} imply that for $g=P_1g$, 
\begin{align*}
	|w_{l} M^{-\frac{1}{2}}g|_{H^s_{\gamma/2}}^2
	&\lesssim |w_{l}M^{-\frac{1}{2}}L_Mg|_{H^{-s}_{-\gamma/2}}|w_{l}M^{-\frac{1}{2}}g|_{H^s_{\gamma/2}}, 
%	&\lesssim |w_l\mu^{-\frac{1}{2}}L_Mg|_{H^{-s}_{-\gamma/2}}|w_lM^{-1}\mu^{\frac{1}{2}}g|_{L^2_D}
%	\\&\lesssim |w_l\mu^{-\frac{1}{2}}L_Mg|_{H^{-s}_{-\gamma/2}}|w_lM^{-\frac{1}{2}}g|_{L^2_D}, 
\end{align*}
and thus, for $g=L_M^{-1}f$, 
\begin{align*}
	|w_{l} M^{-\frac{1}{2}}L_M^{-1}f|_{H^s_{\gamma/2}}
	\lesssim |w_{l}M^{-\frac{1}{2}}f|_{H^{-s}_{-\gamma/2}}.
\end{align*}
This gives the first part of \eqref{LM1a}.  
%Consider the terms including $h$ in \eqref{439}, we have 
%\begin{align*}
%	|w_{l} h|_{H^s_{\gamma/2}}\lesssim  
%	|w_{l}\big(M^{-\frac{1}{2}}L_MM^{\frac{1}{2}}h\big)|_{H^{-s}_{-\gamma/2}}. 
%\end{align*}
%Choosing $l\ge K>0$ large enough, we have from Sobolev embedding that 
Using \eqref{eta}, we have $R\th>1$ and hence, $M^{-\frac{1}{2}}\mu^{\frac{1}{2}}\le C$. It follows that 
\begin{align*}
	|w_{l} \mu^{-\frac{1}{2}}L_M^{-1}f|_{H^s_{\gamma/2}}
	&\lesssim |w_{l} M^{-\frac{1}{2}}L_M^{-1}f|_{H^s_{\gamma/2}} + \big|w_{l} \frac{M^{\frac{1}{2}}-\mu^{\frac{1}{2}}}{\mu^{\frac{1}{2}}M^{\frac{1}{2}}}L_M^{-1}f\big|_{H^s_{\gamma/2}}\\
	&\lesssim |w_{l}M^{-\frac{1}{2}}f|_{H^{-s}_{-\gamma/2}} + \eta_0|w_{l} \mu^{-\frac{1}{2}}L_M^{-1}f|_{H^s_{\gamma/2}}.
\end{align*}
Choosing $\eta_0>0$ in \eqref{eta} sufficiently small, we obtain 
\begin{align*}
		|w_{l}\mu^{-\frac{1}{2}}L_M^{-1}f|_{H^s_{\gamma/2}}\lesssim |w_{l}M^{-\frac{1}{2}}f|_{H^{-s}_{-\gamma/2}}\lesssim |w_{l}\mu^{-\frac{1}{2}}f|_{H^{-s}_{-\gamma/2}}. 
\end{align*}
%\begin{align*}
%	|h(v)|\lesssim |w_{l}\big(M^{-\frac{1}{2}}L_MM^{\frac{1}{2}}h\big)|_{H^{-s}_{-\gamma/2}}.
%\end{align*}
As a consequence, we can obtain the estimate on Burnett functions $A_j$ and $B_{ij}$ from \eqref{21}: 
\begin{align*}
	\big|w_l\mu^{-1/2}A_j\big(\frac{v-u}{\sqrt{R\th}}\big)\big|_{H^s_{\gamma/2}}^2+\big|w_l\mu^{-1/2}B_{ij}\big(\frac{v-u}{\sqrt{R\th}}\big)\big|_{H^s_{\gamma/2}}^2\le C<\infty. 
\end{align*} 
This completes the proof of Lemma \ref{LemLM}. 
\end{proof}

\subsection{Estimates on Collision Terms}
In this section, we will give some estimates on collision terms.   
\begin{Lem}\label{Lem24}
	If we choose $\eta_0>0$ in \eqref{eta} and $\delta>0$ in \eqref{delta} small enough, then for $|\al|\le 3$ and any function $g,h$, we have 
	\begin{multline}\label{220b}
		\big|\big(\pa^\al \Gamma\big(g,\frac{M-\mu}{\sqrt{\mu}}\big),w^2(\al) h\big)_{L^2_{x,v}}\big|+\big|\big(\pa^\al \Gamma\big(\frac{M-\mu}{\sqrt{\mu}},g\big),w^2(\al) h\big)_{L^2_{x,v}}\big|\\
		\lesssim \eta_0\|w(\al)\pa^{\al}g\|_{L^2_xL^2_D}\|w(\al)h\|_{L^2_xL^2_D}\\ + \big(\eta_0+\delta^{\frac{1}{3}}+\sqrt{\E_k(t)}\big)\sum_{|\al_1|\le3}\|w(\al_1)\pa^{\al_1}g\|_{L^2_xL^2_D}\|w(\al)h\|_{L^2_xL^2_D},
	\end{multline}
	 where $\E_k(t)$ is given by \eqref{E}. 
\end{Lem}
\begin{proof}
	For brevity, we denote $\nu(v)=\min\{1,\<v\>^{\frac{\gamma+2s}{2}}\}$. From \eqref{Gamma} and \eqref{Gamma1}, 
	%we have 
	%	\begin{multline*}\notag
	%		\big|\big(\pa^\al \Gamma\big(\g,\frac{M-\mu}{\sqrt{\mu}}\big),w_l^2 h\big)_{L^2_v}\big|+|(\pa^\al \Gamma(\frac{M-\mu}{\sqrt{\mu}},\g),w_l^2 h)_{L^2_v}|\\\lesssim \sum_{\alpha_1\le\al}\Big(|\nu(v)w_l\pa^{\al-\al_1}\g|_{L^2_v}\big|w_l\pa^{\al_1}\big(\frac{M-\mu}{\sqrt{\mu}}\big)\big|_{L^2_D}+\big|\nu(v)w_l\pa^{\al_1}\big(\frac{M-\mu}{\sqrt{\mu}}\big)\big|_{L^2_v}|w_l\pa^{\al-\al_1}f|_{L^2_D}\Big)|w_l h|_{L^2_D}.
	%	\end{multline*}
	%Taking integral over $x\in\Omega$, 
	we have 
	\begin{align}\label{212}\notag
		&\quad\,\big|\big(\pa^\al \Gamma\big(g,\frac{M-\mu}{\sqrt{\mu}}\big),w^2(\al) h\big)_{L^2_{x,v}}\big|+\big|\big(\pa^\al \Gamma\big(\frac{M-\mu}{\sqrt{\mu}},g\big),w^2(\al) h\big)_{L^2_{x,v}}\big|\\
		&\lesssim\notag \sum_{\alpha_1\le\al}\Big(\int_\Omega|\nu(v)w(\al)\pa^{\al-\al_1}g|^2_{L^2_v}\big|w(\al)\pa^{\al_1}(\frac{M-\mu}{\sqrt{\mu}})\big|^2_{L^2_D}\\
		&\qquad\qquad+\big|\nu(v)w(\al)\pa^{\al_1}\big(\frac{M-\mu}{\sqrt{\mu}}\big)\big|^2_{L^2_v}|w(\al)\pa^{\al-\al_1}g|^2_{L^2_D}\,dx\Big)^{\frac{1}{2}}\|w(\al) h\|_{L^2_xL^2_D}.
	\end{align}
	For the term with $\frac{M-\mu}{\sqrt{\mu}}$, 
	%we have from 
	%\begin{align*}
	%	|w_l\pa^{\al_1}(\frac{M-\mu}{\sqrt{\mu}})|_{L^2_D}
	%\end{align*}
	it follows from \eqref{123} that 
	\begin{align}\label{325a}
		|w(\al)\pa^{\al_1}(\frac{M-\mu}{\sqrt{\mu}})|^2_{L^2_D}+|\nu(v)w(\al)\pa^{\al_1}(\frac{M-\mu}{\sqrt{\mu}})|^2_{L^2_v}\lesssim \sum_{|\beta|\le 2}\int_{\R^3}\<v\>^b|\pa^{\al_1}_{\beta}(\frac{M-\mu}{\sqrt{\mu}})|^2\,dv, 
	\end{align}
	for some large $b>0$. 
	We first consider the case $|\al_1|=0$. 
	From \eqref{eta}, we have $1<R\th<2$ and thus, $\|\<v\>^be^{-\frac{|v-u|^2}{2R\th}+\frac{|v|^2}{4}}\|_{L^2_v}\le C<\infty$. 
	Then the integration on the right hand side of \eqref{325a} is finite and hence, there exists $R>0$ such that 
	\begin{align*}
		\int_{|v|\ge R}\<v\>^b|\pa_{\beta}(\frac{M-\mu}{\sqrt{\mu}})|^2\,dv\le \eta_0^2.
	\end{align*}
	For the low velocity part, we have 
	\begin{align*}
		\int_{|v|\le R}\mu^{-\ve_0}\big|\pa_{\beta}\big(\frac{M-\mu}{\sqrt{\mu}}\big)\big|^2\,dv
		&\lesssim |(\rho-1,u,\theta-\frac{3}{2})|^2.
	\end{align*}
	Combining the above estimates, we have that when $|\al_1|=0$, 
	\begin{align}\label{33}
		\|w(\al)(\frac{M-\mu}{\sqrt{\mu}})\|^2_{L^\infty_xL^2_D}+\|\nu(v)w(\al)(\frac{M-\mu}{\sqrt{\mu}})\|^2_{L^\infty_xL^2_v}\lesssim\eta_0^2.
	\end{align}
For the case $|\al_1|\ge 1$, we have $\pa^{\al_1}(M-\mu)=\pa^{\al_1}M$ and hence, 
\begin{align*}
	\sum_{|\beta|\le 2}|\<v\>^b\pa^{\al_1}_{\beta}(\frac{M-\mu}{\sqrt{\mu}})|^2_{L^2_v}\lesssim \sum_{\al'_1+\cdots+\al'_K=\al_1}|\pa^{\al'_1}(\rho,u,\th)|^{\text{sgn}(\al'_1)}\cdots|\pa^{\al'_K}(\rho,u,\th)|^{\text{sgn}(\al'_K)}, 
\end{align*}
where $K=|\al_1|$, $\text{sgn}(\al'_j)$ equal to $1$ if $|\al'_j|\ge 1$ and equal to $0$ if $|\al'_j|=0$. 
Thus, 
\begin{multline*}
	\sum_{|\beta|\le 2}\|\<v\>^b\pa^{\al_1}_{\beta}(\frac{M-\mu}{\sqrt\mu})\|^2_{L^2_xL^2_{v}}\lesssim
	\|\pa^{\al_1}(\rho,u,\th)\|_{L^2_x}^2 
	+\sum_{\substack{\al'_1+\al'_2=\al_1\\|\al'_1|\ge 1,\,|\al'_2|\ge 1}}\|\pa^{\al'_1}(\rho,u,\th)\|_{L^6_x}^2\|\pa^{\al'_2}(\rho,u,\th)\|^2_{L^3_x}\\
	\lesssim \|\pa^{\al_1}(\wt\rho,\wt u,\wt\th)\|_{L^2_x}^2 + \E_k(t)\min\{\D_k(t),\E_k(t)\}+\delta^{\frac{2}{3}}(1+t)^{-\frac{4}{3}},
%	\left\{\begin{aligned}
%		&\min\{\delta^{\frac{2}{3}}(1+t)^{-\frac{4}{3}},(1+t)^{-2}\}, \text{ if }|\al|=2,\\
%		&\delta(1+t)^{-1},\text{ if }|\al|=1.
%	\end{aligned}\right.
\end{multline*}
where we decompose $(\rho,u,\th)=(\wt\rho,\wt u,\wt\th)+(\bar\rho,\bar u,\bar\th)$ and use Lemma \ref{Lem21} and \eqref{LL31} to estimate the second part. 
%where we use $\|\pa^{\al'_j}(\rho,u,\th)\|_{L^\infty_x}\le \sqrt{\E_k(t)}+\delta^{\frac{1}{2}}\le 1$ for $1\le|\al'_j|\le |\al_1|-2$, by choosing $\ve$ in \eqref{priori2} and $\delta$ in \eqref{delta} small enough.
 This implies 
\begin{multline}\label{333}
	\|w(\al)\pa^{\al_1}(\frac{M-\mu}{\sqrt{\mu}})\|^2_{L^2_xL^2_D}+\|\nu(v)w(\al)\pa^{\al_1}(\frac{M-\mu}{\sqrt{\mu}})\|^2_{L^2_xL^2_v}\\
	\lesssim \|\pa^{\al_1}(\wt\rho,\wt u,\wt\th)\|_{L^2_x}^2 + \E_k(t)\min\{\D_k(t),\E_k(t)\}+\delta^{\frac{2}{3}}(1+t)^{-\frac{4}{3}}. 
%	\left\{\begin{aligned}
%		&\min\{\delta^{\frac{2}{3}}(1+t)^{-\frac{4}{3}},(1+t)^{-2}\}, \text{ if }|\al|=2,\\
%		&\delta(1+t)^{-1},\text{ if }|\al|=1.
%	\end{aligned}\right.
\end{multline}
	%Applying 
	%\begin{multline*}
	%	\quad\, \sum_{\substack{1\le|\al_2|\le|\al_1|-2}} \||\pa^{\al_2}(\rho,u,\th)|^{|\al_1-\al_2|+1}\|_{L^\infty_x}+\sum_{\substack{1\le|\al_2|=|\al_1|-1}}\||\pa^{\al_2}(\rho,u,\th)|^{2}\|_{L^6_x}\\+\sum_{\substack{1\le|\al_2|=|\al_1|}} \|\pa^{\al_2}(\rho,u,\th)\|_{L^2_x}
	%	\le\sum_{|\al_3|\le K}\|\pa^{\al_3}(\rho,u,\th)\|_{L^2_x}^{|\al_1-\al_2|+1}
	%	\le \sum_{j=1}^K\E^{\frac{j}{2}}(t)+\delta^{\frac{1}{2}}.
	%\end{multline*}
Now we return to \eqref{212}. Applying $L^2_x-L^\infty_x$ and $L^3-L^6$ H\"{o}lder's inequality and Sobolev embedding in Lemma \ref{Lemembedd}, 
%	$\|f\|_{L^3_x(\Omega)}\lesssim\|f\|_{H^1_x(\Omega)}$, $\|f\|_{L^6_x(\Omega)}\lesssim\|\na_xf\|_{L^2_x(\Omega)}$ and
%	$\|f\|_{L^\infty_x(\Omega)}\lesssim\|f\|_{H^2_x(\Omega)}$
%	from \cite[Section V and (V.21)]{Adams2003}, 
	we have 
	\begin{align}\label{217}\notag
		&\quad\,\Big(\int_\Omega|\nu(v)w(\al)\pa^{\al-\al_1}g|^2_{L^2_v}|w(\al)\pa^{\al_1}(\frac{M-\mu}{\sqrt{\mu}})|^2_{L^2_D}\,dx\Big)^{\frac{1}{2}}\\
%		&\lesssim \notag \eta_0^2\|\nu(v)w(\al)\pa^{\al-\al_1}_{\beta_2}g\|_{L^2_xL^2_v}+\Big(\int_\Omega|\nu(v)w(\al)\pa^{\al-\al_1}_{\beta_2}g|^2_{L^2_v}\sum_{\substack{1\le|\al_2|\le K\\\al_2\le\al_1}}|\pa^{\al_2}(\rho,u,\theta)|^{2(|\al_1-\al_2|+1)}\,dx\Big)^{\frac{1}{2}}\\
		&\notag\lesssim \sum_{|\al_1|=0}\|\nu(v)w(\al)\pa^{\al}g\|_{L^2_xL^2_v}\|w(\al)(\frac{M-\mu}{\sqrt{\mu}})\|_{L^\infty_xL^2_D}\\
		%	+ \sum_{|\al_1+\beta_1|=0}\|\nu(v)w_l\pa^{\al-\al_1}_{\beta_2}\g\|^2_{L^2_xL^2_v}
		%	\Big(\eta^2_0+\Big)
		%	\|w_l\pa^{\al_1}(\frac{M-\mu}{\sqrt{\mu}})\|^2_{L^\infty_xL^2_D}
		%	\big)\\
		&\qquad\notag+\sum_{|\al_1|=1}\|\nu(v)w(\al)\pa^{\al-\al_1}g\|_{L^6_xL^2_v}\|w(\al)\pa^{\al_1}(\frac{M-\mu}{\sqrt{\mu}})\|_{L^3_xL^2_D}\\
		%	\|w_l\pa^{\al_1}(\frac{M-\mu}{\sqrt{\mu}})\|_{L^6_xL^2_D}\\
		&\notag\qquad+\sum_{2\le|\al_1|\le 3} \|\nu(v)w(\al)\pa^{\al-\al_1}g\|_{L^\infty_xL^2_v}\|w(\al)\pa^{\al_1}(\frac{M-\mu}{\sqrt{\mu}})\|_{L^2_xL^2_D}\\
		&\lesssim \eta_0\|\nu(v)w(\al)\pa^{\al}g\|_{L^2_xL^2_v} + \big(\eta_0+\delta^{\frac{1}{3}}+\sqrt{\E_k(t)}\big)\sum_{|\al_1|\le3}\|\nu(v)w(\al_1)\pa^{\al_1}g\|_{L^2_xL^2_v}.  
	\end{align}	
%	Here the last inequality follows from $(\rho,u,\th)=(\wt\rho,\wt u,\wt\th)+(\bar\rho,\bar u,\bar\th)$ and Lemma \ref{Lem21}. 
	%where we 
	%We also used 
Note that in the summation in \eqref{217}, we always have $|\al|\ge |\al_1|$ and the weight index $w(\al_1)$ matches derivative index $\pa^{\al_1}$. Similarly, 
	\begin{align*}
		&\Big(\int_\Omega|\nu(v)w(\al)\pa^{\al_1}(\frac{M-\mu}{\sqrt{\mu}})|^2_{L^2_v}|w(\al)\pa^{\al-\al_1}g|^2_{L^2_D}\,dx\Big)^{\frac{1}{2}}\\
		&\quad\lesssim  \eta_0\|w(\al)\pa^{\al}g\|_{L^2_xL^2_D} + \big(\eta_0+\delta^{\frac{1}{3}}+\sqrt{\E_k(t)}\big)\sum_{|\al_1|\le3}\|w(\al_1)\pa^{\al_1}g\|_{L^2_xL^2_D}. 
	\end{align*}
	Therefore, noticing $|\nu(v)(\cdot)|_{L^2_v}\lesssim |\cdot|_{L^2_D}$ for both hard and soft potential, \eqref{212} becomes 
	\begin{align*}
		&\big|\big(\pa^\al \Gamma\big(g,\frac{M-\mu}{\sqrt{\mu}}\big),w^2(\al) h\big)_{L^2_{x,v}}\big|+\big|\big(\pa^\al \Gamma\big(\frac{M-\mu}{\sqrt{\mu}},g\big),w^2(\al) h\big)_{L^2_{x,v}}\big|\\
		&\quad\lesssim\eta_0\|w(\al)\pa^{\al}g\|_{L^2_xL^2_D}\|w(\al)h\|_{L^2_xL^2_D}\\&\qquad\quad + \big(\eta_0+\delta^{\frac{1}{3}}+\sqrt{\E_k(t)}\big)\sum_{|\al_1|\le3}\|w(\al_1)\pa^{\al_1}g\|_{L^2_xL^2_D}\|w(\al)h\|_{L^2_xL^2_D}. 
%		 \big(\eta_0+\delta^{\frac{1}{2}}+\sum_{j=1}^K\E^{\frac{j}{2}}(t)\big)\sum_{|\al_2|+|\beta_2|\le K,\,\beta_2\le \beta}\big(\|\nu(v)w(\al)\pa^{\al_2}g\|_{L^2_xL^2_v}+\|w(\al)\pa^{\al_2}g\|_{L^2_xL^2_D}\big)\|w(\al)h\|_{L^2_xL^2_D}\\
%		&\lesssim C_\eta\big(\eta^2_0+\delta+\sum_{j=1}^K\E^{j}(t)\big)\sum_{|\al_2|+|\beta_2|\le K,\,\beta_2\le \beta}\|w(\al)\pa^{\al_2}g\|^2_{L^2_xL^2_D} + \eta\|w(\al)h\|^2_{L^2_xL^2_D},
	\end{align*}
%	for any $\eta>0$.
	 This completes the proof of Lemma \ref{Lem24}.	
\end{proof}

Next we consider the estimates on the nonlinear term $\Gamma(\mu^{-\frac{1}{2}}G,\mu^{-\frac{1}{2}}G)$ and $\Gamma(\mu^{-\frac{1}{2}}G,\f)$. 
\begin{Lem}\label{Lem25}
	Assume the same conditions as in Lemma \ref{Lem24}. Then 
	%		\begin{multline}\label{220a}
	%			\big|\big(\pa^\al\Gamma\big(\frac{G}{\sqrt\mu},\frac{G}{\sqrt\mu}\big), w_l^2h\big)_{L^2_{x,v}}\big|\lesssim \eta\|w_lh\|^2_{L^2_xL^2_D} +C_\eta\delta^2(1+t)^{-2}+C_\eta\Big(\delta^2+\sum_{j=1}^{2K-1}\E^{j}(t)\Big)\D(t),
	%		\end{multline}??how to use $\D$ for weight??
%	\begin{align}\label{220a}
%		\big|\big(\Gamma(\mu^{-\frac{1}{2}}G,\mu^{-\frac{1}{2}}G),\g\big)_{L^2_{x,v}}\big|+\big|\big(\Gamma(\mu^{-\frac{1}{2}}G,\f),\f\big)_{L^2_{x,v}}\big|\lesssim \big(\delta^{\frac{1}{2}}+\sqrt{\E_k(t)}\big)\D_k(t) + \delta(1+t)^{-2}, 
%	\end{align}
%		and for $1\le|\al|\le 2$, 
	\begin{multline}
		\label{220}
		\big|\big(\pa^\al\Gamma(\mu^{-\frac{1}{2}}G,\mu^{-\frac{1}{2}}G),w^2(\al)\pa^\al\g\big)_{L^2_{x,v}}\big|+\big|\big(\pa^\al\Gamma(\mu^{-\frac{1}{2}}G,\f),w^2(\al)\pa^\al\f\big)_{L^2_{x,v}}\big|\\
		\lesssim \big(\delta^{\frac{1}{2}}+\sqrt{\E_k(t)}\big)\D_k(t)+\delta(1+t)^{-2}, 
	\end{multline}
	where $\E_k(t)$ and $\D_k(t)$ are given by \eqref{E} and \eqref{D} respectively. 
\end{Lem}

\begin{proof}
%	Recalling $G=\overline{G}+\mu^{1/2}\g$, we have 
%	\begin{align}\label{215}
%		\Gamma(\mu^{-\frac{1}{2}}G,\mu^{-\frac{1}{2}}G) = \Gamma(\mu^{-\frac{1}{2}}\ol G,\mu^{-\frac{1}{2}}\ol G) + \Gamma(\mu^{-\frac{1}{2}}\ol G,\g) + \Gamma(\g,\mu^{-\frac{1}{2}}\ol G) + \Gamma(\g,\g).
%	\end{align}
	Denote $\nu(v)=\min\{1,\<v\>^{\frac{\gamma+2s}{2}}\}$ for brevity. Then for any function $f,g\in H^2_xL^2_D$, we have from \eqref{Gamma} and \eqref{Gamma1} that 
	\begin{multline}\label{216}
		|(\pa^\al \Gamma(f,g), w^2(\al)h)_{L^2_{x,v}}|
		\lesssim \sum_{\al_1\le\al}\int_{\Omega}|\nu(v)\pa^{\al_1}f|_{L^2_v}|w(\al)\pa^{\al-\al_1}g|_{L^2_D}|w(\al)h|_{L^2_D}\,dx\\ +\sum_{\al_1\le\al}\int_{\Omega}|\nu(v)w(\al)\pa^{\al_1}f|_{L^2_v}|\pa^{\al-\al_1}g|_{L^2_D}|w(\al)h|_{L^2_D}\,dx. 
	\end{multline}
The first right hand term of \eqref{216} can be estimated as 
%\begin{align*}
%\int_{\Omega}|\nu(v)f|_{L^2_v}|w(0)g|_{L^2_D}|w(0)h|_{L^2_D}\,dx
%&\lesssim \|\nu(v)f\|_{L^\infty_xL^2_v}\|w(0)g\|_{L^2_xL^2_D}\|w(0)h\|_{L^2_xL^2_D}\\
%&\lesssim \|\nu(v)w(2)f\|_{H^2_xL^2_v}\|w(0)g\|_{L^2_xL^2_D}\|w(0)h\|_{L^2_xL^2_D}. 
%\end{align*}
%when $|\al|=0$, and 
\begin{align*}
&\int_{\Omega}|\nu(v)\pa^{\al_1}f|_{L^2_v}|w(\al)\pa^{\al-\al_1}g|_{L^2_D}|w(\al)h|_{L^2_D}\,dx\\
&\quad\lesssim \Big(\sum_{|\al_1|=0}\|\nu(v)\pa^{\al_1}f\|_{L^\infty_xL^2_v}\|w(\al)\pa^{\al-\al_1}g\|_{L^2_xL^2_D}\\
&\quad\quad+\sum_{|\al_1|=1}\|\nu(v)\pa^{\al_1}f\|_{L^3_xL^2_v}\|w(\al)\pa^{\al-\al_1}g\|_{L^6_xL^2_D}\\
&\quad\quad+\sum_{2\le|\al_1|\le 3}\|\nu(v)\pa^{\al_1}f\|_{L^2_xL^2_v}\|w(\al)\pa^{\al-\al_1}g\|_{L^\infty_xL^2_D}\Big)\|w(\al)h\|_{L^2_xL^2_D}\\
&\quad\lesssim \sum_{|\al_1|\le 3}\|\nu(v)w(\al_1)\pa^{\al_1}f\|_{L^2_xL^2_v}\sum_{|\al_1|\le 3}\|w(\al_1)\pa^{\al_1}g\|_{L^2_xL^2_D}\|w(\al)h\|_{L^2_xL^2_D}. 
\end{align*}
%when $1\le|\al|\le 2$.
%Note that there's at least one derivative on $g$ and we used $1\le w(\al_1)$. 
 Similarly, the second right hand term of \eqref{216} can be estimated as 
%\begin{align*}
%	\int_{\Omega}|\nu(v)w(0)f|_{L^2_v}|g|_{L^2_D}|w(0)h|_{L^2_D}\,dx
%	&\lesssim \|\nu(v)w(0)f\|_{L^2_xL^2_v}\|g\|_{L^\infty_xL^2_D}\|w(0)h\|_{L^2_xL^2_D}\\
%	&\lesssim \|\nu(v)w(0)f\|_{L^2_xL^2_v}\|w(2)g\|_{H^2_xL^2_D}\|w(0)h\|_{L^2_xL^2_D}. 
%\end{align*}
%when $|\al|=0$, and 
\begin{align*}
	&\int_{\Omega}|\nu(v)w(\al)\pa^{\al_1}f|_{L^2_v}|\pa^{\al-\al_1}g|_{L^2_D}|w(\al)h|_{L^2_D}\,dx\\
	&\quad\lesssim \Big(\sum_{|\al-\al_1|=0}\|\nu(v)w(\al)\pa^{\al_1}f\|_{L^2_xL^2_v}\|\pa^{\al-\al_1}g\|_{L^\infty_xL^2_D}\\
	&\quad\quad+\sum_{|\al-\al_1|=1}\|\nu(v)w(\al)\pa^{\al_1}f\|_{L^3_xL^2_v}\|\pa^{\al-\al_1}g\|_{L^6_xL^2_D}\\
	&\quad\quad+\sum_{2\le|\al-\al_1|\le3}\|\nu(v)w(\al)\pa^{\al_1}f\|_{L^\infty_xL^2_v}\|\pa^{\al-\al_1}g\|_{L^2_xL^2_D}\Big)\|w(\al)h\|_{L^2_xL^2_D}\\
	&\quad\lesssim \sum_{|\al_1|\le3}\|\nu(v)w(\al_1)\pa^{\al_1}f\|_{L^2_xL^2_v}\sum_{|\al_1|\le3}\|w(\al_1)\pa^{\al_1}g\|_{L^2_xL^2_D}\|w(\al)h\|_{L^2_xL^2_D}.
\end{align*}
%when $1\le|\al|\le 2$. 
As a conclusion, inserting the above estimates into \eqref{216}, we have 
%if $|\al|=0$, then 
%\begin{align}\label{332}
%	|(\Gamma(f,g), w^2(0)h)_{L^2_{x,v}}|\lesssim \sum_{|\al_1|\le 2}\|\nu(v)w(\al_1)\pa^{\al_1}f\|_{L^2_xL^2_v}\sum_{|\al_1|\le2}\|w(\al_1)\pa^{\al_1}g\|_{L^2_xL^2_D}\|w(0)h\|_{L^2_xL^2_D}.
%\end{align}
%If $1\le|\al|\le 2$, then 
\begin{multline}\label{332a}
	|(\pa^\al \Gamma(f,g), w^2(\al)h)_{L^2_{x,v}}|\\\lesssim \sum_{|\al_1|\le 3}\|\nu(v)w(\al_1)\pa^{\al_1}f\|_{L^2_xL^2_v}\sum_{|\al_1|\le3}\|w(\al_1)\pa^{\al_1}g\|_{L^2_xL^2_D}\|w(\al)h\|_{L^2_xL^2_D}.
\end{multline}

Next we give the estimate on $\ol G$. 
	Recall \eqref{olG2} that 
	\begin{align*}
		\ol G = \frac{\sqrt{R}}{\sqrt{\th}}\bar{\th}_{x_1} A_1\big(\frac{v-u}{\sqrt{R\th}}\big) + \bar u_{1x_1}B_{11}\big(\frac{v-u}{\sqrt{R\th}}\big). 
	\end{align*}
%Temporarily denote $\pa^\al=\pa^{\al_0}_t\pa^{\al_1}_{x_1}\pa^{\al_2}_{x_2}\pa^{\al_3}_{x_3}$. 
Then using \eqref{decayAB}, for $b\ge 0$ and $|\al'|\le 2$, one has 
	\begin{align*}
		|\<v\>^b\mu^{-1/2}\pa^{\al'}_{\beta}\ol G|_{L^2_{v}}
		\lesssim 
		%|\pa^{\al'}(\bar{u}_{x_1},\bar\th_{x_1})| + 
		\sum_{\al'_1+\al'_2+\al'_3=\al'}|\pa^{\al'_1}(\bar{u}_{x_1},\bar\th_{x_1})||\pa^{\al'_2}(u,\th)|^{\text{sgn}(\al'_2)}|\pa^{\al'_K}(u,\th)|^{\text{sgn}(\al'_3)}. 
	\end{align*}
where $\text{sgn}(\al'_j)$ equal to $1$ if $|\al'_j|\ge 1$ and equal to $0$ if $|\al'_j|=0$. 
Taking integration over $x\in\Omega$, we have 
\begin{multline}\label{35d}
	\|\<v\>^b\mu^{-1/2}\pa^{\al'}_{\beta}\ol G\|_{L^2_xL^2_{v}}
	\lesssim 
	\|\pa^{\al'}(\bar{u}_{x_1},\bar\th_{x_1})\|_{L^2_x} \\+\sum_{\al'_1+\al'_2=\al',\,|\al'_2|\ge 1}\|\pa^{\al'_1}(\bar{u}_{x_1},\bar\th_{x_1})\|_{L^\infty_x}\|\pa^{\al'_2}(u,\th)\|_{L^2_x}\\
	\qquad+ \sum_{\substack{\al'_1+\al'_2+\al'_3=\al',\\ |\al'_2|,|\al'_3|\ge 1}}\|\pa^{\al'_1}(\bar{u}_{x_1},\bar\th_{x_1})\|_{L^\infty_x}\|\pa^{\al'_2}(u,\th)\|_{L^3_x}\|\pa^{\al'_3}(u,\th)\|_{L^6_x}.  
\end{multline}
Note that $(\bar u,\bar\th)$ depends only on $(t,x_1)$ and $\|\cdot\|_{L^\infty_{x_1}}\lesssim \|\na_x(\cdot)\|^{1/2}_{L^2_{x_1}}\|\cdot\|_{L^2_{x_1}}^{1/2}$. Moreover, using \eqref{GagL3}, \eqref{GagL6} and decomposition $(u,\th)=(\wt u,\wt\th)+(\bar u,\bar\th)$, we have 
\begin{align*}
	\|\pa^{\al'_2}(u,\th)\|_{L^2_x}+\|\pa^{\al'_2}(u,\th)\|_{L^3_x}+\|\pa^{\al'_3}(u,\th)\|_{L^6_x}\lesssim \sqrt{\E_k(t)}+\delta^{\frac{1}{2}}\le 1, 
	%		\|\pa^{\al'_2}\na_x(u,\th)\|_{L^2_x}\|\pa^{\al'_2}(u,\th)\|_{L^2_x}. 
\end{align*} 
by choosing $\ve>0$ in \eqref{priori2} and $\delta>0$ in \eqref{delta} small.   
Then applying Lemma \ref{Lem21}, \eqref{35d} becomes 
\begin{align}\label{34}
	\|\<v\>^b\mu^{-1/2}\pa^{\al'}_{\beta}\ol G\|_{L^2_xL^2_{v}}
	\lesssim \delta^{\frac{1}{2}}(1+t)^{-\frac{1}{2}}, 
\end{align}
when $|\al'|\ge 0$ and 
\begin{align}\label{344}
	\|\<v\>^b\mu^{-1/2}\pa^{\al'}_{\beta}\ol G\|_{L^2_xL^2_{v}}
	\lesssim \delta^{\frac{1}{3}}(1+t)^{-\frac{2}{3}}, 
\end{align}
when $|\al'|\ge 1$, for any $b>0$.

\smallskip 
%Applying \eqref{332} and \eqref{34}, we obtain from $G=\ol G+\mu^{1/2}\g$ that 
%\begin{align*}
%	\big|\big(\Gamma(\mu^{-\frac{1}{2}}G,\mu^{-\frac{1}{2}}G),\g\big)_{L^2_{x,v}}\big|
%	&\lesssim \big(\delta^{\frac{1}{2}}(1+t)^{-\frac{1}{2}}+\sum_{|\al_1|\le 2}\|\nu(v)w(\al_1)\pa^{\al_1}\g\|_{L^2_xL^2_v}\big)\\&\quad\times\big(\delta^{\frac{1}{2}}(1+t)^{-\frac{1}{2}}+\sum_{|\al_1|\le2}\|w(\al_1)\pa^{\al_1}\g\|_{L^2_xL^2_D}\big)\|w(0)\g\|_{L^2_xL^2_D}\\
%	&\lesssim \big(\delta^{\frac{1}{2}}+\sqrt{\E_k(t)}\big)\D_k(t) + \delta(1+t)^{-2}. 
%\end{align*}
%Similarly, 
%\begin{align*}
%	\big|\big(\Gamma(\mu^{-\frac{1}{2}}G,\f),\f\big)_{L^2_{x,v}}\big|\lesssim \big(\delta^{\frac{1}{2}}+\sqrt{\E_k(t)}\big)\D_k(t). 
%\end{align*}
Applying \eqref{332a} and \eqref{34}, we obtain from $G=\ol G+\mu^{1/2}\g$ that 
%for $1\le|\al|\le 2$, we have 
\begin{align*}
	\big|\big(\pa^\al\Gamma(\mu^{-\frac{1}{2}}G,\mu^{-\frac{1}{2}}G),\pa^\al\g\big)_{L^2_{x,v}}\big|
	&\lesssim \big(\delta^{\frac{1}{2}}(1+t)^{-\frac{1}{2}}+\sum_{|\al_1|\le 2}\|\nu(v)w(\al_1)\pa^{\al_1}\g\|_{L^2_xL^2_v}\big)\\&\quad\times\big(\delta^{\frac{1}{2}}(1+t)^{-\frac{1}{2}}+\sum_{|\al_1|\le2}\|w(\al_1)\pa^{\al_1}\g\|_{L^2_xL^2_D}\big)\|w(\al)\pa^\al\g\|_{L^2_xL^2_D}\\
	&\lesssim \big(\delta^{\frac{1}{2}}+\sqrt{\E_k(t)}\big)\D_k(t) + \delta^{}(1+t)^{-2}. 
\end{align*}
and 
\begin{align*}
	\big|\big(\pa^\al\Gamma(\mu^{-\frac{1}{2}}G,\f),\pa^\al\f\big)_{L^2_{x,v}}\big|\lesssim \big(\delta^{\frac{1}{2}}+\sqrt{\E_k(t)}\big)\D_k(t). 
\end{align*}
This completes the proof of Lemma \ref{Lem25}. 		
\end{proof}

\section{Estimates on Fluid Quantities}\label{Sec4}
In this section, we will derive the estimates on fluid quantities $(\wt\rho,\wt u,\wt\th)$. 

\subsection{Estimates on Time Derivatives and $\wt\rho$}
We begin with the estimate on dissipation rate of $\pa_t(\wt\rho,\wt u,\wt\th)$ and $\na_x \wt\rho$. 
Here we follow the arguments in \cite{Duan2020a,Duan2021} but with some modification for deducing the dissipation rate. 

\begin{Lem}\label{Lem32}
	Let $(F_\pm,\phi)$ be the solution to \eqref{1}. Assume $\ve>0$ in \eqref{priori2} is small enough. 
	Then
	for $|\al|\le 2$, we have 
	\begin{multline}\label{pat}
		 \|\pa^\al \pa_t(\wt\rho,\wt u,\wt\th)\|_{L^2_x}^2\lesssim \|\pa^\al\na_x(\wt\rho,\wt u,\wt\th)\|_{L^2_x}^2 + \|\pa^\al\na_x\g\|_{L^2_xL^2_{\gamma/2}}^2 + \delta^{\frac{2}{3}}(1+t)^{-\frac{4}{3}}\\ + \E_k(t)\min\{\E_k(t),\D_k(t)\}.
	\end{multline}
%	for some small generic constant $\lam>0$. 
%	For the time derivative on $(\rho,u,\th)$, we have 
%	\begin{align}\label{patrho}
%		\|\pa_t(\rho,u,\th)\|^2_{L^2_x}
%		&\lesssim \|\na_x(\rho,u,\th,\na_x\phi)\|^2_{L^2_x} + \|\na_x\g\|_{L^2_xL^2_{\gamma/2}}^2 + \delta^{\frac{2}{3}}(1+t)^{-\frac{4}{3}}. 
%	\end{align}
	 For $|\al|\le 2$, we have 
	\begin{multline}
		\label{rho}
		\pa_t(\pa^\al \wt u,\pa^\al \na_x\wt\rho)_{L^2_x}+\lam\|\pa^\al \na_x\wt\rho\|^2_{L^2_x}\lesssim  \|\pa^\al\na_x\g\|_{L^2_xL^2_{\gamma/2}}^2 + \|\pa^\al\na_x(\wt u,\wt\th)\|_{L^2_x}^2 \\+ \delta^{\frac{2}{3}}(1+t)^{-\frac{4}{3}} +\E_k(t)\D_k(t), 
	\end{multline}
for some generic constant $\lam>0$.
\end{Lem}
\begin{proof}	
	It follows from \eqref{fluid} that  
	\begin{equation*}
%		\label{nobar}
		\left\{
		\begin{aligned}
			&\pa_t\rho + \na_x\rho \cdot u +\rho \na_x\cdot u = 0,\\
			& \pa_tu +  \na_xu\cdot u+\frac{2\th}{3\rho}\na_x \rho + \frac{2}{3}\na_x\th +\frac{\na_x\phi}{\rho}\int_{\R^3}F_2\,dv = -\frac{1}{\rho}\int_{\R^3}v\otimes v\cdot\na_x G\,dv,\\
			&\pa_t\th + \frac{2}{3}\th\na_x\cdot u+u\cdot\na_x\th + \frac{\na_x\phi}{\rho}\cdot\int_{\R^3}(v-u)F_2\,dv\\&\qquad = \frac{1}{\rho} \Big(-\int_{\R^3}\frac{|v|^2}{2} v\cdot\na_x G\,dv + u\cdot\int_{\R^3}v\otimes v\cdot\na_x G\,dv\Big). 
		\end{aligned}\right. 
	\end{equation*}
	Together with \eqref{Euler2}, one has 
	\begin{equation}\label{350}\left\{
		\begin{aligned}
			&\pa_t\wt\rho + \na_x\wt\rho \cdot u +\rho \na_x\cdot\wt u + \wt\rho\na_x\cdot \bar{u} + \na_x\bar\rho\cdot\wt u = 0,\\
			& \pa_t\wt u +  \na_x\wt u\cdot  u+\frac{2\th}{3\rho}\na_x\wt\rho + \frac{2}{3}\na_x\wt\th +\wt u\cdot\na_x\bar u+\frac{2}{3}\big(\frac{\th}{\rho}-\frac{\bar\th}{\bar\rho}\big)\na_x\bar\rho
			\\&\qquad +\frac{\na_x\phi}{\rho}\int_{\R^3}F_2\,dv  = -\frac{1}{\rho}\int_{\R^3}v\otimes v\cdot\na_x G\,dv,\\
			&\pa_t\wt\th + \na_x\wt\th\cdot u+ \frac{2}{3}\th\na_x\cdot\wt u+\wt u\cdot\na_x\bar\th + \frac{2}{3}\wt\th\na_x\cdot \bar u + \frac{\na_x\phi}{\rho}\cdot\int_{\R^3}(v-u)F_2\,dv\\
			&\qquad\qquad= \frac{1}{\rho} \Big(-\int_{\R^3}\frac{|v|^2}{2} v\cdot\na_x G\,dv + u\cdot\int_{\R^3}v\otimes v\cdot\na_x G\,dv\Big). 
		\end{aligned}\right. 
	\end{equation}
For the estimate with derivative on $t$, noticing $\|\na_x(\wt\rho,\wt u,\wt\th)\|_{H^1_x}^2\le \ve<1$, we have from \eqref{350}, \eqref{344} and Lemma \ref{Lem21} that 
\begin{align}\label{rhopat}\notag
	\|\pa_t\wt\rho\|^2_{L^2_x}&\le \|\na_x\wt\rho\|^2_{L^2_x} +\|\na_x\cdot\wt u\|^2_{L^2_x} + \|\wt\rho\|^2_{L^3_x}\|\na_x\cdot \bar{u}\|^2_{L^6_x} + \|\na_x\bar\rho\|_{L^6_x}\|\wt u\|_{L^3_x}\\
	&\lesssim \|\na_x(\wt\rho,\wt u)\|_{L^2_x}^2 + \delta^{\frac{2}{3}}(1+t)^{-\frac{4}{3}},
%	
%	&\lesssim \|\na_x(\wt\rho,\wt u,\wt\th)\|_{L^2_x} + \E_{k_0}^{\frac{1}{2}}(t)\D^{\frac{1}{2}}(t) + \delta(1+t)^{-1}, \\
\end{align}
%and
\begin{multline*}
	\|\pa_t\wt u\|^2_{L^2_x} \lesssim \|\na_x(\wt\rho,\wt u,\wt\th)\|_{L^2_x}^2 +\|\wt u\|^2_{L^3_x}\|\na_x\bar u\|^2_{L^6_x}\\
	 +\|(\wt\rho,\wt\th)\|^2_{L^3_x}\|\na_x\bar\rho\|^2_{L^6_x} +\|\na_x\phi\|^2_{L^3_x}\|F_2\|^2_{L^6_xL^2_2}+ \|\<v\>^4\na_x G\|^2_{L^2_{x,v}}\\
	\lesssim \|\na_x(\wt\rho,\wt u,\wt\th)\|^2_{L^2_x} + \|\na_x\g\|^2_{L^2_{x}L^2_{\gamma/2}} + \delta^{\frac{2}{3}}(1+t)^{-\frac{4}{3}} + \E_k(t)\min\{\E_k(t),\D_k(t)\},
\end{multline*}
and
\begin{multline*}
	\|\pa_t\wt\th\|^2_{L^2_x} \lesssim \|\na_x(\wt\rho,\wt u,\wt\th)\|_{L^2_x}^2+\|\wt u\|^2_{L^3_x}\|\na_x\bar\th\|^2_{L^6_x}\\
	  + \|\wt\th\|^2_{L^3_x}\|\na_x\cdot \bar u\|^2_{L^6_x} +\|\na_x\phi\|^2_{L^3_x}\|F_2\|^2_{L^6_xL^2_3}+\|\<v\>^5\na_x G\|^2_{L^2_{x,v}}\\
	\lesssim \|\na_x(\wt\rho,\wt u,\wt\th)\|^2_{L^2_x} + \|\na_x\g\|^2_{L^2_xL^2_{\gamma/2}} + \delta^{\frac{2}{3}}(1+t)^{-\frac{4}{3}} + \E_k(t)\min\{\E_k(t),\D_k(t)\}.
\end{multline*}
Here we used \eqref{GagL3} and \eqref{GagL6}. 
% for $L^3_x$ norm and embedding $\|\cdot\|_{L^6_x}\lesssim\|\na_x(\cdot)\|_{L^2_x}$.
%, $G=\ol G+\mu^{1/2}\g$ and fast velocity decay \eqref{fast} for $\ol G$. 
%These estimates yield \eqref{pat} for $|\al|=0$. 
When $|\al|=1$, we apply \eqref{pat1} and Lemma \ref{Lem21} to find that 
\begin{align}\label{rhopat1}\notag
	\|\pa^\al\pa_t\wt\rho\|^2_{L^2_x}&\le \|\pa^\al\na_x(\wt\rho,\wt u)\|^2_{L^2_x}+\|\na_x(\wt\rho,\wt u)\|^2_{L^3_x}\|\pa^\al(\rho,u)\|^2_{L^6_x} \\\notag&\quad+ \|\pa^\al(\wt\rho,\wt u)\|^2_{L^3_x}\|\na_x(\bar\rho, \bar{u})\|^2_{L^6_x} + \|(\wt\rho,\wt u)\|^2_{L^\infty_x}\|\pa^\al\na_x(\bar\rho,\bar u)\|^2_{L^2_x}\\
%	&\lesssim \|\na_x(\wt\rho,\wt u)\|_{L^2_x}\|\na_x(\rho,u)\|_{L^2_x}^{\frac{1}{2}}\|\na_x^2(\rho,u)\|_{L^2_x}^{\frac{1}{2}}+ \|\na_x(\wt\rho,\wt u)\|_{H^1_x}\|\na_x(\bar\rho,\bar{u})\|_{L^2_x}\\
	&\lesssim  \|\pa^\al\na_x(\wt\rho,\wt u)\|^2_{L^2_x} + \delta^{\frac{2}{3}}(1+t)^{-\frac{4}{3}},
\end{align}
	\begin{multline*}
	 \|\pa^\al\pa_t\wt u\|^2_{L^2_x} \notag
	 \lesssim \|\pa^\al\na_x(\wt\rho,\wt u,\wt\th)\|^2_{L^2_x}+\|\pa^\al(\wt\rho,\wt u,\wt\th)\|^2_{L^3_x}\|\na_x(\bar\rho,\bar u)\|^2_{L^6_x}\\
	  +\|\pa^\al\rho\|^2_{L^3_x}\|\na_x\phi\|^2_{L^2_x}\|F_2\|_{L^6_xL^2_3}^2+\|\pa^\al\na_x\phi\|^2_{L^3_x}\|F_2\|^2_{L^6_xL^2_2}+\|\na_x\phi\|^2_{L^3_x}\|\pa^\al F_2\|^2_{L^6_xL^2_2}\\
	  +\|\pa^\al\rho\|^2_{L^6_x}\|\<v\>^4\na_x G\|^2_{L^3_{x}L^2_v}+ \|\<v\>^4\pa^\al\na_x G\|^2_{L^2_{x,v}}\\
	 \lesssim \|\pa^\al\na_x(\wt\rho,\wt u,\wt\th)\|^2_{L^2_x} + \|\pa^\al\na_x\g\|^2_{L^2_{x}L^2_{\gamma/2}}  + \delta^{\frac{2}{3}}(1+t)^{-\frac{4}{3}}\\+\E_k(t)\min\{\E_k(t),\D_k(t)\},
 \end{multline*}
and
\begin{multline*}
	 \|\pa^\al\pa_t\wt\th\|^2_{L^2_x}
	  \lesssim \|\pa^\al\na_x(\wt\rho,\wt u,\wt\th)\|^2_{L^2_x} + \|\pa^\al(\wt u,\wt\th)\|^2_{L^3_x}\|\na_x(\bar u,\bar\th)\|^2_{L^6_x}
	   + \|(\wt u,\wt\th)\|^2_{L^\infty_x}\|\pa^\al\na_x(\bar u,\bar\th)\|^2_{L^2_x}\\
	   +\|\pa^\al\rho\|^2_{L^3_x}\|\na_x\phi\|^2_{L^2_x}\|F_2\|_{L^6_xL^2_3}^2
	   +\|\pa^\al\na_x\phi\|^2_{L^3_x}\| F_2\|^2_{L^6_xL^2_3} +\|\na_x\phi\|^2_{L^3_x}\|\pa^\al F_2\|^2_{L^6_xL^2_3}\\
	   +\|\pa^\al\rho\|^2_{L^6_x}\|\<v\>^5\na_x G\|^2_{L^3_{x}L^2_v}+\|\<v\>^5\pa^\al\na_x G\|^2_{L^2_{x,v}}\\
	 \lesssim \|\na_x(\wt\rho,\wt u,\wt\th)\|^2_{L^2_x} + \|\pa^\al\na_x\g\|^2_{L^2_xL^2_{\gamma/2}} + \delta^{\frac{2}{3}}(1+t)^{-\frac{4}{3}}+\E_k(t)\min\{\E_k(t),\D_k(t)\}.
\end{multline*}
The above estimates give \eqref{pat} when $|\al|\le 1$. The estimates for $|\al|=2$ can be obtained similarly.  

Next we calculate \eqref{rho} for the case $|\al|\le 1$, since the case of $|\al|=2$ can be derived by using similar calculations. 
For $|\al|\le 1$, applying $\pa^\al$ to the second equation of \eqref{350} and
%	we assume $\al_0=0$ first. 
	taking inner product with $\pa^\al \na_x\wt\rho$ over $\Omega$, we have 
	\begin{multline}\label{325}
		\big(\frac{2\th}{3\rho}\pa^\al \na_x\wt\rho,\pa^\al \na_x\wt\rho\big)_{L^2_x} = -\big(\pa_t\pa^\al \wt u,\pa^\al \na_x\wt\rho\big)_{L^2_x} 
		- \big(\frac{2}{3}\pa^\al \na_x\wt\th+\pa^\al \na_x\wt u\cdot  u,\pa^\al \na_x\wt\rho\big)_{L^2_x} \\
		\quad
		- \big(\pa^\al (\wt u\cdot\na_x\bar u)-\frac{2}{3}\pa^\al \big(\big(\frac{\th}{\rho}-\frac{\bar\th}{\bar\rho}\big)\na_x\bar\rho\big),\pa^\al \na_x\wt\rho\big)_{L^2_x} \\
		-\big(\pa^\al\big(\frac{\na_x\phi}{\rho}\int_{\R^3}F_2\,dv\big),\pa^\al \na_x\wt\rho\big)_{L^2_x} 
		- \big(\pa^\al(\frac{1}{\rho}\int_{\R^3}v\otimes v\cdot\na_x  G\,dv),\pa^\al \na_x\wt\rho\big)_{L^2_x}
	\end{multline}
	For the first term on the right hand side of \eqref{325}, noticing the boundary values from \eqref{boundaryrhouth} on boundary $\Gamma_i$ for $i=2,3$ for the case of rectangular duct, we can take integration by parts to obtain  
	\begin{align*}
		(\pa^\al \wt u,\pa^\al \na_x\wt\rho_t)_{L^2_x} = -(\pa^\al \na_x\wt u,\pa^\al \wt\rho_t)_{L^2_x},
	\end{align*}
%for any $|\al|\le 1$.
% Note that there's no boundary on $x_1$ direction. 
%	When $|\al|=0$, using \eqref{rhopat} to control $\pa_t\wt\rho$,  we have 
%	\begin{align*}
%		&\quad\,-\big(\pa_t\wt u,\na_x\wt\rho\big)_{L^2_x}
%		= -\pa_t(\wt u,\na_x\wt\rho)_{L^2_x} -(\na_x\wt u,\wt\rho_t)_{L^2_x}\\
%%		&\le -\pa_t(\wt u,\na_x\wt\rho)_{L^2_x} + \|\na_x\wt u\|_{L^2_x}\Big(\|\na_x\wt\rho\|_{L^2_x}+\|\na_x\wt u\|_{L^2_x}+\|\wt\rho\|_{L^6_x}\|\na_x\bar u\|_{L^3_x} + \|\na_x\bar\rho\|_{L^3_x}\|\wt u\|_{L^6_x}\Big)\\
%		&\le -\pa_t(\wt u,\na_x\wt\rho)_{L^2_x} + \eta\|\na_x\wt\rho\|_{L^2_x}^2 +C_\eta\|\na_x\wt u\|_{L^2_x}^2 + \delta^{\frac{2}{3}}(1+t)^{-\frac{4}{3}},
%%		+ \big(\delta^{\frac{1}{2}}+\sqrt{\E_{k_0}(t)}\big)\D^{(|\al|)}_k(t).
%%		+\eta\|\na_x\wt\rho\|_{L^2_x}^2 + C_\eta\|\na_x\wt u\|_{L^2_x}^2 + C\delta^2(1+t)^{-2} + C\|(\wt\rho,\wt u)\|_{H^1_x}^2\|\na_x(\wt\rho,\wt u)\|_{H^1_x}^2,
%	\end{align*}
%	for $\eta>0$. 
%	Here we use \eqref{Gag} to control the $L^\infty_x$ norm. 
%	 and decomposition $(\rho,u,\th)=(\wt\rho,\wt u,\wt\th)+(\bar\rho,\bar u,\bar\th)$. 
%	Similarly, when $|\al|\le 2$, 
	Using \eqref{rhopat} and \eqref{rhopat1} for the case $|\al|=0,1$ respectively, we have 
	\begin{align*}
		&\quad\,-\big(\pa_t\pa^\al \wt u,\pa^\al \na_x\wt\rho\big)_{L^2_x}
		= -\pa_t(\pa^\al \wt u,\pa^\al \na_x\wt\rho)_{L^2_x} -(\pa^\al \na_x\wt u,\pa^\al \wt\rho_t)_{L^2_x}\\
%		&\le -\pa_t(\pa^\al \wt u,\pa^\al \na_x\wt\rho)_{L^2_x} + \|\pa^\al \na_x\wt u\|_{L^2_x}\Big(\|\pa^\al \na_x\wt\rho\|_{L^2_x}+\|\pa^\al \na_x\wt u\|_{L^2_x}\\
%		&\qquad+\|\na_x\wt\rho\|_{L^3_x}\|\pa^\al u\|_{L^6_x}+\|\pa^\al\rho\|_{L^6_x}\|\na_x\wt u\|_{L^3_x}+\|\pa^\al\wt\rho\|_{L^6_x}\|\na_x\bar u\|_{L^3_x}\\
%		&\qquad+\|\wt\rho\|_{L^\infty_x}\|\pa^\al\na_x\bar u\|_{L^2_x} + \|\pa^\al \na_x\bar\rho\|_{L^2_x}\|\wt u\|_{L^\infty_x} + \|\na_x\bar\rho\|_{L^3_x}\|\pa^\al \wt u\|_{L^6_x}\Big)\\
		&\le -\pa_t(\pa^\al \wt u,\pa^\al \na_x\wt\rho)_{L^2_x}
%		+ \big(\delta^{\frac{1}{2}}+\sqrt{\E_{k_0}(t)}\big)\D_k(t).
		+\eta\|\pa^\al \na_x\wt\rho\|_{L^2_x}^2 + C_\eta\|\pa^\al \na_x\wt u\|_{L^2_x}^2 +
		\delta^{\frac{2}{3}}(1+t)^{-\frac{4}{3}},  
%		\big(\delta^{\frac{1}{2}}+\sqrt{\E_{k_0}(t)}\big)\D^{(|\al|)}_k(t).
%		 C\delta^2(1+t)^{-2} + C\E_{k_0}(t)\D_k(t).
	\end{align*}
for any $\eta>0$. 
%Here we apply Lemma \ref{Lem21} and \eqref{pat1} for the time decay on $(\bar\rho,\bar u,\bar\th)$. 
	For the second and third terms on the right hand side of \eqref{325}, when $|\al|=0$, we have 
	\begin{align*}
		&\quad\,- \big(\frac{2}{3}\na_x\wt\th+\na_x\wt u\cdot  u
%		,\na_x\wt\rho\big)_{L^2_x} 
%		- \big(
		-\wt u\cdot\na_x\bar u-\frac{2}{3}\big(\frac{\bar\rho\wt\th-\wt\rho\bar\th}{\rho\bar\rho}\big)\na_x\bar\rho,\na_x\wt\rho\big)_{L^2_x} \\
		&\lesssim \eta\|\na_x\wt\rho\|_{L^2_x}^2 + C_\eta\|\na_x(\wt u,\wt\th)\|_{L^2_x}^2 + C_\eta\|(\wt\rho,\wt u,\wt\th)\|^2_{L^3_x}\|\na_x(\bar\rho,\bar u)\|^2_{L^6_x}\\
		&\lesssim \eta\|\na_x\wt\rho\|_{L^2_x}^2 + C_\eta\|\na_x(\wt u,\wt\th)\|_{L^2_x}^2 + C_\eta \delta^{\frac{2}{3}}(1+t)^{-\frac{4}{3}}. 
	\end{align*}
	When $|\al|=1$, we have 
	\begin{align*}
		&\quad\,- \big(\frac{2}{3}\pa^\al\na_x\wt\th+\pa^\al(\na_x\wt u\cdot  u)
%		,\pa^\al\na_x\wt\rho\big)_{L^2_x} 
%		- \big(
		-\pa^\al(\wt u\cdot\na_x\bar u)-\frac{2}{3}\pa^\al\big(\big(\frac{\bar\rho\wt\th-\wt\rho\bar\th}{\rho\bar\rho}\big)\na_x\bar\rho\big),\pa^\al\na_x\wt\rho\big)_{L^2_x} \\
		&\lesssim \eta\|\pa^\al\na_x\wt\rho\|_{L^2_x}^2 + C_\eta\big(\|\pa^\al\na_x(\wt u,\wt\th)\|_{L^2_x}^2 +\|\na_x\wt u\|^2_{L^3_x}\|\pa^\al u\|_{L^6_x}^2 \\
		&\qquad\qquad\qquad+ \|\pa^\al(\wt\rho,\wt u,\wt\th)\|^2_{L^3_x}\|\na_x(\bar\rho,\bar u)\|^2_{L^6_x}+ \|(\wt\rho,\wt u,\wt\th)\|^2_{L^3_x}\|\pa^\al\na_x(\bar\rho,\bar u)\|^2_{L^6_x}\big)\\
		&\lesssim \eta\|\pa^\al\na_x\wt\rho\|_{L^2_x}^2 + C_\eta\|\pa^\al\na_x(\wt u,\wt\th)\|_{L^2_x}^2 + C_\eta\delta^{\frac{2}{3}}(1+t)^{-\frac{4}{3}}  + C_\eta\E_k(t)\D_k(t). 
	\end{align*}
%When $|\al|=2$, we have 
%\begin{align*}
%	&\quad\,- \big(\frac{2}{3}\pa^\al\na_x\wt\th+\pa^\al(\na_x\wt u\cdot  u)
%	-\pa^\al(\wt u\cdot\na_x\bar u)-\frac{2}{3}\pa^\al\big(\big(\frac{\bar\rho\wt\th-\wt\rho\bar\th}{\rho\bar\rho}\big)\na_x\bar\rho\big),\pa^\al\na_x\wt\rho\big)_{L^2_x} \\
%	&\lesssim \eta\|\pa^\al\na_x\wt\rho\|_{L^2_x}^2 + C_\eta\big(\|\pa^\al\na_x(\wt u,\wt\th)\|_{L^2_x}^2 +\|\na^2_x\wt u\|^2_{L^6_x}\|\na_x u\|_{L^3_x}^2+\|\na_x\wt u\|^2_{L^3_x}\|\pa^\al u\|_{L^6_x}^2 \\
%	&\qquad+ \|\pa^\al(\wt\rho,\wt u,\wt\th)\|^2_{L^6_x}\|\na_x(\bar\rho,\bar u)\|^2_{L^3_x}
%	+ \|\na_x(\wt\rho,\wt u,\wt\th)\|^2_{L^6_x}\|\na^2_x(\bar\rho,\bar u)\|^2_{L^3_x}\\&\qquad
%	+ \|\na_x(\wt\rho,\wt u,\wt\th)\|^2_{L^6_x}\|\pa^\al\na_x(\bar\rho,\bar u)\|^2_{L^3_x}\\
%%	&\qquad+ \|(\wt\rho,\wt u,\wt\th)\|^2_{L^6_x}\|\pa^\al\na_x(\bar\rho,\bar u)\|^2_{L^3_x}+ \|(\wt\rho,\wt u,\wt\th)\|^2_{L^6_x}\|\pa^\al\na_x(\bar\rho,\bar u)\|^2_{L^3_x}\big)\\
%	&\lesssim \eta\|\pa^\al\na_x\wt\rho\|_{L^2_x}^2 + C_\eta\|\pa^\al\na_x(\wt u,\wt\th)\|_{L^2_x}^2 + C_\eta\delta^{\frac{2}{3}}(1+t)^{-\frac{4}{3}}  + C_\eta\E_k(t)\D_k(t). 
%\end{align*}
For the fourth term on the right hand side of \eqref{325}, when $|\al|=0$, we have 
\begin{align*}
	\big|\big(\frac{\na_x\phi}{\rho}\int_{\R^3}F_2\,dv,\na_x\wt\rho\big)_{L^2_x}\big|
	&\lesssim \|\na_x\phi\|_{L^3_x}\|F_2\|_{L^6_xL^2_2}\|\na_x\wt\rho\|_{L^2_x}
	\lesssim \eta\|\na_x\wt\rho\|_{L^2_x}^2+C_\eta\E_k(t)\D_k(t).
\end{align*}
When $|\al|= 1$, we have 
\begin{multline*}
	\big|\big(\pa^\al\big(\frac{\na_x\phi}{\rho}\int_{\R^3}F_2\,dv\big),\pa^\al \na_x\wt\rho\big)_{L^2_x}\big|
	\lesssim \big(\|\pa^\al\na_x\phi\|_{L^3_x}\|F_2\|_{L^6_xL^2_2}+\|\na_x\phi\|_{L^\infty_x}\|\pa^\al F_2\|_{L^2_xL^2_2}\\
	\qquad +\|\pa^\al\rho\|_{L^3_x}\|\na_x\phi\|_{L^2_x}\|F_2\|_{L^6_xL^2_2}\big)\|\pa^\al\na_x\wt\rho\|_{L^2_x}\\
	\lesssim \eta\|\pa^\al\na_x\wt\rho\|_{L^2_x}^2
%	+C_\eta\min\{\delta^{\frac{2}{3}}(1+t)^{-\frac{4}{3}},(1+t)^{-2}\}
+C_\eta\E_k(t)\D_k(t). 
\end{multline*}
%When $|\al|=2$, we can deduce 
%\begin{multline*}
%	\,\big|\big(\pa^\al\big(\frac{\na_x\phi}{\rho}\int_{\R^3}F_2\,dv\big),\pa^\al \na_x\wt\rho\big)_{L^2_x}\big|
%	\lesssim \big(\|\pa^\al\na_x\phi\|_{L^3_x}\|F_2\|_{L^6_xL^2_2}+\|\na_x\phi\|_{L^\infty_x}\|\pa^\al F_2\|_{L^2_xL^2_2}\\
%	 +\|\na_x\rho\|_{L^\infty_x}\big(\|\na^2_x\phi\|_{L^6_x}\|F_2\|_{L^3_xL^2_2}+\|\na_x\phi\|_{L^3_x}\|\na_xF_2\|_{L^6_xL^2_2}\big)\\+\|\pa^\al\rho\|_{L^6_x}\|\na_x\phi\|_{L^\infty_x}\|F_2\|_{L^3_xL^2_2}\big)\|\pa^\al\na_x\wt\rho\|_{L^2_x}\\
%	\lesssim \eta\|\pa^\al\na_x\wt\rho\|_{L^2_x}^2+C_\eta\delta^{\frac{2}{3}}(1+t)^{-\frac{4}{3}}+C_\eta\E_k(t)\D_k(t),
%\end{multline*}
%Note that $\|\pa^\al\rho\|^2_{L^6_x}\lesssim\|\pa^\al\na_x\wt\rho\|^2_{L^2_x}+\|\pa^\al\na_x\bar\rho\|^2_{L^2_x}\lesssim \D_k(t)+\delta^{\frac{2}{3}}(1+t)^{-\frac{4}{3}}$.  
	For the fifth term on the right hand of \eqref{325}, 
	recalling $G=\ol G+\sqrt\mu\g$ and applying \eqref{344} to control $\na_x\ol G$, when $|\al|=0$, we have 
	\begin{align*}
		- \big(\frac{1}{\rho}\int_{\R^3}v\otimes v\cdot \na_x G\,dv,\na_x\wt\rho\big)_{L^2_x}
		&\le \eta\|\na_x\wt\rho\|_{L^2_x}^2 + C_\eta\|\na_x\g\|_{L^2_xL^2_{\gamma/2}}^2\\&\quad + C_\eta\delta^{\frac{2}{3}}(1+t)^{-\frac{4}{3}}. 
	\end{align*}
When $|\al|=1$, we have 
\begin{multline*}
- \,\big(\pa^\al\big(\frac{1}{\rho}\int_{\R^3}v\otimes v\cdot \na_x G\,dv\big),\pa^\al \na_x\wt\rho\big)_{L^2_x}
\le \eta\|\pa^\al \na_x\wt\rho\|_{L^2_x}^2 + C_\eta\|\pa^\al \na_x\g\|_{L^2_xL^2_{\gamma/2}}^2\\ + C_\eta\|\pa^\al \na_x\ol G\|_{L^2_xL^2_{4}}^2 + C_\eta\|\pa^\al\rho\|^2_{L^6_x}\|\na_x(\ol G+\sqrt\mu\g)\|^2_{L^3_x}\\
\lesssim \eta\|\pa^\al \na_x\wt\rho\|_{L^2_x}^2 + C_\eta\|\pa^\al \na_x\g\|_{L^2_xL^2_{\gamma/2}}^2 + C_\eta\delta^{\frac{2}{3}}(1+t)^{-\frac{4}{3}} +\E_k(t)\D_k(t).
\end{multline*}
%If $|\al|=2$, we have 
%\begin{multline*}
%- \big(\pa^\al\big(\frac{1}{\rho}\int_{\R^3}v\otimes v\cdot \na_x G\,dv\big),\pa^\al \na_x\wt\rho\big)_{L^2_x}
%\le \eta\|\pa^\al \na_x\wt\rho\|_{L^2_x}^2 + C_\eta\|\pa^\al \na_x\g\|_{L^2_xL^2_{\gamma/2}}^2\\ + C_\eta\delta^{\frac{2}{3}}(1+t)^{-\frac{4}{3}}+ C_\eta\big(\|\pa^\al\rho\|^2_{L^3_x}\|\na_x(\ol G+\sqrt\mu\g)\|^2_{L^6_x}+\|\na_x\rho\|^2_{L^3_x}\|\na^2_x(\ol G+\sqrt\mu\g)\|^2_{L^6_xL^2_v}\big)\\
%\lesssim \eta\|\pa^\al \na_x\wt\rho\|_{L^2_x}^2 + C_\eta\|\pa^\al \na_x\g\|_{L^2_xL^2_{\gamma/2}}^2 + C_\eta\delta^{\frac{2}{3}}(1+t)^{-\frac{4}{3}} +\E_k(t)\D_k(t).
%\end{multline*}
	Substituting the above estimates into \eqref{325} and choosing $\eta>0$ small enough, we obtain
	\begin{multline*}
		\pa_t(\pa^\al \wt u,\pa^\al \na_x\wt\rho)_{L^2_x}+\lam\|\pa^\al \na_x\wt\rho\|^2_{L^2_x}\lesssim  \|\pa^\al\na_x\g\|_{L^2_xL^2_{\gamma/2}}^2 + \|\pa^\al\na_x(\wt u,\wt\th)\|_{L^2_x}^2 \\+ \delta^{\frac{2}{3}}(1+t)^{-\frac{4}{3}} +\E_k(t)\D_k(t), 
	\end{multline*}
for any $|\al|\le 1$ and some $\lam>0$. The case $|\al|=2$ can be obtained similarly. 
%Applying \eqref{pat} to control the time derivative on the right hand side of \eqref{rho1}, we obtain \eqref{rho} for some small $\lam>0$. 
%	Similarly, by using Euler system \eqref{Euler2} instead of \eqref{350}, we can obtain the estimates \eqref{pat1} for $(\bar\rho,\bar u,\bar\th)$ with the help of Lemma \ref{Lem21}. 
	This completes the proof of Lemma \ref{Lem32}. 
\end{proof}

\subsection{Low Order Estimates on Macroscopic Components}
In this section, we will derive the low order energy estimate of the macroscopic components $(\wt\rho,\wt u,\wt\th)$ by using entropy-entropy flux. 
%The argument we use below is slightly different from \cite{Liu2006a,Li2017}, since we use 
We will use the estimate \eqref{LM1a} on $L_M^{-1}$ to obtain dissipation term $\|\na_x(\wt u,\wt\th)\|_{L^2_x}$. 

\begin{Lem}\label{Lem31}
	Let $(F_\pm,\phi)$ be the solution to \eqref{1}. Assume $\ve>0$ in \eqref{priori2} is small enough. Then for any $\eta>0$, 
%	\begin{multline}\label{32a}
%		\frac{d}{dt}\int_\Omega\eta(t)\,dx+\frac{d}{dt}\Big(\frac{\mu^{1/2} L_M^{-1}P_1\big(v\cdot\na_x(\bar\th\ln M)-\frac{3}{2}\bar u_{1x_1}v_1^2\big)M}{M},{\g}{}\Big)_{L^2_{x,v}}
%		+ \lam\|\sqrt{\bar u_{1x_1}}(\wt\rho,\wt u_1,\wt\th)\|^2_{L^2_x}\\
%		+ \lam\|\na_x(\wt u,\wt \th)\|_{L^2_x}^2 
%		\lesssim \delta^{\frac{1}{6}}(1+t)^{-\frac{7}{6}}+(\delta^2+\sqrt{\E_{k_0}(t)})\D_k(t)
%		+ \|\na^2_x(\wt\rho, \wt u,\wt\th)\|^2_{L^2_x}+\|\na_x\g\|^2_{H^1_xL^2_D},
%	\end{multline}
\begin{multline}\label{32a}
	\pa_t\int_\Omega\eta(t)\,dx + \kappa\pa_t(\wt u,\na_x\wt\rho)_{L^2_x}
	+ \lam\sum_{|\al|=1}\|\pa^\al(\wt\rho,\wt u,\wt \th)\|_{L^2_x}^2 
	+ \lam\|\sqrt{\bar u_{1x_1}}(\wt\rho,\wt u_1,\wt\th)\|^2_{L^2_x}\\
	\lesssim \big(\delta^{\frac{1}{3}}+\sqrt{\E_k(t)}\big)\D_k(t)
	+\delta^{\frac{1}{6}}(1+t)^{-\frac{7}{6}} +\sum_{|\al|=1}\|\pa^{\al}\g\|^2_{L^2_xL^2_{\gamma/2}},
\end{multline}
	for some $\lam>0$, where $\eta(t)$ is given by \eqref{238}. 
\end{Lem}

\begin{proof}
	As in \cite{Liu2004,Liu2006a}, we define the macroscopic entropy $S$ by 
	\begin{align}\label{S1}
		-\frac{3}{2}\rho S \equiv \int_{\R^3}M\ln M\,dv.
	\end{align}
	Multiplying \eqref{2} by $\ln M$, integrating over $v$ and using $G\in(\ker L_M)^{\perp}$ and \eqref{fluid}, it holds that (cf. \cite[Eq. (4.8), (4.9)]{Liu2006a})
	\begin{align}\label{32}
		-\frac{3}{2}(\rho S)_t - \frac{3}{2}\na_x\cdot (\rho uS) -\na_x\phi\cdot\int_{\R^3}\na_vF_2\ln M\,dv + \na_x\cdot \int_{\R^3}v\ln M G\,dv = \int_{\R^3}\frac{Gv\cdot\na_xM}{M}\,dv,
	\end{align}
	and 
	\begin{align}\label{S}
		S &= -\frac{2}{3}\ln\rho + \ln (2\pi R\th) + 1,\\
		\notag	p &= \frac{2}{3}\rho\th = k\rho^{\frac{5}{3}}\exp (S) ,\quad k=\frac{1}{2\pi e},\quad e=\frac{3}{2}R\th.
	\end{align}
	A convex entropy-entropy flux pair $(\eta,q)$ around a Maxwellian $\ol M= M_{[\bar\rho,\bar u,\bar{\th}]}$ $(\bar u_2=\bar u_3=0)$ can be given as follows. Denote the conservation laws \eqref{fluid} by 
	\begin{align*}
		m_t + \na_x\cdot n = -\begin{pmatrix}
			0\\\int_{\R^3}v_1v\cdot\na_x G\,dv\\
			\int_{\R^3}v_2v\cdot\na_x G\,dv\\
			\int_{\R^3}v_3v\cdot\na_x G\,dv\\
			\int_{\R^3}\frac{|v|^2}{2}v\cdot\na_x G\,dv
		\end{pmatrix}
	-\begin{pmatrix}
		0\\\pa_{x_1}\phi\int_{\R^3}F_2\,dv\\
		\pa_{x_2}\phi\int_{\R^3}F_2\,dv\\
		\pa_{x_3}\phi\int_{\R^3}F_2\,dv\\
		\na_x\phi\cdot\int_{\R^3}vF_2\,dv
	\end{pmatrix}.
	\end{align*}
	Here 
	\begin{align*}
		m &= (m_0,m_1,m_2,m_3,m_4)^t = \big(\rho,\rho u_1,\rho u_2,\rho u_3, \rho(\frac{1}{2}|u|^2+\th)\big)^t,\\
		n &= (n_0,n_1,n_2,n_3,n_4)^t = \big(\rho,\rho uu_1+p,\rho uu_2+p,\rho uu_3+p, \rho u(\frac{1}{2}|u|^2+\frac{5}{3}\th)\big)^t.
	\end{align*}
	Then we define an entropy-entropy flux pair $(\eta,q)$ as 
	\begin{align*}
		\eta &= \bar\th\big\{-\frac{3}{2}\rho S+\frac{3}{2}\bar\rho \bar S+\frac{3}{2}\na_m(\rho S)|_{m=\bar{m}}\cdot(m-\bar m)\big\},\\
		q_j &= \bar\th \big\{-\frac{3}{2}\rho u_jS+\frac{3}{2}\bar\rho\bar u_j\bar S+\frac{3}{2}\na_m(\rho S)|_{m=\bar{m}}\cdot(n-\bar n)\big\}.
	\end{align*}
	Since 
	\begin{align*}
		(\rho S)_{m_0}=S+\frac{|u|^2}{2\th}-\frac{5}{3},\ (\rho S)_{m_i}=-\frac{u_i}{\th},\ i=1,2,3,\ (\rho S)_{m_4}=\frac{1}{\th},
	\end{align*}
	we have 
	\begin{equation}\label{238}\begin{aligned}
			\eta &= \frac{3}{2}\Big\{\rho\th -\bar\th\rho S+\rho\Big[\big(\bar S-\frac{5}{3}\big)\bar\th+\frac{|u-\bar u|^2}{2}\Big]+\frac{2}{3}\bar\rho\bar\th\Big\},\\
			&= \frac{3}{2}\Big\{\frac{1}{2}\rho|u-\bar u|^2+\frac{2}{3}\rho\bar\th\Psi\Big(\frac{\bar\rho}{\rho}\Big)+\rho\bar\th\Psi\Big(\frac{\th}{\bar\th}\Big)\Big\}\\
			q_j &= u_j\eta + (u_j-\bar u_j)(\rho\th-\bar\rho\bar\th),\ j=1,2,3,
		\end{aligned}
	\end{equation}
	where $\Psi(s) = s-\ln s-1 $ is a strictly convex function near $s=1$. Using Taylor's expansion and direct calculations, one can obtain 
	\begin{align}\label{eta1}
		C^{-1}|(\wt\rho,\wt u,\wt\th)|^2 \le \eta \le C|(\wt\rho,\wt u,\wt\th)|^2,
	\end{align}
	for some generic constant $C>1$. 
	The entropy-entropy flux constructed has the following properties: $\eta(\bar m)=0$, $\na_m\eta(\bar m)=0$, and the Hessian $\na_m^2\eta$ is positive definite for any $m$ satisfying $\rho,\th>0$ (cf. \cite{Liu2004}). Thus, for $m$ in any bounded region in $\Sigma=\{m:\rho>0,\th>0\}$, there exists a positive constant $C>1$ such that 
	\begin{align*}
		C^{-1}|m-\bar m|^2\le \eta\le C|m-\bar m|^2. 
	\end{align*}	
	From \eqref{238}, it is straightforward to derive the following entropy equation 
	\begin{align}\label{35}
		\eta_t + \na_x\cdot q=\na_{(\bar\rho,\bar u,\bar S)}\eta\cdot (\bar\rho,\bar u,\bar S)_t+\na_{(\bar\rho,\bar u,\bar S)}q\cdot (\bar\rho,\bar u,\bar S)_{x_1}+I_1+I_2,
	\end{align}
	where $I_1$ and $I_2$ are given by 
	\begin{align}\label{i1i2}
		I_1 = \bar\th \big\{\big(-\frac{3}{2}\rho S\big)_t-\frac{3}{2}\na_x\cdot(\rho uS)\big\},\quad
		I_2 = \frac{3}{2}\bar\th\big\{\na_m(\rho S)|_{m=\bar m}\cdot(m_t+\na_x\cdot n)\big\}.
	\end{align}
	Using \eqref{Euler2}, \eqref{S} and \eqref{238}, the similar arguments as \cite{Yang2005,Wang2019,Duan2021} imply 
	\begin{align*}
		&-\Big[\na_{(\bar\rho,\bar u,\bar S)}\eta\cdot (\bar\rho,\bar u,\bar S)_t+\na_{(\bar\rho,\bar u,\bar S)}q\cdot (\bar\rho,\bar u,\bar S)_{x_1}\Big]\\
		&= \bar u_{1x_1}\Big[\frac{3}{2}\rho(u_1-\bar u_1)^2+\frac{2}{3}\rho\bar\th\Psi\Big(\frac{\bar\rho}{\rho}\Big)+\rho\bar\th\Psi\Big(\frac{\th}{\bar\th}\Big)\Big]+\frac{3}{2}\rho\bar\th_{x_1}(u_1-\bar u_1)(\frac{2}{3}\ln\frac{\bar\rho}{\rho}+\ln\frac{\th}{\bar\th})\\
		&\gtrsim \bar u_{1x_1}(\wt\rho^2+\wt u_1^2+\wt\th^2).
	\end{align*}
	For the term $I_1$ in \eqref{i1i2}, we have from \eqref{32} that 
	\begin{align*}
		I_1 
		&= -\bar\th \na_x\cdot \Big(\int_{\R^3}vG\ln M\,dv\Big) + \bar\th \int_{\R^3}\frac{Gv\cdot\na_xM}{M}\,dv+\bar\th\na_x\phi\cdot\int_{\R^3}\na_vF_2\ln M\,dv\\
		&=  -\na_x\cdot \Big(\int_{\R^3}v \bar\th G\ln M\,dv\Big) +  \int_{\R^3}Gv\cdot\na_x(\bar\th\ln M)\,dv+\bar\th\na_x\phi\cdot\int_{\R^3}\na_vF_2\ln M\,dv.
	\end{align*}
	For $I_2$ in \eqref{i1i2}, recalling $\bar u_2=\bar u_3=0$, we have 
	\begin{align*}
		I_2 &= \na_x\cdot\Big(\int_{\R^3}\Big[\frac{3}{2}v_1\bar u_1v-\frac{3}{2}v|v|^2\Big]G\,dv\Big)-\frac{3}{2}\bar u_{1x_1}\int_{\R^3}v^2_1G\,dv\\
		&\quad+\frac{3}{2}\bar u_1\pa_{x_1}\phi\int_{\R^3}F_2\,dv-\frac{3}{2}\na_x\phi\cdot\int_{\R^3}vF_2\,dv. 
	\end{align*}
	Combining the above estimates and integrating \eqref{35} over $x\in\Omega$, we have 
	\begin{align}\label{340}
		\frac{d}{dt}\int_\Omega\eta(t)\,dx + \lam\|\sqrt{\bar u_{1x_1}}(\wt\rho,\wt u_1,\wt\th)\|^2_{L^2_x}\lesssim \int_\Omega\int_{\R^3}\big(v\cdot\na_x(\bar\th\ln M)-\frac{3}{2}\bar u_{1x_1}v^2_1\big)G\,dvdx+I_3,
	\end{align}
for some $\lam>0$, 
where 
\begin{align*}
	I_3 = \int_{\Omega}\Big(\bar\th\na_x\phi\cdot\int_{\R^3}\na_vF_2\ln M\,dv+\frac{3}{2}\bar u_1\pa_{x_1}\phi\int_{\R^3}F_2\,dv-\frac{3}{2}\na_x\phi\cdot\int_{\R^3}vF_2\,dv\Big)\,dx. 
\end{align*}
The terms with divergence vanishes for the case of torus is direct from periodic property. For the case of rectangular duct, we used the fact that on boundary $\Gamma_j$ $(j=2,3)$,
	\begin{align*}
		q_j=\int_{\R^3}v_j\bar\th G\ln M\,dv = \int_{\R^3}\Big[\frac{3}{2}v_1\bar u_1v_j-\frac{3}{2}v_j|v|^2\Big]G\,dv = 0,
	\end{align*}
which follows from \eqref{boundaryrhouth}, \eqref{233a}, \eqref{229a} and changing of variable $v\mapsto R_xv$. 
Thus, 
\begin{align*}
	\int_{\Omega}\na_x\cdot q\,dx = \int_{\Omega}\na_x\cdot\int_{\R^3}v\bar\th G\ln M\,dvdx=\int_{\Omega}\int_{\R^3}\Big[\frac{3}{2}v_1\bar u_1v-\frac{3}{2}v|v|^2\Big]G\,dvdx = 0. 
\end{align*}
	
	Now we calculate the term $I_3$ in \eqref{340}. 
	By integration by parts on $v$, we have 
	\begin{align*}
		\int_{\Omega}\bar\th\na_x\phi\cdot\int_{\R^3}\na_vF_2\ln M\,dvdx = \int_{\Omega}\frac{\bar\th}{R\th}\na_x\phi\cdot\int_{\R^3}(v-u)F_2\,dvdx. 
	\end{align*}
%	Noticing $|u|\le \eta_0$ from \eqref{eta}, 
Then 
	\begin{align*}
%		\label{I3}\notag
		I_3&= \int_{\Omega}\int_{\R^3}\Big(\frac{3\wt\th}{2\th}\na_x\phi\cdot vF_2+\frac{3\wt\th}{2\th}\bar u_1\pa_{x_1}\phi\,F_2 - \frac{3\bar\th}{2\th}\na_x\phi\cdot\wt u F_2\Big)\,dvdx\\
		&\lesssim \|(\wt u,\wt\th)\|_{L^2_x}\|\na_x\phi\|_{L^3_x}\|F_2\|_{L^6_xL^2_3}
%		\lesssim \|\na_x(\wt u,\wt\th)\|_{L^2_x}\|\na_x\phi\|_{H^1_x}\|F_2\|_{L^2_xL^2_3}
		\lesssim \sqrt{\E_k(t)}\D_k(t).
	\end{align*}

	Next we calculate the first right-hand term of \eqref{340}. 
%	Noticing that  $(M^{-1/2}L_MM^{1/2}f,g)_{L^2_v}=(f,M^{-1/2}L_MM^{1/2}g)_{L^2_v}$, we have $(f,M^{-1}L_M^{-1}g)_{L^2_v}=(L^{-1}_Mf,M^{-1}g)_{L^2_v}$. Using such self-adjoint property of $L_M^{-1}$ and \eqref{G}, we have 
	Notice that $(P_1f,\frac{g}{M})_{L^2_v}=(f,\frac{P_1g}{M})_{L^2_v}$ and $L_M^{-1}=P_1L_M^{-1}$.  Using \eqref{G}, the first right-hand term of \eqref{340} can be represented as  
	\begin{multline}\label{341}
		\big(\big(v\cdot\na_x(\bar\th\ln M)-\frac{3}{2}\bar u_{1x_1}v^2_1\big)M,\frac{G}{M}\big)_{L^2_{x,v}}\\
		= \Big(P_1\big(v\cdot\na_x(\bar\th\ln M)M-\frac{3}{2}\bar u_{1x_1}v^2_1M\big),\frac{L^{-1}_MP_1v\cdot\na_xM +L^{-1}_M\Theta}{M}\Big)_{L^2_{x,v}}. 
	\end{multline}
We begin with the first right hand term of \eqref{341}. 
	Similar to \eqref{olG2}, a direct calculation gives 
	\begin{align}\label{418}\notag
		P_1v\cdot\na_xM &= P_1v\cdot \Big(\frac{|v-u|^2\na_x\th}{2R\th^2}+\frac{(v-u)\cdot\na_x u}{R\th}\Big)M\\
		&= \frac{\sqrt{R}}{\sqrt{\th}}\sum_{j=1}^3\pa_{x_j}\th \wh A_j\big(\frac{v-u}{\sqrt{R\th}}\big)M + \sum_{i,j=1}^3\pa_{x_i}u_j\wh B_{ij}\big(\frac{v-u}{\sqrt{R\th}}\big)M, 
	\end{align}
	and 
	\begin{align}\label{I5I8}\notag
		&P_1\big(v\cdot\na_x(\bar\th\ln M)-\frac{3}{2}\bar u_{1x_1}v_1^2\big)M 
		= P_1v_1\bar\th_{x_1}M\ln M +\bar\th P_1v\cdot\na_xM-\frac{3}{2}P_1\bar u_{1x_1}v^2_1M\\
		&=\notag - \bar\th_{x_1}P_1v_1M\frac{|v-u|^2}{2R\th}+(\th-\wt \th)P_1v\cdot\Big\{\frac{|v-u|^2\na_x\th}{2R\th^2}+\frac{(v-u)\cdot\na_xu}{R\th}\Big\}M-\frac{3}{2}P_1\bar u_{1x_1}v^2_1M\\
		&\notag=\bar\th_{x_1}P_1v\cdot\Big\{\frac{|v-u|^2\na_x\wt\th}{2R\th^2}+\frac{(v-u)\cdot\na_x\wt u}{R\th}\Big\}M - \wt\th P_1v\cdot \Big(\frac{(v-u)^2\na_x\th}{2R\th^2}+\frac{(v-u)\cdot\na_x u}{R\th}\Big)M\\
		&\notag= \sqrt{R\th}\sum_{j=1}^3\pa_{x_j}\wt\th \wh A_j\big(\frac{v-u}{\sqrt{R\th}}\big)M + \th\sum_{i,j=1}^3\pa_{x_i}\wt u_j\wh B_{ij}\big(\frac{v-u}{\sqrt{R\th}}\big)M \\
		&\qquad+\frac{\sqrt{R}\wt\th}{\sqrt{\th}}\sum_{j=1}^3\pa_{x_j}\th \wh A_j\big(\frac{v-u}{\sqrt{R\th}}\big)M + \wt\th\sum_{i,j=1}^3\pa_{x_i}u_j\wh B_{ij}\big(\frac{v-u}{\sqrt{R\th}}\big)M = \sum_{j=4}^{7}I_j. 
	\end{align}
%	Note that some terms inside $P_1v_1\bar\th_{x_1}M\ln M$ and $P_1v\cdot\na_xM$ vanish due to projection $P_1$; (cf. \cite[pp. (2.15)]{Duan2020a}). 
	Then by definition \eqref{21} of $A_j$ and $B_{ij}$, we have from \eqref{418} that 
	\begin{align}\label{I8I9}
		L_M^{-1}P_1v\cdot\na_xM =
		\frac{\sqrt{R}}{\sqrt{\th}}\sum_{j=1}^3\pa_{x_j}\th  A_j\big(\frac{v-u}{\sqrt{R\th}}\big) + \sum_{i,j=1}^3\pa_{x_i}u_j B_{ij}\big(\frac{v-u}{\sqrt{R\th}}\big)
%		\sqrt{R\th}\sum_{j=1}^3\pa_{x_j}\wt\th  A_j\big(\frac{v-u}{\sqrt{R\th}}\big) + \th\sum_{i,j=1}^3\pa_{x_i}\wt u_j B_{ij}\big(\frac{v-u}{\sqrt{R\th}}\big) \\
%		+\frac{\sqrt{R}\wt\th}{\sqrt{\th}}\sum_{j=1}^3\pa_{x_j}\th  A_j\big(\frac{v-u}{\sqrt{R\th}}\big) + \wt\th\sum_{i,j=1}^3\pa_{x_i}u_j B_{ij}\big(\frac{v-u}{\sqrt{R\th}}\big) 
		=:I_8 +I_9. 
	\end{align}
	Using the velocity decay from \eqref{decayAB}, we have 
	\begin{align*}
		\big|\big(\frac{I_6+I_7}{M},I_8+I_9\big)_{L^2_{x,v}}\big|
		&\lesssim \int_\Omega |\frac{\mu^{-1/2}(I_6+I_7)}{M}|_{L^2_v}|\mu^{1/2}(I_8+I_9)|_{L^2_v}\,dx
		\lesssim 
		\int_{\Omega}|\wt\th||\na_x(u,\th)|^2\,dx
		\\
%		&\lesssim \|\wt\th\|_{L^3_x}\|\na_x(u,\th)\|_{L^6_x}
%		&\lesssim \int_{\Omega}|\wt\th|\big(|\na_x(\wt u,\wt\th)|^2+|\pa_{x_1}(\bar u,\bar\th)|^2\big)\,dx\\
		&\lesssim \|\wt\th\|_{L^3_x}\big(\|\na_x(\wt u,\wt\th)\|_{L^3_x}^2+\|\pa_{x_1}(\bar u,\bar\th)\|_{L^3_x}^2\big)\\
		&\lesssim \|\wt\th\|_{H^1_x}\|\na_x(\wt u,\wt\th)\|_{L^2_x}^2+\|\wt\th\|_{H^1_x}\|\pa_{x_1}(\bar u,\bar\th)\|^2_{L^3_x}\\
		&\lesssim \delta^{\frac{2}{3}}(1+t)^{-\frac{4}{3}}+\sqrt{\E_k(t)}\D_k(t), 
	\end{align*}
	where we used \eqref{GagL3} to obtain $\|\wt\th\|_{L^3_x}\lesssim \sqrt{\E_k(t)}\le 1$ and
%	(note that $\|\wt\th\|_{L^2_x}<\infty$ from \eqref{priori2} and hence we can apply \eqref{Gag}), 
	Lemma \ref{Lem21}. 
%	 and Young's inequality $ab\lesssim a^4+b^{\frac{4}{3}}$. 
	Using decomposition $\pa_{x_j}\th= \pa_{x_j}\wt\th+\pa_{x_j}\bar\th$ and $\pa_{x_i}u_j=\pa_{x_i}\wt u_j+\pa_{x_i}\bar u_j$, we denote the terms $I_8=\wt I_8+\bar I_8$ and $I_9=\wt I_9+\bar I_9$ correspondingly. 
	From Lemma \ref{LemBur}, identities \eqref{mu} and \eqref{kappa}, we have 
	\begin{align*}
		(A_i,\wh A_j)_{L^2_x}=-\frac{\kappa(\th)}{R^2\th}\delta_{ij}, \quad (B_{ij},\wh B_{kl})=-\frac{\mu(\th)}{R\th}(\delta_{ik}\delta_{jl}+\delta_{il}\delta_{jk}-\frac{2}{3}\delta_{ij}\delta_{kl}),\ \forall\, i,j,k,l=1,2,3.
	\end{align*}
	Since $\th$ is close to $\frac{3}{2}$ and $\mu(\th)$, $\kappa(\th)$ are smooth about $\th$, there exists $C_1,C_2>0$ such that $\mu(\th),\kappa(\th)\in[C_1,C_2]$ and thus, 
	\begin{align*}
		\big(\frac{I_4+I_5}{M},\wt I_8+\wt I_9\big)_{L^2_{x,v}}
		&= -\int_{\Omega}\Big(\frac{\kappa(\th)}{R\th}|\na_x\wt\th|^2+\frac{\mu(\th)}{R}\Big[\frac{|\na_x\wt u+(\na_x\wt u)^t|^2}{2}-\frac{2}{3}(\na_x\cdot\wt u)^2\Big]\Big)dx\\
		&\le -c_0(\|\na_x\wt u\|_{L^2_x}^2+\|\na_x\wt \th\|_{L^2_x}^2),
	\end{align*}
	for some generic constant $c_0>0$. 
	For the part $\bar I_8+\bar I_9$, using orthogonality of $A_j$ and $B_{ij}$ and estimate \eqref{decayAB}, 
%	noticing $(\bar\th,\bar u)$ doesn't depend on $x_2$, $x_3$, using the orthogonality of $A_j$ and $B_{ij}$ and the fast velocity decay from \eqref{decayAB},
	 there exists function smooth $a(u,\th)$ and $b(u,\th)$ such that 
	\begin{align*}
		&\quad\,\big|\big(\frac{I_4+I_5}{M},\bar I_8+\bar I_9\big)_{L^2_{x,v}}\big|
		\lesssim \Big|\int_{\Omega}\big(\wt\th_{x_1}\bar\th_{x_1}a(u,\th)+\wt u_{x_1}\bar u_{1x_1}b(u,\th)\big)\,dx\Big|\\
		&\lesssim \Big|\int_{\Omega}\big(|(\wt u,\wt\th)||(\bar\th_{x_1x_1},\bar u_{1x_1x_1})|+|(\wt u,\wt\th)||(\bar\th_{x_1},\bar u_{1x_1})||(u_{x_1},\th_{x_1})|\big)\,dx\Big|\\
		&\lesssim \|(\wt u,\wt\th)\|_{L^\infty_x}\|(\bar\th_{x_1x_1},\bar u_{1x_1x_1})\|_{L^1_x}+
		\|(\wt u,\wt\th)\|_{L^6_x}\|(\bar\th_{x_1},\bar u_{1x_1})\|_{L^3_x}\big(\|(\wt u_{x_1},\wt\th_{x_1})\|_{L^2_x}+\|(\bar u_{x_1},\bar\th_{x_1})\|_{L^2_x}\big)\\
		&\lesssim \|\na_x(\wt u,\wt\th)\|^{\frac{1}{2}}_{H^2_x}\|(\wt u,\wt\th)\|_{H^2_x}^{\frac{1}{2}}\delta^{\frac{1}{8}}(1+t)^{-\frac{7}{8}}
		+\|\na_x(\wt u,\wt\th)\|_{L^2_x}\|(\wt u,\wt\th)\|_{H^1_x}\delta^{\frac{1}{3}}(1+t)^{-\frac{2}{3}}+\delta^{\frac{5}{6}}(1+t)^{-\frac{7}{6}}\\
		&\lesssim \delta^{\frac{1}{6}}(1+t)^{-\frac{7}{6}} +
%		 \|\na_x(\wt u,\wt\th)\|^2_{L^2_x}
		 \big(\delta^{\frac{1}{3}}+\E_k(t)\big)\D_k(t), 
	\end{align*}
where we used \eqref{Gag1}. 
	Here we take integration by parts on $x_1\in\R$ in the second inequality, 
%	use Lemma \ref{Lem21} and the  interpolation formula \eqref{Gag} in the forth inequality 
and Young's inequality $ab\lesssim a^4+b^{3/4}$ in the fifth inequality. 
%Note that $\|\na_x(\wt u,\wt\th)\|_{L^2_x}\le \ve<1$ from \eqref{priori2}. 
%	Using Gagliardo–Nirenberg interpolation inequality on $x_1\in\R$, (cf. \cite[Theorem 12.83]{Leoni2017} and \cite[Page 125]{Nirenberg1959}), we have 
%	\begin{align*}
%		\|g\|_{L^\infty_{x_1}}\lesssim \|g\|_{L^2_{x_1}}^{\frac{1}{2}}\|\pa_{x_1}g\|_{L^2_{x_1}}^{\frac{1}{2}}.
%	\end{align*}
%	For $(x_2,x_3)\in \cup_{i=1}^N(a_{i,2},b_{i,2})\times (a_{i,3},b_{i,3})$ as in \eqref{Omega}, we can apply again the Gagliardo–Nirenberg interpolation inequality in \cite[Theorem 12.83]{Leoni2017} (Although this theorem is for $\R^3$, but one can use extension theorem \cite[Thoerem VI.5, pp. 181]{Stein1971} to obtain the interpolation inequality in any bounded Lipschitz domains.) to obtain that 
%	%we use Sobolev inequality on two dimension (cf. \cite[Theorem 5.4]{Adams2003}) to obtain 
%	\begin{align*}
%		\|g\|_{L^\infty_{x_2,x_3}}\lesssim 
%		\|\na_{x_2,x_3}g\|_{L^2_{x_2,x_3}}.
%		%	\|g\|_{L^2_{x_2,x_3}}. 
%	\end{align*}
%	Combining the above two inequality, we have 
%	\begin{align}
%		\label{Gag}
%		\|g\|_{L^\infty_{x}}\lesssim\|\na_{x_2,x_3}g\|_{L^2_{x}}^{\frac{1}{2}}\|\pa_{x_1}\na_{x_2,x_3}g\|_{L^2_{x}}^{\frac{1}{2}}.
%	\end{align}
	It follows from \eqref{341}, \eqref{I5I8}, \eqref{I8I9} and the above estimates that 
	\begin{multline}\label{342a}
		\Big(\frac{P_1\big(v\cdot\na_x(\bar\th\ln M)M-\frac{3}{2}\bar u_{1x_1}v^2_1M\big)}{M},L^{-1}_M(P_1v\cdot\na_xM)\Big)_{L^2_{x,v}}\\
		\le -c_0\|\na_x(\wt u,\wt \th)\|_{L^2_x}^2 + C\delta^{\frac{1}{6}}(1+t)^{-\frac{7}{6}}+C\big(\delta^{\frac{1}{3}}+\sqrt{\E_k(t)}\big)\D_k(t),
	\end{multline}for some generic constant $c_0>0$. 
	
	Next we estimate the term with $\Theta$ in \eqref{341}. From \eqref{Theta2}, we know that 
	\begin{align}\label{Theta}
		\Theta = \pa_tG +P_1v\cdot\na_xG-P_1\na_x\phi\cdot\na_vF_2- 2Q(G,G). 
	\end{align}
	For the first factor on the right hand side of \eqref{341}, 
	using 
%	the velocity decay \eqref{decayAB} of Burnett functions and 
	\eqref{I5I8}, 
%	for any multi-index $\beta$, 
	we have 
	\begin{multline}\label{343a}
		\Big\|\frac{P_1\big(v\cdot\na_x(\bar\th\ln M)-\frac{3}{2}\bar u_{1x_1}v_1^2\big)M}{M^{1/2}}\Big\|_{L^2_{x}L^2_v}
		\lesssim \|\na_x(\wt u,\wt\th)\|_{L^2_{x}} + \|\wt \th\|_{L^\infty_x}\|\na_x(u,\th)\|_{L^2_x}\\
		\lesssim \|\na_x(\wt u,\wt\th)\|_{L^2_{x}} + \|\wt \th\|_{H^2_x}\big(\|\na_x(\wt u,\wt\th)\|_{L^2_x}+\delta^{\frac{1}{2}}(1+t)^{-\frac{1}{2}}\big)\\
		\lesssim \|\na_x(\wt u,\wt\th)\|_{L^2_{x}} + \delta(1+t)^{-1}+ \sqrt{\E_k(t)\D_k(t)}.
	\end{multline}
	Here we used \eqref{Gag} in the second inequality. For the term $\pa_tG+v\cdot\na_xG$ in \eqref{Theta}, 
%	we split it as $\pa_tG=\pa_t\ol G+\mu^{1/2}\pa_t\g$. 
%	For $\pa_t\ol G$, 
	we have from \eqref{LM1a} and \eqref{344} that 
	\begin{align*}
		&\|M^{-1/2}L_M^{-1}{\pa_t G}\|_{L^2_{x}L^2_v}+\|M^{-1/2}L_M^{-1}{P_1v\cdot\na_x G}\|_{L^2_{x}L^2_v}\\
		&\quad\lesssim \|M^{-1/2}\pa_t G\|_{L^2_xL^2_{-\gamma/2}}+\|M^{-1/2}P_1v\cdot\na_xG\|_{L^2_{x}L^2_{-\gamma/2}}\\
%		 + \|\pa_t\g\|_{L^2_xL^2_{k_0+3|\gamma|/2}}\\
		&\quad\lesssim \delta^{\frac{1}{3}}(1+t)^{-\frac{2}{3}} + \|\na_{x,t}\g\|_{L^2_xL^2_{\gamma/2}},
	\end{align*}
where we use 
\begin{align}\label{bbbb}
	|\<v\>^bM^{-1/2}\mu^{1/2}|\le C_b,
\end{align}
 for any $b>0$ to control the extra weight $-\gamma/2$, which is from the fact that $R\th>1$, i.e. \eqref{eta}.  
Then we obtain 
\begin{multline}\label{347}
	\Big(\frac{P_1\big(v\cdot\na_x(\bar\th\ln M)-\frac{3}{2}\bar u_{1x_1}v_1^2\big)M}{M},{L_M^{-1}(\pa_tG+P_1v\cdot\na_x G)}{}\Big)_{L^2_{x,v}}\\
	\le 
	%		\frac{d}{dt}\Big(\frac{P_1\big(v\cdot\na_x(\bar\th\ln M)-\frac{3}{2}\bar u_{1x_1}v_1^2\big)M}{M},{L_M^{-1}\g}\Big)_{L^2_{x,v}} 
	%		+ 
	\frac{c_0}{8}\|\na_x(\wt u,\wt\th)\|_{L^2_{x}}^2 + C\|\na_{x,t}\g\|^2_{L^2_xL^2_{\gamma/2}}
	%		+ \eta\|\g\|_{L^2_xH^s_{k+\gamma/2}}^2 \\
	+ C\delta^{\frac{2}{3}}(1+t)^{-\frac{4}{3}}
	%		 + \|\na^2_x(\wt\rho, \wt u,\wt\th)\|^2_{L^2_x}+\|\na^2_x\g\|^2_{L^2_xH^s_{k+\gamma/2}}\big)
	+ \big(\delta+\sqrt{\E_k(t)}\big)\D_k(t), 
\end{multline}
where $c_0$ is given in \eqref{342a}. 
For the term $-P_1\na_x\phi\cdot\na_vF_2+2Q(G,G)$ in \eqref{Theta}, which only occurs for the VPB case, we have $\gamma\ge 0$ and $s\ge 1/2$. Then we have from \eqref{343a}, \eqref{bbbb}, \eqref{LM1a} and \eqref{2134} that 
\begin{align}\notag\label{348}
	&\quad\,\Big(\frac{ P_1\big(v\cdot\na_x(\bar\th\ln M)-\frac{3}{2}\bar u_{1x_1}v_1^2\big)M}{M},L_M^{-1}\big(-P_1\na_x\phi\cdot\na_vF_2+2Q(G,G)\big)\Big)_{L^2_{x,v}}\\
	&\notag\lesssim \Big\|\frac{P_1\big(v\cdot\na_x(\bar\th\ln M)-\frac{3}{2}\bar u_{1x_1}v_1^2\big)M}{M^{1/2}}\Big\|_{L^2_{x}L^2_v}\\&\quad\times\notag\Big(\|\na_x\phi\|_{H^1_x}\|M^{-1/2}\na_vF_2\|_{H^1_xH^{-s}_{-\gamma/2}}+\|M^{-1/2}Q(G,G)\|_{L^2_xH^{-s}_{-\gamma/2}}\Big)\\
	%		+\|v\cdot\na_x\g\|_{L^2_{x}L^2_{k_0+3|\gamma|/2}}\\
	&\notag\lesssim \big(\|\na_x(\wt u,\wt\th)\|_{L^2_{x}} + \delta(1+t)^{-1}+ \sqrt{\E_k(t)\D_k(t)}\big)_{L^2_{x,v}}\\
	&\notag\quad\times
	\Big(
%	\|\pa_{x_1x_1}(\bar u_1,\bar\th)\|_{L^2_x}+\|\pa_{x_1}(\bar u_1,\bar\th)\|_{L^2_x}\|\na_x(u,\th)\|_{L^2_x}+
%	\min\{\delta^{\frac{1}{3}}(1+t)^{-\frac{2}{3}},(1+t)^{-1}\}+
	\|\na_x\phi\|_{H^1_x}\|\f\|_{H^1_xH^{s}_{\gamma/2}}+\big\|\frac{\nu G}{\sqrt\mu}\big\|_{H^1_xL^2_v}\big\|\frac{G}{\sqrt\mu}\big\|_{H^1_xL^2_D}\Big)\\
	&\le C\delta^{\frac{2}{3}}(1+t)^{-\frac{4}{3}}+ \frac{c_0}{8}\|\na_x(\wt u,\wt\th)\|_{L^2_{x}}^2  + C\sqrt{\E_k(t)}\D_k(t).
\end{align}
where we use \eqref{344} to estimate $\ol G$ and $|\na_v(\mu^{1/2}(\cdot))|_{H^{-s}_{-\gamma/2}}\lesssim |\cdot|_{H^s_{\gamma/2}}$. 
%for $k\ge k_0+3|\gamma|/2-\gamma/2+1$. 
%Here we used Young's inequality $ab\lesssim a^{4/3}+b^4$. 
%For the term $Q(G,G)$ in \eqref{Theta}, we use \eqref{343a}, \eqref{LM1} and \eqref{2134} to obtain 
%\begin{align}
%	\label{349}
%	\notag
%	&\quad\,\Big(\frac{ P_1\big(v\cdot\na_x(\bar\th\ln M)-\frac{3}{2}\bar u_{1x_1}v_1^2\big)M}{M},L_M^{-1}\Big)_{L^2_{x,v}}\\
%	&\notag=\Big\|\frac{P_1\big(v\cdot\na_x(\bar\th\ln M)-\frac{3}{2}\bar u_{1x_1}v_1^2\big)M}{M}\Big\|_{L^2_{x}L^2_{-7}}\|Q(G,G)\|_{L^2_xH^{-s}_{k_0+3|\gamma|/2}}\\
%	&\notag\lesssim \Big(\|\na_x(\wt u,\wt\th)\|_{L^2_{x}} + \delta(1+t)^{-1}+ \sqrt{\E_k(t)\D_k(t)}\Big)\|G\|_{L^2_xL^2_{14}}\|\na_xG\|_{L^2_xH^s_{k_0+2|\gamma|+2s}}\\
%	&\le \frac{c_0}{8} \|\na_x(\wt u,\wt\th)\|^2_{L^2_{x}} + \delta^2(1+t)^{-2}+  C \sqrt{\E_k(t)}\D_k(t). 
%\end{align}
%for $k\ge k_0+2|\gamma|+2s$. 
Substituting estimates \eqref{342a}, \eqref{347} and \eqref{348} into \eqref{341}, and then inserting the resultant estimate into  \eqref{340}, we obtain  
\begin{multline}\label{430aa}
	\frac{d}{dt}\int_\Omega\eta(t)\,dx
	%		 + \frac{d}{dt}\Big(\frac{P_1\big(v\cdot\na_x(\bar\th\ln M)-\frac{3}{2}\bar u_{1x_1}v_1^2\big)M}{M},{L_M^{-1}\g}\Big)_{L^2_{x,v}}
	+ \frac{c_0}{2}\|\na_x(\wt u,\wt \th)\|_{L^2_x}^2 
	+ \lam\|\sqrt{\bar u_{1x_1}}(\wt\rho,\wt u_1,\wt\th)\|^2_{L^2_x}\\
	\lesssim \big(\delta^{\frac{1}{3}}+\sqrt{\E_k(t)}\big)\D_k(t)
	+\delta^{\frac{1}{6}}(1+t)^{-\frac{7}{6}} +\sum_{|\al|=1}\|\pa^{\al}\g\|^2_{L^2_xL^2_{\gamma/2}},
\end{multline}
for some generic constant $\lam>0$, where we choose $\delta,\eta>0$ sufficiently small. 
For the estimate on $\wt\rho$ and time derivative estimates, we take linear combination $\eqref{430aa}+\sum_{|\al|=0}\big(\kappa\times\eqref{rho}+\kappa^2\times\eqref{pat}\big)$ to deduce 
\begin{multline*}
	\frac{d}{dt}\int_\Omega\eta(t)\,dx + \kappa\pa_t(\wt u,\na_x\wt\rho)_{L^2_x}
	+ \lam\sum_{|\al|=1}\|\pa^\al(\wt\rho,\wt u,\wt \th)\|_{L^2_x}^2 
	+ \lam\|\sqrt{\bar u_{1x_1}}(\wt\rho,\wt u_1,\wt\th)\|^2_{L^2_x}\\
	\lesssim \big(\delta^{\frac{1}{3}}+\sqrt{\E_k(t)}\big)\D_k(t)
	+\delta^{\frac{1}{6}}(1+t)^{-\frac{7}{6}} +\|\na_{x,t}\g\|^2_{L^2_xL^2_{\gamma/2}},
\end{multline*}
for some $\lam>0$. 
This completes the proof of Lemma \ref{Lem31}. 
\end{proof}

\subsection{High Order Estimates on Macroscopic Components}
Next we consider the second order energy estimates of the macroscopic components $(\wt\rho,\wt u,\wt\th)$. 
%The following Lemma follows from \cite[Lemma 4.1]{Wang2019a} for cutoff VPB system. We provide the proof for non-cutoff case for the sake of completeness. 
%Here we shall use a different method to the works \cite{Li2017,Duan2020a}.
\begin{Lem}\label{Lem41} 
%	Let $K=2$ be the total order of derivatives.
%Let $\al=(\al_0,\al_1,\al_2,\al_3)$ be multi-index such that $1\le|\al|\le 2$ and $\al_0=0$ (i.e. there's no time derivative). 
Let
 $k\ge 0$ and 
% 1-\gamma$ and
  $(F_\pm,\phi)$ be the solution to \eqref{F1} (satisfying \eqref{specular} and \eqref{Neumann} for the case of rectangular duct). Assume $\ve>0$ in \eqref{priori2} is small enough. Let $\E_{int}(t)$ be given by 
\begin{align}\label{Eint}
	\E_{int}(t) := \sum_{\substack{1\le|\al|\le2\\\al_0=0}}\int_{\Omega}\Big(\frac{R\th}{2\rho^2}|\pa^\al\wt\rho|^2+\frac{1}{2}|\pa^\al\wt u|^2+\th^{-1}|\pa^\al\wt\th|^2\Big)\,dx + \kappa\sum_{\substack{1\le|\al|\le2\\\al_0=0}}(\pa^\al \wt u,\pa^\al \na_x\wt\rho)_{L^2_x}, 
\end{align}
for some small $\kappa>0$. 
Then we have 
\begin{multline}\label{highmacro}
	\pa_t\E_{int}(t)  
	+\lam \sum_{2\le|\al|\le3}\|\pa^\al(\wt\rho,\wt u,\wt\th)\|_{L^2_x}^2
	\lesssim \big(\delta^{\frac{1}{3}}+\sqrt{\E_k(t)}\big)\D_k(t)\\ + \delta^{\frac{2}{3}}(1+t)^{-\frac{4}{3}}+ \sum_{2\le|\al|\le3}\|\pa^\al \g\|^2_{L^2_xL^2_{\gamma/2}}, 
\end{multline}
for some generic constants $\lam>0$. 
\end{Lem}
\begin{proof}
Recall \eqref{350} that 
	\begin{equation}\label{ruth}\left\{
	\begin{aligned}
		&\pa_t\wt\rho + \na_x\wt\rho \cdot u +\rho \na_x\cdot\wt u + \wt\rho\na_x\cdot \bar{u} + \na_x\bar\rho\cdot\wt u = 0,\\
		& \pa_t\wt u +  \na_x\wt u\cdot  u+\frac{R\th}{\rho}\na_x\wt\rho + R\na_x\wt\th +\wt u\cdot\na_x\bar u+R\big(\frac{\wt\th\bar\rho-\wt\rho\bar\th}{\rho\bar\rho}\big)\na_x\bar\rho
		\\&\qquad +\frac{\na_x\phi}{\rho}\int_{\R^3}F_2\,dv  = -\frac{1}{\rho}\int_{\R^3}v\otimes v\cdot\na_x G\,dv,\\
		&\pa_t\wt\th + \na_x\wt\th\cdot u+ R\th\na_x\cdot\wt u+\wt u\cdot\na_x\bar\th + R\wt\th\na_x\cdot \bar u + \frac{\na_x\phi}{\rho}\cdot\int_{\R^3}(v-u)F_2\,dv\\
		&\qquad\qquad= \frac{1}{\rho} \Big(-\int_{\R^3}\frac{|v|^2}{2} v\cdot\na_x G\,dv + u\cdot\int_{\R^3}v\otimes v\cdot\na_x G\,dv\Big). 
	\end{aligned}\right. 
\end{equation}
%\begin{equation}\label{ruth}\left\{
%	\begin{aligned}
%		&\pa_t\rho + \na_x\rho \cdot u +\rho \na_x\cdot u = 0,\\
%		& \pa_tu_i +  \na_xu_i\cdot u+\frac{R\th}{\rho}\pa^{e_i} \rho + R\pa^{e_i}\th +\frac{\pa^{e_i}\phi}{\rho}\int_{\R^3}F_2\,dv = -\frac{1}{\rho}\int_{\R^3}v_i v\cdot\na_x G\,dv,\\
%		&\pa_t\th + R\th\na_x\cdot u+u\cdot\na_x\th + \frac{\na_x\phi}{\rho}\cdot\int_{\R^3}(v-u)F_2\,dv\\&\qquad = \frac{1}{\rho} \Big(-\int_{\R^3}\frac{|v|^2}{2} v\cdot\na_x G\,dv + u\cdot\int_{\R^3}v\otimes v\cdot\na_x G\,dv\Big). 
%	\end{aligned}\right. 
%\end{equation}
For any $n=1,2,3$, applying $\pa_{x_n}$ to \eqref{ruth}, we obtain 
\begin{equation}\label{ruth1}\left\{
\begin{aligned}
	&\pa_t\pa_{x_n}\wt\rho + \pa_{x_n}\na_x\wt\rho \cdot u + \na_x\wt\rho \cdot\pa_{x_n} u +\pa_{x_n}\rho \na_x\cdot\wt u +\rho \na_x\cdot\pa_{x_n}\wt u \\&\qquad+\pa_{x_n}\wt\rho\na_x\cdot \bar{u} + \wt\rho\na_x\cdot\pa_{x_n} \bar{u} + \pa_{x_n}\na_x\bar\rho\cdot\wt u   + \na_x\bar\rho\cdot\pa_{x_n}\wt u= 0,\\
	& \pa_t\pa_{x_n}\wt u + \pa_{x_n}\na_x\wt u\cdot  u+\na_x\wt u\cdot \pa_{x_n}u+\pa_{x_n}(\frac{R\th}{\rho})\na_x\wt\rho
	+\frac{R\th}{\rho}\na_x\pa_{x_n}\wt\rho \\&\qquad+ R\pa_{x_n}\na_x\wt\th +\pa_{x_n}\wt u\cdot\na_x\bar u +\wt u\cdot\na_x\pa_{x_n}\bar u+R\pa_{x_n}\big(\frac{\wt\th\bar\rho-\wt\rho\bar\th}{\rho\bar\rho}\big)\na_x\bar\rho\\&\qquad+R\big(\frac{\wt\th\bar\rho-\wt\rho\bar\th}{\rho\bar\rho}\big)\pa_{x_n}\na_x\bar\rho +\pa_{x_n}\Big(\frac{\na_x\phi}{\rho}\int_{\R^3}F_2\,dv\Big) = -\pa_{x_n}\Big(\frac{1}{\rho}\int_{\R^3}v\, v\cdot\na_x G\,dv\Big),\\
	&\pa_t\pa_{x_n}\wt\th + \pa_{x_n}\na_x\wt\th\cdot u+ \na_x\wt\th\cdot\pa_{x_n} u+ R\pa_{x_n}\th\na_x\cdot\wt u+ R\th\pa_{x_n}\na_x\cdot\wt u\\&\qquad+\pa_{x_n}\wt u\cdot\na_x\bar\th +\wt u\cdot\na_x\pa_{x_n}\bar\th + R\pa_{x_n}\wt\th\na_x\cdot \bar u+ R\wt\th\na_x\cdot\pa_{x_n} \bar u
	\\&\qquad + \pa_{x_n}\Big(\frac{\na_x\phi}{\rho}\cdot\int_{\R^3}(v-u)F_2\,dv\Big) = \pa_{x_n}\Big(-\frac{1}{\rho} \int_{\R^3}\frac{v\cdot(v-2u)}{2} v\cdot\na_x G\,dv\Big). 
\end{aligned}\right. 
\end{equation}
We will take a proper linear combination of equations in \eqref{ruth1}. 
%Multiplying $\eqref{ruth1}_1$ by $\frac{R\th}{\rho^2}\pa_{x_n}\rho$, $\eqref{ruth1}_2$ by $\pa_{x_n}u$ and $\eqref{ruth1}_3$ by $\frac{\pa_{x_n}\th}{\th}$ respectively, then adding them together and integrating over $\Omega$, we estimate the typical terms as the following. 
%
Multiplying $\eqref{ruth1}_1$ by $\frac{R\th}{\rho^2}\pa_{x_n}\wt\rho$ and taking integration over $x\in\Omega$, we deduce that 
\begin{multline}\label{ru1}
	\frac{1}{2}\pa_t\int_{\Omega}\frac{R\th}{\rho^2}|\pa_{x_n}\wt\rho|^2\,dx -\frac{1}{2}\int_{\Omega}\frac{R\pa_t\th}{\rho^2}|\pa_{x_n}\wt\rho|^2\,dx + \int_{\Omega}\frac{R\th\pa_t\rho}{\rho^3}|\pa_{x_n}\wt\rho|^2\,dx\\
	 + 
	\int_{\Omega}(\pa_{x_n}\na_x\wt\rho \cdot u + \na_x\wt\rho \cdot\pa_{x_n} u +\pa_{x_n}\rho \na_x\cdot\wt u +\rho \na_x\cdot\pa_{x_n}\wt u)\,\frac{R\th}{\rho^2}\pa_{x_n}\wt\rho\,dx \\
	+\int_{\Omega}(\pa_{x_n}\wt\rho\na_x\cdot \bar{u} + \wt\rho\na_x\cdot\pa_{x_n} \bar{u} + \pa_{x_n}\na_x\bar\rho\cdot\wt u   + \na_x\bar\rho\cdot\pa_{x_n}\wt u)\,\frac{R\th}{\rho^2}\pa_{x_n}\wt\rho\,dx = 0. 
\end{multline}
The typical trilinear terms will be estimated in \eqref{e1} and \eqref{e11}. Here we should deal with the fourth and seventh terms in \eqref{ru1}. 
For the case of rectangular duct, using the boundary values from \eqref{boundaryrhouth}, 
 we have that for $i=2,3$, if $n=i$, then $\pa_{x_n}\wt\rho=0$ on $\Gamma_i$ and if $n\neq i$, then $\pa_{x_n} u_i=u_i=0$ on $\Gamma_i$ ($\pa_{x_n}$ is tangent derivative on $\Gamma_i$ when $k\neq i$). Thus, 
we can take integration by parts to obtain 
\begin{align}\label{425}
	\int_{\Omega}\pa_{x_n}\na_x\wt\rho\cdot u\,\frac{R\th}{\rho^2}\pa_{x_n}\wt\rho\,dx
	=-\frac{1}{2}\int_{\Omega}\pa_{x_n}\wt\rho\na_x\cdot u\,\frac{R\th}{\rho^2}\pa_{x_n}\wt\rho\,dx
	- \frac{1}{2}\int_{\Omega}\pa_{x_n}\wt\rho u\cdot\na_x\big(\frac{R\th}{\rho^2}\big)\pa_{x_n}\wt\rho\,dx,
%	 - \int_{\Omega}\pa_{x_n}\rho u\cdot\pa_{x_n}\na_x\rho\,\frac{R\th}{\rho^2}\,dx,
\end{align}
and 
\begin{align}\label{426}
	\int_{\Omega}\rho\na_x\cdot\pa_{x_n}\wt u\,\frac{R\th}{\rho^2}\pa_{x_n}\wt\rho\,dx
	= 
	-\int_{\Omega}\pa_{x_n}\wt u\cdot\na_x\big(\frac{R\th}{\rho}\big)\pa_{x_n}\wt\rho\,dx
	-\int_{\Omega}\pa_{x_n}\wt u\cdot\na_x\pa_{x_n}\wt\rho\,\frac{R\th}{\rho}\,dx. 
\end{align}
The right hand side of \eqref{425} and first right-hand term of \eqref{426} are typical trilinear terms. 
%The second right-hand term of \eqref{426} will be eliminated thanks to the fifth left-hand term of \eqref{ru2}. 
%Using \eqref{ruth} to represent $\pa_t\th$ and $\pa_t\rho$, we have 
%\begin{multline*}
%	\frac{1}{2}\pa_t\big(\frac{R\th}{\rho^2}|\pa_{x_n}\rho|^2\big) +\frac{1}{2}\frac{R}{\rho^2}|\pa_{x_n}\rho|^2\big(R\th\na_x\cdot u+u\cdot\na_x\th\big)   -\frac{R\th}{\rho^3}\na_x\cdot(\rho u)|\pa_{x_n}\rho|^2\\
%	 + 
%	\pa_{x_n}\na_x\cdot(\rho u)\,\frac{R\th}{\rho^2}\pa_{x_n}\rho =  -\frac{1}{2}\frac{R}{\rho^3}|\pa_{x_n}\rho|^2\,\na_x\phi\cdot\int_{\R^3}(v-u)F_2\,dv\\
%	+\frac{1}{2}\frac{R}{\rho^3}|\pa_{x_n}\rho|^2 \Big(-\int_{\R^3}\frac{|v|^2}{2} v\cdot\na_x G\,dv + u\cdot\int_{\R^3}v\otimes v\cdot\na_x G\,dv\Big)\Big).
%\end{multline*}
Taking the inner product of $\eqref{ruth1}_2$ with $\pa_{x_n}\wt u$ over $\Omega$, we have 
\begin{multline}\label{ru2}
	\frac{1}{2}\pa_t\int_{\Omega}|\pa_{x_n}\wt u|^2\,dx +\sum_{i=1}^3\int_{\Omega}\big(\na_x\pa_{x_n}\wt u_i\cdot u+\na_x\wt u_i\cdot \pa_{x_n}u\big)\,\pa_{x_n}\wt u_i\,dx\\
	+\int_{\Omega}\big(\pa_{x_n}(\frac{R\th}{\rho})\na_x\wt\rho
	+\frac{R\th}{\rho}\na_x\pa_{x_n}\wt\rho+R\pa_{x_n}\na_x\wt\th +\pa_{x_n}\wt u\cdot\na_x\bar u \big)\cdot\pa_{x_n}\wt u\,dx \\
	+ \int_{\Omega}\big(\wt u\cdot\na_x\pa_{x_n}\bar u+R\pa_{x_n}\big(\frac{\wt\th\bar\rho-\wt\rho\bar\th}{\rho\bar\rho}\big)\na_x\bar\rho+R\big(\frac{\wt\th\bar\rho-\wt\rho\bar\th}{\rho\bar\rho}\big)\pa_{x_n}\na_x\bar\rho\big)\cdot\pa_{x_n}\wt u\,dx\\ +\int_{\Omega}\pa_{x_n}\wt u\cdot\pa_{x_n}\Big(\frac{\na_x\phi}{\rho}\int_{\R^3}F_2\,dv\Big)\,dx = -\int_{\Omega}\pa_{x_n}\wt u\cdot\pa_{x_n}\Big(\frac{1}{\rho}\int_{\R^3}v\, v\cdot\na_x G\,dv\Big)\,dx.
\end{multline}
We again estimate the trouble terms, i.e. the second, fifth and sixth terms in \eqref{ru2}. The fifth term of \eqref{ru2} will be eliminated by the second right-hand term of \eqref{426} after taking summation of \eqref{ru1} and \eqref{ru2}. For the second and sixth terms in \eqref{ru2}, taking boundary values \eqref{boundaryrhouth} into account, we have   
%notice from \eqref{boundaryrhouth} 
that for $i=2,3$, if $n=i$, then $u_i=\pa_{x_n}\rho=0$ on $\Gamma_i$ and if $n\neq i$, then $\pa_{x_n}u_i=0$ on $\Gamma_i$. 
% ($\pa_{x_n}$ is tangent derivative on $\Gamma_i$ in this case). 
Then by integration by parts, we have 
\begin{align}\label{425a}
	\int_{\Omega}\na_x\pa_{x_n}\wt u_i\cdot u\,\pa_{x_n}\wt u_i\,dx
	= -\frac{1}{2}\int_{\Omega}\pa_{x_n}\wt u_i\na_x\cdot u\,\pa_{x_n}\wt u_i\,dx,
%	-\int_{\Omega}\pa_{x_n}u_i\na_x\cdot u\,\pa_{x_n}u_i\,dx
\end{align}
%For the sixth term in \eqref{ru2}, notice
and 
\begin{align}\label{426a}
	\int_{\Omega}R\pa_{x_n}\na_x\wt\th\cdot\pa_{x_n}\wt u\,dx
	= -\int_{\Omega}R\pa_{x_n}\wt\th\na_x\cdot\pa_{x_n}\wt u\,dx. 
\end{align}
The right hand side of \eqref{425a} becomes typical trilinear term. 
%The right hand side of \eqref{426a} will be eliminated by the forth term of left hand side of \eqref{ru3}. 
Multiplying $\eqref{ruth1}_3$ by $\frac{\pa_{x_n}\wt\th}{\th}$ and taking integration over $\Omega$, we have 
\begin{multline}\label{ru3}
	\pa_t\int_{\Omega}\th^{-1}|\pa_{x_n}\wt\th|^2\,dx 
	+\int_{\Omega}\frac{\pa_t\th}{\th^2}|\pa_{x_n}\wt\th|^2\,dx
%	R\th\na_x\cdot u+u\cdot\na_x\th + \frac{\na_x\phi}{\rho}\cdot\int_{\R^3}(v-u)F_2\,dv\\&\qquad = \frac{1}{\rho} \Big(-\int_{\R^3}\frac{|v|^2}{2} v\cdot\na_x G\,dv + u\cdot\int_{\R^3}v\otimes v\cdot\na_x G\,dv\Big)
	 + \int_{\Omega}\big(\pa_{x_n}\na_x\wt\th\cdot u+ \na_x\wt\th\cdot\pa_{x_n} u\big)\,\frac{\pa_{x_n}\wt\th}{\th}\,dx\\+\int_{\Omega}\big(R\pa_{x_n}\th\na_x\cdot\wt u+ R\th\pa_{x_n}\na_x\cdot\wt u + \pa_{x_n}\wt u\cdot\na_x\bar\th +\wt u\cdot\na_x\pa_{x_n}\bar\th\big)\,\frac{\pa_{x_n}\wt\th}{\th}\,dx\\ +\int_{\Omega}\big( R\pa_{x_n}\wt\th\na_x\cdot \bar u+ R\wt\th\na_x\cdot\pa_{x_n} \bar u\big)\,\frac{\pa_{x_n}\wt\th}{\th}\,dx
	 + \int_{\Omega}\frac{\pa_{x_n}\wt\th}{\th}\pa_{x_n}\Big(\frac{\na_x\phi}{\rho}\cdot\int_{\R^3}(v-u)F_2\,dv\Big)\,dx\\ = \int_{\Omega}\frac{\pa_{x_n}\wt\th}{\th}\pa_{x_n}\Big( 
%	 \big(
	 -\frac{1}{\rho}\int_{\R^3}\frac{v\cdot(v-2u)}{2} v\cdot\na_x G\,dv
%	  + u\cdot\int_{\R^3}v\otimes v\cdot\na_x G\,dv\big)
	  \Big)\,dx.
\end{multline}
The sixth term of \eqref{ru3} can be eliminated by \eqref{426a}. 
Here we need to deal with the third term in \eqref{ru3}. It follows from \eqref{boundaryrhouth} that
%that if $k=i$, then $\pa_{x_n}\th=0$ on $\Gamma_i$ and if $k\neq i$, then 
$u_i=0$ on $\Gamma_i$ for $i=2,3$. Then by taking integration by parts, 
\begin{align*}
	\int_{\Omega}u\cdot\na_x\pa_{x_n}\wt\th\,\frac{\pa_{x_n}\wt\th}{\th}\,dx
	= -\frac{1}{2}\int_{\Omega}\na_x\cdot\big(\frac{u}{\th}\big)|\pa_{x_n}\wt\th|^2\,dx, 
\end{align*}
which becomes a typical trilinear term.

\smallskip
Next we deal with the typical trilinear terms in \eqref{ru1}, \eqref{ru2} and \eqref{ru3} as the followings. 
Applying boundary values \eqref{boundaryrhouth}, we can take integration by parts to obtain 
\begin{align*}
	\|\na_x(\wt\rho,\wt u,\wt\th)\|_{L^2_x}\le \|(\wt\rho,\wt u,\wt\th)\|^{\frac{1}{2}}_{L^2_x}\|\na^2_x(\wt\rho,\wt u,\wt\th)\|^{\frac{1}{2}}_{L^2_x}. 
\end{align*}
Then using \eqref{pat1} and \eqref{pat} for time derivatives and \eqref{GagL3}, \eqref{GagL6} for Sobolev embedding, we have 
%For any positive integer $1\le a_1,a_2,a_3\le 2$ such that $a_1+a_2\le 2$, the typical trilinear terms can be estimated as 
\begin{align}\label{e1}\notag
	&\quad\,\int_{\Omega}\big(|\na_{t,x}(\rho,u,\th)|+|\na_{x}(\bar\rho,\bar u,\bar\th)|\big)|\na_{x}(\wt\rho,\wt u,\wt\th)||\na_x(\wt\rho,\wt u,\wt\th)|\,dx\\
	&\notag\le \big(\|\na_{t,x}(\rho,u,\th)\|_{L^6_x}+\|\na_{x}(\bar\rho,\bar u,\bar\th)\|_{L^6_x}\big)\|\na_{x}(\wt\rho,\wt u,\wt\th)\|_{L^3_x}\|\na_x(\wt\rho,\wt u,\wt\th)\|_{L^2_x}\\
	&\notag\lesssim \big(\|\na_x(\wt\rho,\wt u,\wt\th)\|_{H^1_x}+\|\mu^{1/2}\na_x\g\|_{H^1_{x}L^2_5}+\delta^{\frac{1}{3}}(1+t)^{-\frac{2}{3}}+\sqrt{\E_k(t)\D_k(t)}\big)\\&\notag\qquad\times\|\na_x(\wt\rho,\wt u,\wt\th)\|^2_{H^1_x}\\
	&\lesssim \sqrt{\E_k(t)}\D_k(t) +\delta^{\frac{2}{3}}(1+t)^{-\frac{4}{3}},  
%	\\
%	&\le (1+t)^{-\frac{4}{3}}\|\na_{t,x}(\rho,u,\th)\|^2_{L^2_x}+(1+t)^{\frac{4}{3}}\|\na_x(\rho,u,\th,\g)\|_{L^2_x}\|\na^2_x(\rho,u,\th)\|_{L^2_x}\|\na^2_x(\rho,u,\th,G)\|_{L^2_x}^2,
\end{align}
and 
\begin{align}\label{e11}\notag
	&\quad\,\int_{\Omega}|(\wt\rho,\wt u,\wt\th)||\na^2_{x}(\bar\rho,\bar u,\bar\th)||\na_x(\wt\rho,\wt u,\wt\th)|\,dx\\
	&\notag\le \|(\wt\rho,\wt u,\wt\th)\|_{L^6_x}\|\na^2_{x}(\bar\rho,\bar u,\bar\th)\|_{L^2_x}\|\na_x(\wt\rho,\wt u,\wt\th)\|_{L^3_x}\\
	&\notag\lesssim \delta^{\frac{1}{3}}(1+t)^{-\frac{2}{3}}\|(\wt\rho,\wt u,\wt\th)\|_{H^1_x}\|\na_x(\wt\rho,\wt u,\wt\th)\|_{H^1_x}\\
	&\lesssim \sqrt{\E_k(t)}\D_k(t) +\delta^{\frac{2}{3}}(1+t)^{-\frac{4}{3}}.
\end{align}
%where we let $k\ge -\gamma/2$. 
%where we apply \eqref{GagL3} to estimate $L^3_x$ norm and \eqref{344} to estimate $\ol G$. Note that $\E_{k_0}(t)$ contains only spatial derivatives on $(\wt\rho,\wt u,\wt\th)$ from first to forth order. 
%%To control the time derivative, we deduce from \eqref{ruth} and \eqref{344} that 
%%\begin{align*}
%%	\|\pa_t(\rho,u,\th)\|^2_{L^2_x}
%%	&\lesssim \|\na_x(\rho,u,\th)\|^2_{L^2_x} + \|\na_x\phi\|^2_{L^3_x}\|F_2\|^2_{L^6_xL^2_3} + \|\na_x\ol G\|^2_{L^2_xL^2_5} + \|\na_x\g\|^2_{L^2_xL^2_5}\\
%%	&\lesssim \E_k(t) + \E_{k_0}(t)\D_k + \delta^{\frac{2}{3}}(1+t)^{-\frac{4}{3}}. 
%%\end{align*}
%%Inserting this into \eqref{430}, we obtain 
%%\begin{multline}\label{e1}
%%	\int_{\Omega}|\na_{t,x}(\rho,u,\th)||\na_x(\rho,u,\th)|^2\,dx
%%	\lesssim (1+t)^{-\frac{3}{2}}\E_k(t) + \E_{k_0}(t)\D_k + \delta^{\frac{2}{3}}(1+t)^{-\frac{4}{3}}\\ + (1+t)^{\frac{3}{2}}\sqrt{\E_{k_0}(t)\E_{k_0}(t)}\|\na^2_x(\rho,u,\th,G)\|_{L^2_x}^2. 
%%\end{multline}
The trilinear terms including $\na_x\phi$ and $F_2$ can be estimated as 
\begin{align*}
%	\label{e2}\notag
	&\quad\,\int_{\Omega}|\na_x(\wt u,\wt\th)|\big(|\na_x\rho||\na_x\phi||F_2|_{L^2_3}+|\na^2_x\phi||F_2|_{L^2_3}+|\na_x\phi||\na_xF_2|_{L^2_3}\big)\,dx\\
	&\le\notag \|\na_x(\wt u,\wt\th)\|_{L^2_x}\big(\|\na_x\rho\|_{L^2_x}\|\na_x\phi\|_{L^\infty_x}\|F_2\|_{L^6_xL^2_3}+\|\na^2_x\phi\|_{L^2_x}\|F_2\|_{L^\infty_xL^2_3}+\|\na_x\phi\|_{L^6_x}\|\na_xF_2\|_{L^3_xL^2_3}\big)\\
	&\lesssim\sqrt{\E_k(t)}\D_k(t). 
\end{align*}
%Next we deal with the last terms in \eqref{ru2} and \eqref{ru3} when $\pa_{x_n}$ is applied on $G$. 
For the trilinear terms induced from the last terms of \eqref{ru2} and \eqref{ru3}, we have from \eqref{344} that 
\begin{align*}
%	\label{e3}\notag
	&\quad\,\int_{\Omega}|\na_x(\wt\rho,\wt u,\wt\th)||\na_x(\rho, u,\th)||\na_xG|\,dx\\
	&\notag\le \|\na_x(\wt\rho,\wt u,\wt\th)\|_{L^6_x}\|\na_x(\rho, u,\th)\|_{L^3_x}\big(\|\na_x\ol G\|_{L^2_x}+\|\mu^{1/2}\na_x\g\|_{L^2_x}\big)\\
	&\notag\lesssim \|\na_x(\wt\rho,\wt u,\wt\th)\|_{H^1_x}\big(\|\na_x(\wt\rho,\wt u,\wt\th)\|_{H^1_x}+\delta^{\frac{1}{3}}(1+t)^{-\frac{2}{3}}\big)\\
	&\notag\qquad\times\big(\delta^{\frac{1}{3}}(1+t)^{-\frac{2}{3}}+\|\mu^{1/2}\na_x\g\|_{L^2_x}\big)\\
	&\lesssim \big(\delta^{\frac{1}{3}}+\sqrt{\E_k(t)}\big)\D_k(t)+\delta^{\frac{2}{3}}(1+t)^{-\frac{4}{3}}. 
\end{align*}
Recalling \eqref{G} that 
\begin{align}\label{444}
	G &= L_M^{-1}\big(P_1v\cdot\na_xM\big) + L_M^{-1}\Theta,
\end{align}
where  
%\begin{align*}
	$\Theta = \pa_tG +P_1v\cdot\na_xG-P_1\na_x\phi\cdot\na_vF_2- 2Q(G,G)$.
%\end{align*}
Recalling \eqref{P1M}, similar to \eqref{olG2}, a direct calculation gives 
\begin{align*}
	%	P_1v\cdot\na_xM &=
	P_1v\cdot\na_xM
	= \frac{\sqrt{R}}{\sqrt{\th}}\sum_{i=1}^3\pa_{x_i}\th \wh A_i\big(\frac{v-u}{\sqrt{R\th}}\big)M + \sum_{i,j=1}^3\pa_{x_i}u_j\wh B_{ij}\big(\frac{v-u}{\sqrt{R\th}}\big)M, 
\end{align*}
and hence, 
\begin{align*}
	%	P_1v\cdot\na_xM &=
	L_M^{-1}P_1v\cdot\na_xM
	= \frac{\sqrt{R}}{\sqrt{\th}}\sum_{i=1}^3\pa_{x_i}\th  A_i\big(\frac{v-u}{\sqrt{R\th}}\big) + \sum_{i,j=1}^3\pa_{x_i}u_j B_{ij}\big(\frac{v-u}{\sqrt{R\th}}\big). 
\end{align*} 
Using the fact from \cite[pp. 31]{Wang2019a} or \cite[pp. 51]{Duan2021}: 
\begin{align*}
	\big(B_{ij}\big(\frac{v-u}{\sqrt{R\th}}\big),\wh B_{ml}\big(\frac{v-u}{\sqrt{R\th}}\big)\big)_{L^2_v} = -\frac{\mu(\th)}{R\th}\big(\delta_{im}\delta_{jl}+\delta_{il}\delta_{jm}-\frac{2}{3}\delta_{ij}\delta_{ml}\big),
\end{align*}
and 
\begin{align*} 
	\big(A_i\big(\frac{v-u}{\sqrt{R\th}}\big),\wh A_j\big(\frac{v-u}{\sqrt{R\th}}\big)\big)_{L^2_v} = -\frac{\kappa(\th)}{R^2\th}\delta_{ij},
	\quad\big(A_i\big(\frac{v-u}{\sqrt{R\th}}\big),\wh B_{jm}\big(\frac{v-u}{\sqrt{R\th}}\big)\big)_{L^2_v} = 0,  
\end{align*}
%Then using \eqref{mu}, \eqref{kappa} and Lemma \ref{LemBur}, 
we have 
\begin{align*}
	\int_{\R^3}v_m\, v\cdot\na_x L_M^{-1}P_1v\cdot\na_xM\,dv
%	&= \sum_{l=1}^3\pa_{x_l}\Big\{R\th\sum_{i,j=1}^3\pa_{x_i}u_j \big(B_{ij}\big(\frac{v-u}{\sqrt{R\th}}\big),\wh B_{ml}\big(\frac{v-u}{\sqrt{R\th}}\big)\big)_{L^2_v}\Big\}\\
	&= -\sum_{l=1}^3\pa_{x_l}\Big\{\mu(\th)\big(\pa_{x_m}u_l+\pa_{x_l}u_m-\frac{2}{3}\delta_{ml}\na_x\cdot u\big)\Big\},
\end{align*}
and 
\begin{align*}
	&\int_{\R^3}\frac{|v|^2}{2} v\cdot\na_x L_M^{-1}P_1v\cdot\na_xM\,dv\\
%	&= \sum_{l=1}^3\pa_{x_l}\Big\{{R^2\th}\sum_{i=1}^3\pa_{x_i}\th  \big(A_i\big(\frac{v-u}{\sqrt{R\th}}\big),\wh A_l\big(\frac{v-u}{\sqrt{R\th}}\big)\big)_{L^2_v}\\&\qquad\qquad + \sum_{i,j,m=1}^3\pa_{x_i}u_j \,u_m R\th\big(B_{ij}\big(\frac{v-u}{\sqrt{R\th}}\big),\wh B_{ml}\big(\frac{v-u}{\sqrt{R\th}}\big)\big)_{L^2_v}\Big\}\\
	&\quad= -\sum_{l=1}^3\pa_{x_l}\Big\{\kappa(\th)\pa_{x_l}\th + \sum_{m=1}^3\mu(\th)u_m\big(\pa_{x_m}u_l+\pa_{x_l}u_m-\frac{2}{3}\delta_{ml}\na_x\cdot u\big)\Big\}, 
\end{align*}
Note that $\kappa(\th)$ and $\mu(\th)$ are smooth in $\th$ and hence, 
%there exist $C_1,C_2>0$ such that 
%$\kappa(\th),\mu(\th)$ and their first order derivatives $\kappa'(\th),\mu'(\th)$ are bounded from above and below, i.e.  
$\kappa(\th),\mu(\th),\kappa'(\th),\mu'(\th)\in [C_1,C_2]$, for some $C_1,C_2>0$. 
% for some $\kappa_1,\kappa_2>0$. 
Then the bilinear terms induced from the last terms in \eqref{ru2} and \eqref{ru3} with respect to the first term in \eqref{444}, 
after taking summation over $1\le n\le 3$, 
can be represented as 
\begin{align*}
%	\label{445}\notag
&\quad\,-\sum_{n,m=1}^3\int_{\Omega}\frac{\pa_{x_n}\wt u_m}{\rho}\pa_{x_n}\int_{\R^3}v_m\, v\cdot\na_x L_M^{-1}\big(P_1v\cdot\na_xM\big)\,dvdx\\
&\notag\le\sum_{n,l,m=1}^3\int_{\Omega}\frac{\mu(\th)}{\rho}\pa_{x_n}\wt u_m\pa_{x_n}\pa_{x_l}\big(\pa_{x_m}\wt u_l+\pa_{x_l}\wt u_m-\frac{2}{3}\delta_{ml}\na_x\cdot\wt u\big)\,dx\\&\qquad\notag+C\int_{\Omega}|\na_x^2\wt u||\na_x^2\bar u|\,dx+C\int_{\Omega}|\na_x\wt u||\na^2_xu||\na_x\th|\,dx\\
&\notag\le\sum_{n,m=1}^3\int_{\Omega}\frac{\mu(\th)}{\rho}\pa_{x_n}\wt u_m\pa_{x_n}\big(\frac{1}{3}\pa_{x_m}\na_x\cdot\wt u+\Delta_x\wt u_m\big)\,dx
+C\|\na_x^2\wt u\|_{L^2_x}\delta^{\frac{1}{3}}(1+t)^{-\frac{2}{3}}
\\\notag&\qquad+C\|\na_x\wt u\|_{L^6_x}\big(\|\na^2_x\wt u\|_{L^2_x}+\delta^{\frac{1}{3}}(1+t)^{-\frac{2}{3}}\big)\big(\|\na_x\wt\th\|_{L^3_x}+\delta^{\frac{1}{3}}(1+t)^{-\frac{2}{3}}\big)\\
&\notag\le-\sum_{n=1}^3\int_{\Omega}\frac{\mu(\th)}{\rho}\big(\frac{1}{3}\|\pa_{x_n}\na_x\cdot u\|_{L^2_x}^2  + \|\pa_{x_n}\na_xu\|_{L^2_x}^2\big)\,dx+\eta\|\na^2_x\wt u\|_{L^2_x}^2\\&\qquad+\big(\delta^{\frac{1}{3}}+C\sqrt{\E_k(t)}\big)\D_k(t) + C_\eta\delta^{\frac{2}{3}}(1+t)^{-\frac{4}{3}},  
\end{align*}
and 
\begin{align*}
%	\label{445a}\notag
	 &\quad\,-\sum_{n=1}^3\int_{\Omega}\frac{\pa_{x_n}\wt\th}{\th\rho}\pa_{x_n}\int_{\R^3}\frac{v\cdot(v-2u)}{2} v\cdot\na_x L_M^{-1}\big(P_1v\cdot\na_xM\big)\,dvdx\\
	&\notag\le \sum_{n,l=1}^3\int_{\Omega}\frac{\pa_{x_n}\wt\th}{\rho\th}\pa_{x_l}\Big\{\kappa(\th)\pa_{x_nx_l}\th\Big\}\,dx +\sum_{n=1}^3\int_{\Omega}\frac{\mu(\th)\pa_{x_n}\wt\th}{\rho\th}\pa_{x_n}\Big\{\frac{|\na_xu+(\na_xu)^t|^2}{2}-\frac{2}{3}(\na_x\cdot u)^2\Big\}\,dx\\
	%		&\notag\qquad+C\int_{\Omega}\big(|\na_xu||\na^2_xu|+|\na_xu|^2\big)|\na_x(\rho,\th)|\,dx\\
	&\notag\le - \int_{\Omega}\frac{\kappa(\th)|\na_x^2\wt\th|^2}{\rho\th}\,dx + \int_\Omega\big(|\na_x(\rho,\th)||\na_x\wt\th||\na^2_x\th|+|\na^2_x\wt\th||\na_x^2\bar\th|+|\na_x\wt\th||\na_xu||\na^2_xu|\big)\,dx\\
	%		&\notag\qquad
	%		+C\int_{\Omega}\big(|\na_xu||\na^2_xu|+|\na_xu|^2\big)|\na_x(\rho,\th)|\,dx\\
	&\le - \int_{\Omega}\frac{\kappa(\th)|\na_x^2\wt\th|^2}{\rho\th}\,dx +\eta\|\na^2_x\wt \th\|_{L^2_x}^2+ \big(\delta^{\frac{1}{3}}+\sqrt{\E_{k_0}(t)}\big)\D_k(t) + C_\eta\delta^{\frac{2}{3}}(1+t)^{-\frac{4}{3}}, 
\end{align*}
for any $\eta>0$, 
where we apply integration by parts by using zero boundary values from \eqref{boundaryrhouth} for the case of rectangular duct and use the fact that 
%We used \eqref{eta} to estimate $|u|$ and \eqref{540} to estimate trilinear terms. 
\begin{align*}
	&\quad\,\int_\Omega\big(|\na_x(\rho,u,\th)||\na_x\wt\th||\na^2_x(u,\th)|+|\na^2_x\wt\th||\na_x^2\bar\th|+|\na_x\wt\th||\na_xu||\na^2_xu|\big)\,dx\\
	&\lesssim \|\na_x(\rho,u,\th)\|_{L^3_x}\|\na_x\wt\th\|_{L^6_x}\|\na^2_x(u,\th)\|_{L^2_x}+\|\na^2_x\wt\th\|_{L^2_x}\|\na_x^2\bar\th\|_{L^2_x}\\
	&\lesssim \big(\|\na_x(\wt\rho,\wt u,\wt\th)\|_{H^1_x}+\delta^{\frac{1}{3}}(1+t)^{-\frac{2}{3}}\big)\|\na_x\wt\th\|_{H^1_x}\big(\|\na_x^2(\wt u,\wt\th)\|_{L^2_x}+\delta^{\frac{1}{2}}(1+t)^{-\frac{2}{3}}\big)\\&\qquad + \|\na^2_x\wt\th\|_{L^2_x}\delta^{\frac{1}{3}}(1+t)^{-\frac{2}{3}}\\
	&\lesssim \big(\delta^{\frac{1}{3}}+\sqrt{\E_{k_0}(t)}\big)\D_k(t) + C_\eta\delta^{\frac{1}{3}}(1+t)^{-\frac{4}{3}} +\eta\|\na^2_x\wt\th\|_{L^2_x}^2.
\end{align*}
Finally, we estimate the last terms in \eqref{ru2} and \eqref{ru3} with respect to the second term in \eqref{444}. 
Denote 
\begin{align*}
	\Theta_1 = \pa_tG +P_1v\cdot\na_xG,\qquad
	\Theta_2 = -P_1\na_x\phi\cdot\na_vF_2- 2Q(G,G). 
\end{align*}
Using boundary values from Lemma \ref{LemMacro}, we can obtain for $x\in \Gamma_n$ $(n=2,3)$ that 
\begin{align*}
	\pa_{x_n}\wt\th =\pa_{x_n}\wt u_m &=0, \ \ m\neq k,\\
	\int_{\R^3}v_m\, v\cdot\na_x L_M^{-1}\Theta_1\,dv &= 0, \ \ m= k. 
%	\int_{\R^3}\frac{v\cdot(v-2u)}{2} v\cdot\na_x L_M^{-1}\Theta\,dv\Big) &= 0. 
\end{align*}
Also, by duality on $P_1$ and $L_M^{-1}$, i.e. $(P_1f,\frac{g}{M})_{L^2_v}=(f,\frac{P_1g}{M})_{L^2_v}$ and $(L_M^{-1}f,\frac{g}{M})_{L^2_v}=(f,\frac{L_M^{-1}g}{M})_{L^2_v}$, we have 
\begin{align*}
	\int_{\R^3}v_m\, v\cdot\na_x L_M^{-1}\Theta_1\,dv&=\int_{\R^3}L_M^{-1}P_1(v_m\, v)M\cdot \frac{\na_x\Theta_1}{M}\,dv,\\ \int_{\R^3}\frac{v\cdot(v-2u)}{2} v\cdot\na_x L_M^{-1}\Theta_1\,dv
	&=\int_{\R^3}L_M^{-1}P_1\big(\frac{v\cdot(v-2u)}{2} vM\big)\cdot \frac{\na_x\Theta_1}{M}\,dv.   
\end{align*} 
It follows from \eqref{LM1a} that 
\begin{align*}
	|M^{-\frac{1}{2}}L_M^{-1}P_1(v_m\, v)M|_{L^2_v}+|M^{-\frac{1}{2}}L_M^{-1}P_1\big(\frac{v\cdot(v-2u)}{2} vM\big)|_{L^2_v}\le C<\infty. 
\end{align*}
Thus, we can take integration by parts about $\pa_{x_n}$ to obtain 
\begin{multline*}
%	\label{446}
	-\sum_{n,m=1}^3\int_{\Omega}\pa_{x_n}\wt u_m\pa_{x_n}\Big(\frac{1}{\rho}\int_{\R^3}v_m\, v\cdot\na_x L_M^{-1}\Theta\,dv\Big)\,dx\\+\sum_{n=1}^3\int_{\Omega}\frac{\pa_{x_n}\wt\th}{\th}\pa_{x_n}\Big(-\frac{1}{\rho}\int_{\R^3}\frac{v\cdot(v-2u)}{2} v\cdot\na_x L_M^{-1}\Theta\,dv\Big)\,dx\\
	\le \eta\|\na^2_x\wt u\|_{L^2_x}^2+\eta\|\na_x^2\wt\th\|_{L^2_x}^2 + C_\eta\|M^{-\frac{1}{2}}\na_x\Theta_1\|_{L^2_xL^2_v}^2 + C_\eta\|\<v\>^5\na_xL_M^{-1}\Theta_2\|_{L^2_xL^2_v}^2. 
\end{multline*}
%{\red give $L_M^{-1}$ to the other side, then $\g$ below has no weight.?}
Using \eqref{eta}, we have $R\th>1$ and thus $\<v\>^bM^{-\frac{1}{2}}\mu^{\frac{1}{2}}\le C_b$ for any $b>0$. 
Then applying \eqref{344}, we can deduce 
\begin{align*}
	\|M^{-\frac{1}{2}}\na_x\Theta_1\|_{L^2_xL^2_v}^2
	&\lesssim 
	\|\mu^{-\frac{1}{2}}(\pa_t\na_xG +P_1v\cdot\na^2_xG)\|^2_{L^2_xL^2_{\gamma/2}}\\
%	\lesssim \sum_{|\al|=1}\|\mu^{-\frac{1}{2}}\pa^\al\na_x G\|^2_{L^2_xL^2_{\gamma/2}}\\
	&\lesssim \delta^{\frac{2}{3}}(1+t)^{-\frac{4}{3}} + \sum_{|\al|=2}\|\pa^\al\g\|^2_{L^2_xL^2_{\gamma/2}}.  
%	+ \big(\delta^{\frac{1}{3}}+\sqrt{\E_k(t)}\big)\D^h_k(t), 
\end{align*} 
Applying \eqref{LM1a} and \eqref{2134}, we have 
\begin{align*}
	\|\<v\>^5\na_xL_M^{-1}\Theta_2\|_{L^2_xL^2_v}^2&\lesssim \|-\na_x(P_1\na_x\phi\cdot\na_vF_2)- 2\na_xQ(G,G)\|^2_{L^2_xH^{-s}_{-\gamma/2}}\\
	&\lesssim\|\na_x\phi\|^2_{H^2_x}\|\f\|^2_{H^2_xH^{s}_{\gamma/2}}+\|\nu\mu^{-\frac{1}{2}}G\|^2_{H^2_xL^2_v}\|\mu^{-\frac{1}{2}}G\|^2_{H^2_xL^2_D}\\
	&\lesssim \delta^{\frac{2}{3}}(1+t)^{-\frac{4}{3}} +\big(\delta^{\frac{1}{3}}+\sqrt{\E_k(t)}\big)\D_k(t).  
\end{align*}
%where we let $k\ge 1-\gamma$. 

\smallskip 
Taking linear combination $\eqref{ru1}+\eqref{ru2}+\eqref{ru3}$, letting $\eta>0$ sufficiently small and applying the above estimates, 
% \eqref{425}, \eqref{426}, \eqref{425a}, \eqref{426a}, \eqref{430}, \eqref{e2}, \eqref{445}, \eqref{445a}, \eqref{446} and \eqref{446a}, 
we obtain 
\begin{multline*}
	\pa_t\int_{\Omega}\Big(\frac{R\th}{2\rho^2}|\na_x\wt\rho|^2+\frac{1}{2}|\na_x\wt u|^2+\th^{-1}|\na_x\wt\th|^2\Big)\,dx 
	+\lam \|\na^2_x(\wt u,\wt\th)\|_{L^2_x}^2\\
	\lesssim \big(\delta^{\frac{1}{3}}+\sqrt{\E_k(t)}\big)\D_k(t) + \delta^{\frac{2}{3}}(1+t)^{-\frac{4}{3}}+ \sum_{|\al|=2}\|\pa^\al\g\|^2_{L^2_xL^2_{\gamma/2}}, 
\end{multline*}
for some generic constant $\lam>0$. The estimates for third-order dissipation terms can be obtained by using similar computations and one can deduce 
\begin{multline}\label{450}
	\pa_t\sum_{\substack{1\le|\al|\le2\\\al_0=0}}\int_{\Omega}\Big(\frac{R\th}{2\rho^2}|\pa^\al\wt\rho|^2+\frac{1}{2}|\pa^\al\wt u|^2+\th^{-1}|\pa^\al\wt\th|^2\Big)\,dx 
	+\lam \|\na^2_x(\wt u,\wt\th)\|_{L^2_x}^2+\lam \|\na^3_x(\wt u,\wt\th)\|_{L^2_x}^2\\
	\lesssim \big(\delta^{\frac{1}{3}}+\sqrt{\E_k(t)}\big)\D_k(t) + \delta^{\frac{2}{3}}(1+t)^{-\frac{4}{3}}+ \sum_{2\le|\al|\le3}\|\pa^\al\g\|^2_{L^2_xL^2_{\gamma/2}}, 
\end{multline}
% and used $\|\na^2_x(\bar u,\bar\th)\|_{L^2_x}^2\lesssim \delta^{\frac{1}{3}}(1+t)^{-\frac{7}{6}}$. 
%
%For high-order estimates, we let $1\le m\le 4$ and $|\al|=m$ such that $\al=(0,\al_1,\al_2,\al_3)$. Then applying $\pa^\al$ to system \eqref{ruth}, taking inner product of the resultant system with $\frac{R\th}{\rho^2}\pa^\al\rho$, $\pa^\al u$ and $\frac{\pa^\al\th}{\th}$ respectively and adding them together, similar to the calculation for deriving \eqref{450}, we can deduce 
%\begin{multline}\label{452}
%	\pa_t\int_{\Omega}\Big(\frac{R\th}{2\rho^2}|\na^m_x\rho|^2+\frac{1}{2}|\na^m_xu|^2+\th^{-1}|\na^m_x\th|^2\Big)\,dx 
%	+\lam \|\na^{m+1}_x(\wt u,\wt\th)\|_{L^2_x}^2\\
%	\lesssim \big(\eta_0+\delta^{\frac{1}{3}}+\sqrt{\E_k(t)}\big)\D_k(t) + \delta^{\frac{1}{3}}(1+t)^{-\frac{7}{6}}+ \sum_{|\al|=m+1}\|w(\al)\pa^\al \g\|^2_{L^2_xL^2_{\gamma/2}}. 
%\end{multline}
In order to obtain the estimate on $\wt\rho$ and time derivative, we take linear combination $\eqref{450}+\sum_{1\le|\al|\le2,\,|\al_0|=0}\kappa\times\eqref{rho}+\kappa^2\times\eqref{pat}$ with sufficiently small $\kappa>0$ to obtain 
\begin{multline*}
	\pa_t\E_{int}(t)  
	+\lam \sum_{2\le|\al|\le3}\|\pa^\al(\wt\rho,\wt u,\wt\th)\|_{L^2_x}^2\\
	\lesssim \big(\delta^{\frac{1}{3}}+\sqrt{\E_k(t)}\big)\D_k(t) + \delta^{\frac{2}{3}}(1+t)^{-\frac{4}{3}}+ \sum_{2\le|\al|\le3}\|\pa^\al \g\|^2_{L^2_xL^2_{\gamma/2}}, 
\end{multline*}
where $\E_{int}(t)$ is given by \eqref{Eint}. 
This completes the proof of Lemma \ref{Lem41}.

\end{proof}

\section{Estimates on Non-Fluid Quantities}\label{Sec5}

In this section, we will derive energy estimates on $\na_x\phi$, $\g$ and $\f$. 
Here we will modify the arguments in \cite{Duan2020a,Duan2021} to obtain the estimations for the physical boundary case.

\subsection{Estimates on $\f$ and $\na_x\phi$}
In this subsection, we derive the energy estimates on $\f=\mu^{-1/2}F_2$ and $\na_x\phi$.

\begin{Lem}\label{Lem53}
	Assume $\gamma>\max\{-3,-2s-\frac{3}{2}\}$ for Boltzmann case and $\gamma\ge 0$, $\frac{1}{2}\le s<1$ for VPB case. 
	Let $(F_\pm,\phi)$ be the solution to \eqref{1}. Then there exists functional $\E_{k,3}$ 
%	and $\E^h_{k,3}$ 
	satisfying 
	\begin{align}\label{Ek2}
		\E_{k,3}\approx \sum_{|\al|\le 3}\Big(\|w(\al)\pa^\al \f\|^2_{L^2_{x,v}}+\|\pa^\al\na_x\phi\|_{L^2_x}^2\Big),
	\end{align}
%and 
%	\begin{align}\label{Ehk2}
%		\E^h_{k,3}\approx \sum_{1\le|\al|\le 3}\Big(\|w(\al)\pa^\al \f\|^2_{L^2_{x,v}}+\|\pa^\al\na_x\phi\|_{L^2_x}^2\Big), 
%	\end{align}
such that 
\begin{multline}\label{f1}
	\pa_t\E_{k,3} + \lam\sum_{|\al|\le 3}\|w(\al)\pa^\al\f\|^2_{L^2_xL^2_D}+
	\lam\sum_{|\al|\le 3}\big(\|\pa^\al a\|_{L^2_x}^2+\|\pa^\al\na_x\phi\|_{L^2_x}^2\big)
	\\\lesssim \delta(1+t)^{-2}+ 
	\big(\eta_0+\delta^{\frac{1}{3}}+\sqrt{\E_k(t)}\big)\D_k(t). 
\end{multline}
%and 
%\begin{multline}\label{fh}
%	\pa_t\E^h_{k,3} + \lam\sum_{1\le|\al|\le 3}\|w(\al)\pa^\al\f\|^2_{L^2_xL^2_D}+
%	\lam\sum_{1\le|\al|\le 3}\big(\|\pa^\al a\|_{L^2_x}^2+\|\pa^\al\na_x\phi\|_{L^2_x}^2\big)
%	\\
%	\lesssim \delta(1+t)^{-2} + 
%	\big(\eta_0+\delta^{\frac{1}{3}}+\sqrt{\E_k(t)}\big)\D_k(t). 
%\end{multline}
\end{Lem}
\begin{proof}
	We proceed by considering macroscopic and microscopic components $P_{\mu,2}\f$ and $\{I-P_{\mu,2}\}\f$. 
	
	\smallskip\noindent{\bf Step 1. Macroscopic estimates on $P_{\mu,2}\f$.}
	We first derive the time derivative estimates. Splitting $\f = a\mu^{1/2} + \{I-P_{\mu,2}\}\f$, we rewrite \eqref{F2} as  
	\begin{multline}\label{556}
		\pa_ta\mu^{1/2}+v\cdot\na_xa\mu^{1/2}+\na_x\phi\cdot v\mu^{1/2} - \frac{\na_x\phi\cdot\na_v(M-\mu+G)}{\sqrt\mu} \\= - (\pa_t+v\cdot\na_x)\{I-P_{\mu,2}\}\f+\L_2\f+\Gamma\big(\frac{M-\mu}{\sqrt\mu},\f\big) + \Gamma\big(\frac{G}{\sqrt\mu},\f\big).
	\end{multline}
Note that $(\Gamma(f,g),\mu^{1/2})_{L^2_v}=0$. 
	Taking inner product of \eqref{556} with $\mu^{1/2}$ over $\R^3_v$, we have 
	\begin{align}\label{557}
		\pa_ta = -\big(v\cdot\na_x\{I-P_{\mu,2}\}\f,\mu^{1/2}\big)_{L^2_v}.
	\end{align}
	For any $|\al|\le 2$, we apply $\pa^\al$ to \eqref{557} and deduce 
	\begin{align*}
%		\label{540}
		\|\pa^\al\pa_t a\|_{L^2_x}^2 \lesssim \|\pa^\al\na_x\{I-P_{\mu,2}\}\f\|_{L^2_xL^2_D}^2.
%		 C\|\pa^\al\{I-P_{\mu,2}\}\f\|_{L^2_xL^2_3}^2. 
%		+\eta\|\pa^\al\pa_t a\|^2_{L^2_x},
	\end{align*}
%for $\eta>0$. 
%	Applying $\pa^\al$ to \eqref{557} yields 
%	\begin{align}\label{550}
%		\|\pa_t\pa^\al a\|_{L^2_x} \le \|\pa^\al\na_x\{I-P_{\mu,2}\}\f\|_{L^2_xL^2_3}. 
%	\end{align}
%	For any $|\al|\le 2$, a
	Applying $\pa^\al$ to \eqref{557} and taking inner product with $\pa_t\pa^\al\phi$ over $\Omega\times\R^3$, we have 
	\begin{equation}
		\begin{aligned}\label{550q}
			(\pa_t\pa^\al a,\pa_t\pa^\al\phi)_{L^2_x} = -\big(v\cdot\na_x\pa^\al\{I-P_{\mu,2}\}\f,\mu^{1/2}\pa_t\pa^\al\phi\big)_{L^2_xL^2_v}. 
%			(\pa_t\pa^\al a,\pa^\al\phi)_{L^2_x} = -\big(v\cdot\na_x\pa^\al\{I-P_{\mu,2}\}\f,\pa^\al\phi\big)_{L^2_xL^2_v}. 
		\end{aligned}
	\end{equation}
Here, we will take integration by parts on \eqref{550q} and prepare the following boundary values for the case of rectangular duct. 
	Using Lemma \ref{Lemspecular} and similar to \eqref{P0h}, we have $\pa^\al\{I-P_{\mu,2}\}\f(R_xv)=(-1)^{|\al_i|}\pa^\al\{I-P_{\mu,2}\}\f(v)$ on $\Gamma_i$, for $i=2,3$. Then by change of variable $v\mapsto R_xv$, when $\al_i=0,2$ for some $i=2,3$, we have for $x\in\Gamma_i$ that 
	\begin{align*}
%		\label{F222}\notag
		(v_i\pa^\al\{I-P_{\mu,2}\}\f,\mu^{1/2})_{L^2_v} &= ((R_xv)_i\pa^\al\{I-P_{\mu,2}\}\f(R_xv),\mu^{1/2})_{L^2_v}
	\\&	= -(v_i\pa^\al\{I-P_{\mu,2}\}\f,\mu^{1/2})_{L^2_v} =0.  
	\end{align*}
Also, using boundary values \eqref{Neumann}, if $\al_i=1,3$ for some $i=2,3$, then 
\begin{equation*}
%	\label{boundaryphi}
	\begin{aligned}
		\pa^{\al}\phi&=0\ \ 
%		\text{ or } \pa^{\al}\phi = 0\ 
		\text{ on }\Gamma_i. 
%		\pa_{x_ix_jx_n}\phi&=0\text{ or } \pa_{x_jx_n}\phi = 0\ \text{ on }\Gamma_i,\\
%		\pa_{x_ix_jx_nx_m}\phi&=0\text{ or } \pa_{x_jx_nx_m}\phi = 0\ \text{ on }\Gamma_i,
	\end{aligned}
\end{equation*}
	In view of these zero boundary values, we can take integration by parts in \eqref{550q} to obtain 
	\begin{align*}
		(\pa_t\pa^\al a,\pa_t\pa^\al\phi)_{L^2_x} = -(\pa_t\pa^\al \Delta_x\phi,\pa_t\pa^\al\phi)_{L^2_x}
		= \|\pa_t\pa^\al\na_x\phi\|_{L^2_x}^2,
%		(\pa_t\pa^\al a,\pa^\al\phi)_{L^2_x} = -(\pa_t\pa^\al \Delta_x\phi,\pa^\al\phi)_{L^2_x}
%		= \frac{1}{2}\pa_t\|\pa^\al\na_x\phi\|_{L^2_x}^2,
	\end{align*}
	and 
	\begin{align*}
		-\big(v\cdot\na_x\pa^\al\{I-P_{\mu,2}\}\f,\mu^{1/2}\pa^\al\phi\big)_{L^2_xL^2_v}
		= \big(\pa^\al\{I-P_{\mu,2}\}\f,\mu^{1/2}v\cdot\pa_t\pa^\al\na_x\phi\big)_{L^2_xL^2_v}. 
%		-\big(v\cdot\na_x\pa^\al\{I-P_{\mu,2}\}\f,\pa^\al\phi\big)_{L^2_xL^2_v}
%		= \big(\pa^\al\{I-P_{\mu,2}\}\f,v\cdot\pa^\al\na_x\phi\big)_{L^2_xL^2_v}, 
	\end{align*}
Thus, \eqref{550q} yields  
	\begin{equation*}
			\|\pa_t\pa^\al\na_x\phi\|^2_{L^2_x} \lesssim \|\pa^\al\{I-P_{\mu,2}\}\f\|^2_{L^2_xL^2_D}, 
%			\frac{1}{2}\pa_t\|\pa^\al\na_x\phi\|^2_{L^2_x} &\le C_\eta\|\pa^\al\{I-P_{\mu,2}\}\f\|_{L^2_xL^2_3}^2 
%			+ \eta\|\pa^\al\na_x\phi\|_{L^2_x}^2. 
	\end{equation*}
and hence, 
\begin{align}\label{550a}
	\|\pa_t^3\na_x\phi\|^2_{L^2_x} &\lesssim \|\pa^2_t\{I-P_{\mu,2}\}\f\|^2_{L^2_xL^2_D}.
\end{align}
	
	To derive the spatial derivative estimate, 
we take inner product of \eqref{556} with $v\mu^{1/2}$ over $\R^3_v$ 
%and apply integration by parts 
to obtain 
	\begin{align*}\notag
		\na_xa+\na_x\phi &= \big(\na_x\phi\cdot\na_v(M-\mu+G),v\mu^{1/2}\big)_{L^2_v}\\&\quad - \big((\pa_t+v\cdot\na_x)\{I-P_{\mu,2}\}\f+\L_2\f+\Gamma\big(\frac{M-\mu}{\sqrt\mu},\f\big) + \Gamma\big(\frac{G}{\sqrt\mu},\f\big),v\mu^{1/2}\big)_{L^2_v}. 
	\end{align*}
Note that $G$ is microscopic (satisfying \eqref{micro}). Then 
\begin{align*}
		\big(\na_x\phi\cdot\na_v(M-\mu+G),v\big)_{L^2_v}&\notag= -\big(\na_x\phi(M-\mu+G),1\big)_{L^2_v}= -(\rho-1)\na_x\phi,  
	\end{align*}
and hence, 
\begin{align}\label{558}
	\na_xa+\rho\na_x\phi = - \big((\pa_t+v\cdot\na_x)\{I-P_{\mu,2}\}\f+\L_2\f+\Gamma\big(\frac{M-\mu}{\sqrt\mu},\f\big) + \Gamma\big(\frac{G}{\sqrt\mu},\f\big),v\mu^{1/2}\big)_{L^2_v}. 
\end{align}
%
%Splitting $F_1=\mu+M-\mu+\ol G+\sqrt\mu\g$, we have  
%\begin{align}\label{558}
%	\na_xa+\rho\na_x\phi = -\big((\pa_t+v\cdot\na_x)\{I-P_{\mu,2}\}\f+L_{\mu,2}F_2 + Q(M-\mu+\ol G+\sqrt\mu\g,F_2),v\big)_{L^2_v}. 
%\end{align}
%	Using Lemma \ref{Lemspecular} and \eqref{435}, we have $\pa^\al\{I-P_{\mu,2}\}\f(R_xv)=(-1)^{|\al_i|}\pa^\al\{I-P_{\mu,2}\}\f(v)$ on $\Gamma_i$, for $i=2,3$. Then by change of variable $v\mapsto R_xv$, when $\al_i=0,2,4$ for some $i=2,3$, we have for $x\in\Gamma_i$ that 
%\begin{align}\label{F222}
%	(v_i\pa^\al\{I-P_{\mu,2}\}\f,1)_{L^2_v} = ((R_xv)_i\pa^\al\{I-P_{\mu,2}\}\f(R_xv),1)_{L^2_v}
%	= -(v_i\pa^\al\{I-P_{\mu,2}\}\f,1)_{L^2_v} =0.  
%\end{align}
%Also, using boundary values \eqref{Neumann} and \eqref{boundphi}, we have for any $|\al'|\le 4$  and $i=2,3$ that
%\begin{equation}\label{boundaryphi}
%		\begin{aligned}
%				\pa_{x_i}\pa^{\al'}\phi&=0\text{ or } \pa^{\al'}\phi = 0\ \text{ on }\Gamma_i. 
%		%		\pa_{x_ix_jx_n}\phi&=0\text{ or } \pa_{x_jx_n}\phi = 0\ \text{ on }\Gamma_i,\\
%		%		\pa_{x_ix_jx_nx_m}\phi&=0\text{ or } \pa_{x_jx_nx_m}\phi = 0\ \text{ on }\Gamma_i,
%			\end{aligned}
%	\end{equation}
	Similar to the calculations for deriving \eqref{boundaryrhouth}, one can obtain 
	\begin{align*}
		\pa_{x_ix_ix_i}a=\pa_{x_i}a=0\ \ \text{ on $\Gamma_i$ for $i=2,3$. }
	\end{align*} 
Together with \eqref{Neumann}, we can take integration by parts and apply the third equation of \eqref{F1} to obtain  
	\begin{align}\label{54}
		(\pa^\al\na_x\phi,\pa^\al\na_xa)_{L^2_x} = -(\pa^\al\Delta_x\phi,\pa^\al a)_{L^2_x} = \|\pa^\al a\|_{L^2_x}^2 = \frac{1}{2}\|\pa^\al\Delta_x\phi\|_{L^2_x}^2=\frac{1}{2}\|\pa^\al\na_x^2\phi\|_{L^2_x}^2, 
	\end{align} 
	For $|\al|\le 2$, applying $\pa^\al$ to \eqref{558} and taking inner product with $\pa^\al\na_xa$, we have  
	\begin{multline}\label{550b}
		\|\pa^\al\na_xa\|_{L^2_x}^2+\|\pa^\al a\|_{L^2_x}^2 \lesssim 
%		-\pa_t(\{I-P_{\mu,2}\}\f,v\cdot\pa^\al\na_xa)_{L^2_{x,v}} + (v\cdot\na_x\{I-P_{\mu,2}\}\f,\pa^\al\pa_ta)_{L^2_{x,v}}\\
%		+ \big(v\cdot\na_x\{I-P_{\mu,2}\}\f+L_{\mu,2}F_2 + Q(M-\mu+\ol G+\sqrt\mu\g,F_2),v\cdot\pa^\al\na_xa\big)_{L^2_v},  
%		 \|\rho-1\|_{L^\infty_x}\|\pa^\al\na_x\phi\|_{L^2_x}\|\pa^\al\na_xa\|_{L^2_x}\\&\notag\qquad +
 \sum_{\substack{|\al'|\le1}}\|\pa^{\al'}\pa^\al\{I-P_{\mu,2}\}\f\|_{L^2_xL^2_D}^2 + \big(\eta_0+\delta^{\frac{1}{2}}+\sqrt{\E_k(t)}\big)\D_k(t),  
% \\
%		&\notag\le \eta_0\D_k(t) + \eta\sum_{1\le|\al_1|\le 2}\|w(\al)\pa^{\al_1}\{I-P_{\mu,2}\}\f\|^2_{L^2_xL^2_{\gamma/2}}
%		\\&\qquad+C_\eta\sum_{|\al'|=1}\|\pa^{\al'}\pa^\al\{I-P_{\mu,2}\}\f\|^2_{L^2_xL^2_v} + \eta\|\pa^\al\na_xa\|^2_{L^2_x}. 
	\end{multline}
where we use \eqref{332a}, \eqref{33}, \eqref{333}, \eqref{34} and \eqref{344} to estimate the collision terms and use the fact that when $1\le|\al|\le 2$, 
\begin{align}\label{514aa}
	\|\pa^\al(\rho\,\na_x\phi)\|^2_{L^2_x}
%	&\lesssim \|\pa^{\al}\rho\|^2_{L^3_x}\|\na_x\phi\|_{L^6_x}^2
	\lesssim (\delta^{\frac{1}{3}}+\|\pa^\al\wt\rho\|_{H^2_x})^2\|\na_x\phi\|_{H^2_x}^2
	&\lesssim \big(\delta^{\frac{1}{2}}+\sqrt{\E_k(t)}\big)\D_k(t). 
\end{align}
Similarly, applying $\pa^\al$ to \eqref{558} and taking inner product with $\pa^\al\na_x\phi$, we have from \eqref{54}, \eqref{514aa} and \eqref{332a} that 
\begin{align}\label{550c}
	\|\pa^\al \na^2_x\phi\|_{L^2_x}^2+\lam\|\pa^\al\na_x\phi\|_{L^2_x}^2
	\lesssim \sum_{\substack{|\al'|\le1}}\|\pa^{\al'}\pa^\al\{I-P_{\mu,2}\}\f\|_{L^2_xL^2_D}^2 + \big(\eta_0+\delta^{\frac{1}{2}}+\sqrt{\E_k(t)}\big)\D_k(t). 
\end{align}
%
% from \eqref{550a} that 
%	\begin{multline}\label{550c}
%		\pa_t(\pa^\al\{I-P_{\mu,2}\}\f,v\cdot\pa^\al\na_x\phi)_{L^2_{x,v}}+
%		\lam\|\pa^\al \na^2_x\phi\|_{L^2_x}^2+\lam\|\pa^\al\na_x\phi\|_{L^2_x}^2 \\
%		\lesssim 
%		\eta\|\pa^\al\na_x\phi\|^2_{L^2_x} +  C_\eta\sum_{\substack{|\al'|=1,\,|\al'_0|=0}}\|\pa^{\al'}\pa^\al\{I-P_{\mu,2}\}\f\|_{L^2_xH^s_{k+\gamma/2}}^2 \\+ \big(\eta_0+\delta^{\frac{1}{3}}+\sqrt{\E_k(t)}\big)\sum_{\substack{|\al'|\le 4,\,|\al'_0|\le 2}}\|\pa^{\al'}F_2\|^2_{L^2_xH^s_{k+\gamma/2}}.  
%%		 \\
%%		&\notag\le \eta_0\D_k(t) + \eta\sum_{1\le|\al_1|\le 2}\|w(\al)\pa^{\al_1}\{I-P_{\mu,2}\}\f\|^2_{L^2_xL^2_{\gamma/2}}
%%		\\&\qquad+C_\eta\sum_{1\le|\al_1|\le 2}\|\pa^{\al_1}\{I-P_{\mu,2}\}\f\|^2_{L^2_xL^2_v} + \eta\|\pa^\al\na_x\phi\|^2_{L^2_x}.
%	\end{multline}
Taking linear combination $\eqref{550b}+\eqref{550c}+\eqref{550a}$ and summation over 
%$|\al|=1$ and 
$|\al|\le 2$, we obtain 
%\begin{align}\label{516}
%	\sum_{1\le|\al|\le 3}\big(\|\pa^\al a\|_{L^2_x}^2+\|\pa^\al\na_x\phi\|_{L^2_x}^2\big)
%	\lesssim \sum_{\substack{1\le|\al|\le 3}}\|\pa^\al\{I-P_{\mu,2}\}\f\|_{L^2_xL^2_D}^2 + \big(\eta_0+\delta^{\frac{1}{2}}+\sqrt{\E_k(t)}\big)\D_k(t),  
%\end{align}
%and 
\begin{align}\label{516a}
	\sum_{|\al|\le 3}\big(\|\pa^\al a\|_{L^2_x}^2+\|\pa^\al\na_x\phi\|_{L^2_x}^2\big)
	\lesssim \sum_{\substack{|\al|\le 3}}\|\pa^\al\{I-P_{\mu,2}\}\f\|_{L^2_xL^2_D}^2 + \big(\eta_0+\delta^{\frac{1}{2}}+\sqrt{\E_k(t)}\big)\D_k(t). 
\end{align}

\smallskip\noindent{\bf Step 2. Microscopic estimates on $\{I-P_{\mu,2}\}\f$.}
For $|\al|\le 3$, applying $\pa^\al$ to \eqref{F2} and taking inner product with $w^2(\al)\pa^\al \f$, we deduce 
\begin{multline}\label{518}
	\frac{1}{2}\pa_t\|w(\al)\pa^\al \f\|^2_{L^2_{x,v}} + \frac{1}{2}\int_{\pa\Omega}\int_{\R^3}v\cdot n|w(\al)\pa^\al \f|^2\,dvdS(x) + \big(\pa^\al\na_x\phi\cdot v\mu^{\frac{1}{2}},w^2(\al)\pa^\al \f\big)_{L^2_{x,v}} \\-  \big(\frac{\pa^\al(\na_x\phi\cdot\na_v(M-\mu+\ol G))}{\sqrt\mu},w^2(\al)\pa^\al \f\big)_{L^2_{x,v}} -  \big(\frac{\pa^\al(\na_x\phi\cdot\na_v(\sqrt\mu\g))}{\sqrt\mu},w^2(\al)\pa^\al \f\big)_{L^2_{x,v}}\\= (\L_2\pa^\al\f,w^2(\al)\pa^\al \f)_{L^2_{x,v}}
	+\big(\pa^\al\Gamma\big(\frac{M-\mu}{\sqrt\mu},\f\big) + \pa^\al\Gamma\big(\frac{G}{\sqrt\mu},\f\big),w^2(\al)\pa^\al \f\big)_{L^2_{x,v}}. 
\end{multline}
We denote the second to seventh terms in \eqref{518} to be $I_i$ $(1\le i\le 6)$. 
For the case of rectangular duct, applying Lemma \ref{Lemspecular}, we can obtain $\pa^\al \f(R_xv)=(-1)^{|\al_i|}\pa^\al \f(v)$ for $x\in\Gamma_i$ $(i=2,3)$. Then by change of variable $v\mapsto R_xv$, the boundary integral is 
\begin{align*}
	I_1&=\int_{\pa\Omega}\int_{\R^3}v\cdot n|w(\al)\pa^\al \f|^2\,dvdS(x)\\
	& = \int_{\pa\Omega}\int_{\R^3}R_xv\cdot n|w(\al)\pa^\al \f(R_xv)|^2\,dvdS(x)\\
	& = -\int_{\pa\Omega}\int_{\R^3}v\cdot n|w(\al)\pa^\al \f(v)|^2\,dvdS(x) =0. 
\end{align*}
For the case of torus, the boundary integral also vanishes due to periodic property. 
For $I_2$, we have 
\begin{align*}
	|I_2|\lesssim C_\eta\|\pa^\al\na_x\phi\|^2_{L^2_x} + \eta\|w(\al)\pa^\al\f\|_{L^2_{x}L^2_{\gamma/2}}^2. 
\end{align*}
The terms $I_3$ and $I_4$ only occur for the VPB system, in which case $1/2\le s<1$ and $\gamma\ge 0$. Then it follows from \eqref{33}, \eqref{333}, \eqref{34}, \eqref{344} and \eqref{vdv} that 
\begin{align*}
	|I_3|
%	&\lesssim \sum_{\al'\le\al}\big\|\frac{|\pa^{\al'}\na_x\phi||w(\al)\na_v\pa^{\al-\al'}(M-\mu+\ol G)|}{\sqrt\mu}\big\|_{L^2_xL^2_v}\|w(\al)\pa^\al \f\|_{L^2_{x,v}}\\
&\lesssim \sum_{\al'\le\al}\big\|\frac{|\pa^{\al'}\na_x\phi||w(\al)\na_v\pa^{\al-\al'}(M-\mu+\ol G)|}{\sqrt\mu}\big\|_{L^2_xL^2_v}\|w(\al)\pa^\al \f\|_{L^2_{x,v}}\\
	&\lesssim \Big(\sum_{|\al-\al'|=0}\|\pa^{\al'}\na_x\phi\|_{L^2_x}\|\frac{w(\al)\na_v\pa^{\al-\al'}(M-\mu+\ol G)}{\sqrt\mu}\|_{L^\infty_xL^2_v}\\
	&\quad+\sum_{|\al-\al'|=1}\|\pa^{\al'}\na_x\phi\|_{L^6_x}\|\frac{w(\al)\na_v\pa^{\al-\al'}(M-\mu+\ol G)}{\sqrt\mu}\|_{L^3_xL^2_v}\\
	&\quad+\sum_{2\le|\al-\al'|\le3}\|\pa^{\al'}\na_x\phi\|_{L^\infty_x}\|\frac{w(\al)\na_v\pa^{\al-\al'}(M-\mu+\ol G)}{\sqrt\mu}\|_{L^2_xL^2_v}\Big)\|w(\al)\pa^\al \f\|_{L^2_xL^2_D}\\
	&\lesssim \big(\eta_0+\delta^{\frac{1}{3}}\big)\|\pa^{\al}\na_x\phi\|_{L^2_x}\|w(\al)\pa^\al \f\|_{L^2_xL^2_D}\\&\qquad
	+\sum_{|\al'|\le 3}\|\pa^{\al'}\na_x\phi\|_{L^2_x}\big(\sqrt{\E_k(t)}+\delta^{\frac{1}{3}}\big)\|w(\al)\pa^\al \f\|_{L^2_xL^2_D}\\
	&\lesssim 
%	\left\{\begin{aligned}
%		&\big(\eta_0+\delta^{\frac{1}{3}}+\sqrt{\E_k(t)}\big)\D_k(t),\quad\text{ if }|\al|\ge 1,\\
%		&
		\big(\eta_0+\delta^{\frac{1}{3}}+\sqrt{\E_k(t)}\big)\D_k(t). 
%		,\quad\text{ if }|\al|=0. 
%	\end{aligned}\right. 
\end{align*}
Similarly, for $I_4$, noticing $\na_v(\sqrt\mu\g)=-\frac{v}{2}\sqrt\mu\g+\sqrt\mu\na_v\g$, we have 
\begin{align*}
	I_4&\lesssim\sum_{\al'\le\al}\big\||\pa^{\al'}\na_x\phi||w(\al)\big(\<v\>^{\frac{1}{2}}+\<D_v\>^s\big)\pa^{\al-\al'}\g|\big\|_{L^2_{x,v}}\|w(\al)(\<v\>^{\frac{1}{2}}+\<D_v\>^s)\pa^\al \f\|_{L^2_xL^2_v}\\
%	+|w(\al)\<D_v\>^s\pa^{\al-\al'}\g|||w(\al)\<D_v\>^s\pa^\al \f|\big)\|_{L^1_xL^2_{v}} \\
	&\lesssim \big(\eta_0+\delta^{\frac{1}{3}}+\sqrt{\E_k(t)}\big)\D_k(t). 
%	\left\{\begin{aligned}
%		&\big(\eta_0+\delta^{\frac{1}{3}}+\sqrt{\E_k(t)}\big)\D_k(t),\quad\text{ if }|\al|\ge 1,\\
%		&\big(\eta_0+\delta^{\frac{1}{3}}+\sqrt{\E_k(t)}\big)\D_k(t),\quad\text{ if }|\al|=0. 
%	\end{aligned}\right. 
\end{align*}
For $I_5$, we have from \eqref{L1} that 
\begin{align*}
	I_5\le -\lam\|w(\al)\pa^\al \f\|^2_{L^2_xL^2_D}+C\|\pa^\al \f\|^2_{L^2_xL^2_{B_C}}. 
\end{align*}
For $I_6$, it follows from \eqref{220b}
%, \eqref{220a}
 and \eqref{220} that 
\begin{align*}
	|I_6| \lesssim \big(\eta_0+\delta^{\frac{1}{3}}+\sqrt{\E_k(t)}\big)\D_k(t)+\delta(1+t)^{-2}. 
%	\left\{
%	\begin{aligned}
%		&\big(\eta_0+\delta^{\frac{1}{3}}+\sqrt{\E_k(t)}\big)\D_k(t)+\delta(1+t)^{-2},\quad\text{ if }|\al|\ge 1,\\
%		&\big(\eta_0+\delta^{\frac{1}{3}}+\sqrt{\E_k(t)}\big)\D_k(t)+\delta(1+t)^{-2},\quad\text{ if }|\al|=0. 
%	\end{aligned}\right.
\end{align*}
Combining the above estimates and choosing $\eta>0$ small enough, we deduce from \eqref{518} that 
\begin{multline}\label{517}
	\frac{1}{2}\pa_t\|w(\al)\pa^\al \f\|^2_{L^2_{x,v}} + \lam\|w(\al)\pa^\al\f\|^2_{L^2_xL^2_D} \lesssim \|\pa^\al\na_x\phi\|^2_{L^2_x} +\delta(1+t)^{-2}+\|\pa^\al\f\|^2_{L^2_xL^2_{B_C}}\\ + \big(\eta_0+\delta^{\frac{1}{3}}+\sqrt{\E_k(t)}\big)\D_k(t). 
%	\left\{
%	\begin{aligned}
%		&\big(\eta_0+\delta^{\frac{1}{3}}+\sqrt{\E_k(t)}\big)\D_k(t),\quad\text{ if }|\al|\ge 1,\\
%		&\big(\eta_0+\delta^{\frac{1}{3}}+\sqrt{\E_k(t)}\big)\D_k(t),\quad\text{ if }|\al|=0. 
%	\end{aligned}\right.
\end{multline}
Next we consider the energy estimate without weight. Applying $\pa^\al$ to \eqref{F2} and taking inner product with $\pa^\al\f$, we deduce 
\begin{multline}\label{518a}
	\frac{1}{2}\pa_t\|\pa^\al \f\|^2_{L^2_{x,v}} + \frac{1}{2}\int_{\pa\Omega}\int_{\R^3}v\cdot n|\pa^\al \f|^2\,dvdS(x) + \big(\pa^\al\na_x\phi\cdot v\mu^{\frac{1}{2}},\pa^\al \f\big)_{L^2_{x,v}} \\-  \big(\frac{\pa^\al(\na_x\phi\cdot\na_v(M-\mu+\ol G))}{\sqrt\mu},\pa^\al \f\big)_{L^2_{x,v}} -  \big(\frac{\pa^\al(\na_x\phi\cdot\na_v(\sqrt\mu\g))}{\sqrt\mu},\pa^\al \f\big)_{L^2_{x,v}}\\= (\L_2\pa^\al\f,\pa^\al \f)_{L^2_{x,v}}
	+\big(\pa^\al\Gamma\big(\frac{M-\mu}{\sqrt\mu},\f\big) + \pa^\al\Gamma\big(\frac{G}{\sqrt\mu},\f\big),\pa^\al \f\big)_{L^2_{x,v}}. 
\end{multline}
The estimates on second, forth, fifth and seventh terms of \eqref{518a} are similar to $I_1$, $I_3$, $I_4$ and $I_6$. Thus, we consider the third and sixth terms of \eqref{518a}. 
For the third term, taking inner product of \eqref{F2} with $\sqrt\mu$, we have 
\begin{align*}
	\pa_t\int_{\R^3}\f\mu^{\frac{1}{2}}\,dv + \na_x\cdot\int_{\R^3}v\f\mu^{\frac{1}{2}}\,dv = 0. 
\end{align*} 
Also, for $i=2,3$, if $\al_i=0$  or $2$, then we have from Lemma \ref{Lemspecular} and change of variable $v\mapsto R_xv$ that 
\begin{align*}
	\int_{\R^3}v_i\mu^{\frac{1}{2}}\pa^\al \f\,dv = 0, \ \ \text{ on }\Gamma_i. 
\end{align*}
If $\al_i=1$ or $3$, then $\pa^\al\phi=0$ on $\Gamma_i$, which follows from \eqref{Neumann} and \eqref{boundphi}. 
Then by integration by parts and using the third equation of \eqref{F1}, we have 
%using decomposition $F_2=P_{\mu}F_2+\{I-P_{\mu}\}F_2$ with $P_{\mu}$ given in \eqref{P3}, we have 
\begin{multline*}
%	\label{451}
	\big(\pa^\al\na_x\phi\cdot v\mu^{\frac{1}{2}},\pa^\al \f\big)_{L^2_{x,v}}
	= -\big(\pa^\al\phi,\na_x\cdot\int_{\R^3}v\pa^\al \f\mu^{\frac{1}{2}}\,dv\big)_{L^2_x}\\
	= \big(\pa^\al\phi,\pa_t\int_{\R^3}\pa^\al \f\mu^{\frac{1}{2}}\,dv\big)_{L^2_x}
	= -\frac{1}{2}\big(\pa^\al\phi,\pa_t\pa^\al\Delta_x\phi\big)_{L^2_x}
	= \frac{1}{4}\pa_t\|\pa^\al\na_x\phi\|_{L^2_x}^2. 
	%	\big(\pa^\al\na_x\phi\cdot v\mu^{\frac{1}{2}},w^2(\al)\pa^\al P_{\mu}F_2\big)_{L^2_{x,v}}+\big(\pa^\al\na_x\phi\cdot v\mu,w^2(\al)\pa^\al \{I-P_{\mu}\}F_2\big)_{L^2_{x,v}}\\
	%	&\ge C_\al\big(\pa^\al\na_x\phi,\int_{\R^3}v\pa^\al F_2\,dv\big)_{L^2_x} -C \|\pa^\al\na_x\phi\|_{L^2_x}\|\pa^\al\{I-P_{\mu}\}F_2\|_{L^2_{x,v}}. 
\end{multline*}
For the sixth term, we have from \eqref{Lg} that 
\begin{align*}
	(\L_2\pa^\al\f,\pa^\al \f)_{L^2_{x,v}}\le -\lam \|\{I-P_{\mu,2}\}\pa^\al\f\|^2_{L^2_xL^2_D}.
\end{align*}
Combining these estimates, we have 
\begin{multline}\label{520}
	\frac{1}{2}\pa_t\|\pa^\al \f\|^2_{L^2_{x,v}} + \frac{1}{4}\pa_t\|\pa^\al\na_x\phi\|_{L^2_x}^2
	+\lam \|\{I-P_{\mu,2}\}\pa^\al\f\|^2_{L^2_xL^2_D}\lesssim \delta(1+t)^{-2}\\+\big(\eta_0+\delta^{\frac{1}{3}}+\sqrt{\E_k(t)}\big)\D_k(t). 
%	\\ + 
%	\left\{
%	\begin{aligned}
%		&\big(\eta_0+\delta^{\frac{1}{3}}+\sqrt{\E_k(t)}\big)\D_k(t),\quad\text{ if }|\al|\ge 1,\\
%		&\big(\eta_0+\delta^{\frac{1}{3}}+\sqrt{\E_k(t)}\big)\D_k(t),\quad\text{ if }|\al|=0. 
%	\end{aligned}\right.
\end{multline}

Now we take linear combination $\kappa\times\eqref{516a}+\kappa^2\times\eqref{517}+\eqref{520}$ 
%and $\kappa\times\eqref{516}+\kappa^2\times\eqref{517}+\eqref{520}$
 and summation over $|\al|\le 3$
%  and $1\le|\al|\le 3$ respectively
   with sufficiently small $\kappa>0$ to deduce 
\begin{multline*}
	\pa_t\E_{k,3} + \lam\sum_{|\al|\le 3}\|w(\al)\pa^\al\f\|^2_{L^2_xL^2_D}+
	\kappa\lam\sum_{|\al|\le 3}\big(\|\pa^\al a\|_{L^2_x}^2+\|\pa^\al\na_x\phi\|_{L^2_x}^2\big)
	\\\lesssim \delta(1+t)^{-2}+ 
	\big(\eta_0+\delta^{\frac{1}{3}}+\sqrt{\E_k(t)}\big)\D_k(t),
\end{multline*}
%and 
%\begin{multline*}
%	\pa_t\E^h_{k,3} + \lam\sum_{1\le|\al|\le 3}\|w(\al)\pa^\al\f\|^2_{L^2_xL^2_D}+
%	\kappa\lam\sum_{1\le|\al|\le 3}\big(\|\pa^\al a\|_{L^2_x}^2+\|\pa^\al\na_x\phi\|_{L^2_x}^2\big)
%	\\
%	\lesssim \delta(1+t)^{-2} + 
%		\big(\eta_0+\delta^{\frac{1}{3}}+\sqrt{\E_k(t)}\big)\D_k(t),
%\end{multline*}
where $\E_{k,3}$ 
%and $\E^h_{k,3}$
 is given respectively by 
\begin{align*}
	\E_{k,3} = \sum_{|\al|\le 3}\Big(\frac{1}{2}\|\pa^\al \f\|^2_{L^2_{x,v}} + \frac{1}{4}\|\pa^\al\na_x\phi\|_{L^2_x}^2
	+\frac{\kappa^2}{2}\|w(\al)\pa^\al \f\|^2_{L^2_{x,v}}\Big). 
\end{align*}
%and 
%\begin{align*}
%	\E^h_{k,3} = \sum_{1\le|\al|\le 3}\Big(\frac{1}{2}\|\pa^\al \f\|^2_{L^2_{x,v}} + \frac{1}{4}\|\pa^\al\na_x\phi\|_{L^2_x}^2
%	+\frac{\kappa^2}{2}\|w(\al)\pa^\al \f\|^2_{L^2_{x,v}}\Big). 
%\end{align*}
It's direct to check \eqref{Ek2}. 
%Taking summation of \eqref{540}, \eqref{550}, \eqref{550b} and \eqref{550c}, we have 
%\begin{align*}
%	\frac{1}{2}\pa_t\|\pa^\al a\|_{L^2_x}^2 +\|\pa^\al\na_{t,x}a\|_{L^2_x}^2+\|\pa^\al a\|_{L^2_x}^2+\|\pa^\al\na_x\phi\|_{L^2_x}^2\lesssim
%	\sum_{|\al'|=1}\|\pa^{\al'}\pa^\al\{I-P_2\}F_2\|^2_{L^2_xL^2_6}.
%\end{align*}
%Taking summation of the two estimates in \eqref{550a} and \eqref{550c}, we have 
%\begin{align*}
%	\frac{1}{2}\pa_t\|\pa^\al\na_x\phi\|^2_{L^2_x}+\|\pa^\al\na_{t,x}\na_x\phi\|^2_{L^2_x}+\|\pa^\al\na_x\phi\|_{L^2_x}^2
%	 &\lesssim 
%	\sum_{|\al'|\le1}\|\pa^{\al'}\pa^\al\{I-P_2\}F_2\|^2_{L^2_xL^2_6}. 
%\end{align*}
%
%
%	Combining \eqref{550}, \eqref{550a}, \eqref{550b} and \eqref{550c} and taking summation over $|\al|\le 1$ and letting $\eta>0$ small enough, we obtain \eqref{macro1} 
These estimates conclude Lemma \ref{Lem53}.  
	
\end{proof}

\subsection{Estimates on $\g$}
In this subsection, we make use of equation \eqref{g2} to derive the zeroth and first order derivative energy estimate of $\g$.

\begin{Lem}\label{Lem33}
%	Let $K=2$ be the total order of derivatives and
	Assume $\gamma>\max\{-3,-2s-\frac{3}{2}\}$ for Boltzmann case and $\gamma\ge 0$, $\frac{1}{2}\le s<1$ for VPB case. 
	Let $(F_\pm,\phi)$ be the solution to \eqref{1}. Assume $\ve>0$ in \eqref{priori2} is small enough. Then for any $|\al|\le 2$, 
	\begin{multline}\label{esg}
			\frac{1}{2}\pa_t\|\pa^\al\g\|^2_{L^2_{x,v}}+\frac{\kappa}{2}\pa_t\|w(\al)\pa^\al\g\|^2_{L^2_{x,v}}
			+\lam\kappa \|w(\al)\pa^\al\g\|^2_{L^2_xL^2_D}+\lam \|\pa^\al\g\|^2_{L^2_xL^2_D}\\
			\lesssim\|\pa^\al\na_x\g\|^2_{L^2_xL^2_D} + \|\pa^\al\na_x(\wt u,\wt\th)\|^2_{L^2_x}+ \delta^{\frac{2}{3}}(1+t)^{-\frac{4}{3}}+\big(\eta_0+\delta^{\frac{1}{3}}+\sqrt{\E_k(t)}\big)\D_k(t), 
%			\\ + \left\{\begin{aligned}
%				&\big(\eta_0+\delta^{\frac{1}{3}}+\sqrt{\E_k(t)}\big)\D_k(t),\quad \text{ if }|\al|=1,\\
%				&\big(\eta_0+\delta^{\frac{1}{3}}+\sqrt{\E_k(t)}\big)\D_k(t),\quad \text{ if }|\al|=0. 
%			\end{aligned}\right.
	\end{multline}
for some constants $\kappa,\lam>0$. 
\end{Lem}

\begin{proof}
%	Notice that $\g$ satisfies \eqref{micro} and hence, $P_{\mu}\g=0$, where $P_2$ is given by \eqref{P2}. Then we apply decomposition $\g_\pm=\pm P_2F_2+\g$ and rewrite \eqref{g2} to be 
%\begin{multline}\label{326}
%	\pa_t\g+v\cdot\na_x\g\mp \na_x\phi\cdot\na_vF_\pm=2Q(F_1,\g_\pm)+2Q(\g,M)+2Q(G,\ol G)\\\mp\pa_tP_2F_2\mp v\cdot\na_xP_2F_2
%	-P_0\na_x\phi\cdot\na_vF_2+P_0v\cdot\na_x\g-P_1v\cdot\na_x\ol G-\pa_t\ol G\\
%	-P_1v\cdot \Big\{\frac{|v-u|^2\na_x\wt\th}{2R\th^2}+\frac{(v-u)\cdot \na_x\wt{u}}{R\th}\Big\}M.
%%	\pa_t\g_\pm+v\cdot\na_x\g_\pm\mp \na_x\phi\cdot\na_vF_\pm=2Q(F_1,\g_\pm)+2Q(\g,M)+2Q(G,\ol G)
%%	-P_0\na_x\phi\cdot\na_vF_2\\+P_0v\cdot\na_x\g-P_1v\cdot\na_x\ol G-\pa_t\ol G
%%	-P_1v\cdot \Big\{\frac{|v-u|^2\na_x\wt\th}{2R\th^2}+\frac{(v-u)\cdot \na_x\wt{u}}{R\th}\Big\}M. 
%\end{multline}
Recall \eqref{g2} that 
\begin{multline}\label{326}
\pa_t \g+v\cdot\na_x\g-\frac{P_1\na_x\phi\cdot\na_vF_2}{\sqrt{\mu}} - \L\g = \Gamma\big(\frac{M-\mu}{\sqrt\mu},\g\big)+ \Gamma\big(\g,\frac{M-\mu}{\sqrt\mu}\big) + \Gamma\big(\frac{G}{\sqrt\mu},\frac{G}{\sqrt\mu}\big) \\+ \frac{P_0(v\sqrt\mu\cdot\na_x\g)}{\sqrt\mu} - \frac{P_1(v\cdot\na_x\ol G)}{\sqrt\mu} - \frac{\pa_t\ol G}{\sqrt\mu} - \frac{1}{\sqrt\mu}P_1v\cdot\Big\{\frac{|v-u|^2\na_x\wt\th}{2R\th^2}+\frac{(v-u)\cdot\na_x\wt u}{R\th}\Big\}M.
\end{multline}
For $|\al|\le 1$, applying $\pa^\al$ to \eqref{326} and taking inner product with $w^2(\al)\pa^\al\g$ over $\Omega\times\R^3$ yields  
%	 and noticing $\g\in\N_1^\perp$, where $\N_1$ is the null space of $\L$, we have from \eqref{Lg}, \eqref{220b} and \eqref{220} with $\eta>0$ small enough that 
	\begin{align}\label{327}\notag
		&\quad\,\frac{1}{2}\pa_t\|w(\al)\pa^\al\g\|^2_{L^2_{x,v}}+ \frac{1}{2}\int_{\pa\Omega}\int_{\R^3}v\cdot n|w(\al)\pa^\al\g|^2\,dvdS(x)\\
		&\notag\quad-\big( \pa^\al\big(\frac{P_1\na_x\phi\cdot\na_vF_2}{\sqrt{\mu}}\big),w^2(\al)\pa^\al\g\big)_{L^2_{x,v}}\\
		&\notag\lesssim \big(\pa^\al \L\g,w^2(\al)\pa^\al\g\big)_{L^2_{x,v}}\\&\quad\notag
		+\big(\Gamma\big(\frac{M-\mu}{\sqrt\mu},\g\big)+ \Gamma\big(\g,\frac{M-\mu}{\sqrt\mu}\big)+\Gamma\big(\frac{G}{\sqrt\mu},\frac{G}{\sqrt\mu}\big),w^2(\al)\pa^\al\g\big)_{L^2_{x,v}}\\
%		&\notag\quad+\big(\mp\pa_t\pa^\al P_2F_2\mp \pa^\al (v\cdot\na_xP_2F_2)-\pa^\al (P_0\na_x\phi\cdot\na_vF_2),w^2(\al)\pa^\al\g\big)_{L^2_{x,v}}\\
		&\notag\quad+\big(\pa^\al (\frac{P_0(v\sqrt\mu\cdot\na_x\g)}{\sqrt\mu})-\pa^\al (\frac{P_1(v\cdot\na_x\ol G)}{\sqrt\mu})-\frac{\pa_t\ol G}{\sqrt\mu},w^2(\al)\pa^\al\g\big)_{L^2_{x,v}}\\
		&\quad-\big(\pa^\al \frac{1}{\sqrt\mu}P_1v\cdot\Big\{\frac{|v-u|^2\na_x\wt\th}{2R\th^2}+\frac{(v-u)\cdot\na_x\wt u}{R\th}\Big\}M,w^2(\al)\pa^\al\g\big)_{L^2_{x,v}},
%		 - \Big(\frac{P_1(v\cdot\na_x\ol G)}{\sqrt\mu},\g_\pm\Big)_{L^2_{x,v}} - \Big(\frac{\pa_t\ol G}{\sqrt\mu},\g_\pm\Big)_{L^2_{x,v}},
	\end{align}
%	for some $\lam>0$ and any $\eta>0$, 
	where $dS(x)$ is the spherical measure on $\pa\Omega$. 
	We denote the second to seventh term of \eqref{327} by $J_i$ $(1\le i\le 6)$. Then we estimate \eqref{327} term by term.
	For the boundary term, we have from \eqref{29} that for $x\in \Gamma_i$ $(i=2,3)$, 
	\begin{equation*}
%		\label{435}
		\begin{aligned}
			\pa^\al\g(x,R_xv) &= (-1)^{|\al_i|}\pa^\al\g(x,v). 
%			\pa^\al F_2(x,R_xv) &= (-1)^{|\al_i|+|\beta_i|}\pa^\al F_2(x,v). 
		\end{aligned}
	\end{equation*} 
%Here, for the velocity derivative, we apply $\pa_{v_j}R_xv=-\delta_{ij}$ on $\Gamma_i$ for $i=2,3$, $j=1,2,3$. Moreover, for $P_2F_2(x,v)$ and any $x\in\Gamma_i$ $(i=2,3)$, 
%we apply Lemma \ref{Lemspecular} to deduce that  
%\begin{align*}
%	P_2F_2(x,R_xv) = \int_{\R^3}f(u)\,du\,\mu
%	= P_2F_2(x,v),
%\end{align*}
%and similarly, 
%\begin{align}\label{P2F2}
%	\pa^\al P_2F_2(x,R_xv) = (-1)^{|\al_i|+|\beta_i|}\pa^\al P_2F_2(x,v).
%\end{align}
%Together with \eqref{435}, we have 
%$|\pa^\al\g(x,v)|=|\pa^\al\g(x,R_xv)|$ for $x\in\pa\Omega$. 
%Here, the velocity derivative preserve the identity because $\pa_{v_j}R_xv=-\delta_{ij}$ on $\Gamma_i$ for $i=1,2$, $j=1,2,3$.
 Then using change of variable $v\mapsto R_xv$, one has 
	\begin{align*}
%		\label{b1}\notag
		J_1&=\int_{\pa\Omega}\int_{\R^3}v\cdot n|w(\al)\pa^\al\g(v)|^2\,dvdS(x)\\
		&\notag= \int_{\pa\Omega}\int_{\R^3}R_xv\cdot n|w(\al)\pa^\al\g(R_xv)|^2\,dvdS(x)\\
		&= -\int_{\pa\Omega}\int_{\R^3}v\cdot n|w(\al)\pa^\al\g(v)|^2\,dvdS(x)=0.
	\end{align*}
%	Using \eqref{211} and direct calculation on $P_1=I-P_0$, we have 
%	\begin{align}\label{328}
%		P_1v\cdot\Big\{\frac{|v-u|^2\na_x\wt\th}{2R\th^2}+\frac{(v-u)\cdot\na_x\wt u}{R\th}\Big\}M = \frac{\sqrt R}{\sqrt\th}\sum_{j=1}^3\pa_{x_j}\wt\th \wh A_j\big(\frac{v-u}{\sqrt{R\th}}\big)M + \sum_{i,j=1}^3\pa_{x_i}\wt u_j\wh B_{ij}\big(\frac{v-u}{\sqrt{R\th}}\big)M,
%	\end{align}
%When $|\beta|=0$, we know that $J_1$ doesn't occur and hence, $J_1=0$. When $|\beta|\ge 1$, noticing that 
%\begin{align*}
%	w(\al) = \<v\>^{\gamma}w(|\al|+1,|\beta|-1),{\red check?}
%\end{align*}
%for both hard and soft potential, 
%we have 
%\begin{align}\label{aI1}
%	|J_1|\lesssim C\sum_{i=1}^3\|w(|\al|+1,|\beta|-1)\pa^{\al+e_i}_{\beta-e_i}\g\|^2_{L^2_{\gamma/2}}+\frac{\gamma_1}{8}\|w(\al)\pa^\al\g\|^2_{L^2_{\gamma/2}}, 
%\end{align}
%where $\gamma_1>0$ is a generic constant given in Lemma \ref{L31}. 
The term $J_2$ vanishes for Boltzmann equation and occurs only for the case of VPB system. For VPB system, we assume $\gamma\ge 0$ and $1/2\le s<1$. Then we write $F_2=\sqrt\mu \f$ and decompose 
%$\na_vF_2= -\frac{v}{2}\sqrt\mu\f+\sqrt\mu\na_v\f$. 
%and apply integration by parts with respect to $\na_v$ to obtain 
\begin{align*}
%	\label{434}\notag
	J_2&= \big(\pa^\al\big(\frac{\na_x\phi\cdot{v}\f}{2}\big),w^2(\al)\pa^\al\g\big)_{L^2_{x,v}}
	-\big(\pa^\al\big({\na_x\phi\cdot\na_v\f}{}\big),w^2(\al)\pa^\al\g\big)_{L^2_{x,v}}\\
	&\quad+\big(\pa^\al\big(\frac{P_0\na_x\phi\cdot\na_vF_2}{\sqrt{\mu}}\big),w^2(\al)\pa^\al\g\big)_{L^2_{x,v}}
	=: J_{2,1}+J_{2,2}+J_{2,3}. 
%	 \\
%	-\big(\na_x\phi\,\pa^{\al} F_2,w^2(\al)\na_v\pa^\al\g\big)_{L^2_{x,v}}.  
\end{align*} 
For $J_{2,1}$, noticing $\<v\>\lesssim \<v\>^{\gamma+2s}$, we have 
\begin{align*}
	|J_{2,1}|&\lesssim \Big(\sum_{|\al'|=0}\|\pa^{\al'}\na_x\phi\|_{L^\infty_x}\|w(\al)\pa^{\al-\al'}\f\|_{L^2_xL^2_D}\\
	&\quad+\sum_{|\al'|=1}\|\pa^{\al'}\na_x\phi\|_{L^3_x}\|w(\al)\pa^{\al-\al'}\f\|_{L^6_xL^2_D}\\
	&\quad+\sum_{2\le|\al'|\le3}\|\pa^{\al'}\na_x\phi\|_{L^2_x}\|w(\al)\pa^{\al-\al'}\f\|_{L^\infty_xL^2_D}\Big)\|w(\al)\pa^\al\g\|_{L^2_xL^2_D}\\
	&\lesssim \sqrt{\E_k(t)}\D_k(t).
%	\left\{\begin{aligned}
%		&\sqrt{\E_k(t)}\D_k(t),\quad \text{ if }|\al|=1,\\
%		&\sqrt{\E_k(t)}\D_k(t),\quad \text{ if }|\al|=0. 
%	\end{aligned}\right. 
\end{align*}
Similarly, for $J_{2,2}$, noticing $|\<D_v\>^{\frac{1}{2}}(\cdot)|_{L^2_v}\lesssim |\cdot|_{L^2_D}$ for VPB case, we have 
\begin{align*}
	|J_{2,2}|&\lesssim \sum_{\al'\le\al}\|\pa^{\al'}\na_x\phi\,w(\al)\pa^{\al-\al'}\f\|_{L^2_xL^2_D}\|w(\al)\pa^\al\g\|_{L^2_xL^2_D}\\
	&\lesssim \sqrt{\E_k(t)}\D_k(t). 
%	\left\{\begin{aligned}
%		&\sqrt{\E_k(t)}\D_k(t),\quad \text{ if }|\al|=1,\\
%		&\sqrt{\E_k(t)}\D_k(t),\quad \text{ if }|\al|=0. 
%	\end{aligned}\right. 
\end{align*}
For $J_{2,3}$, it follows from \eqref{P0} that 
\begin{align*}
	|P_0\na_vF_2| = \big|-\sum_{i=0}^4\big(F_2,\na_v(\frac{\chi_i}{M})\big)_{L^2_v}\,\chi_i\big|
	\lesssim |\f|_{L^2_D}\sum_{i=0}^4\,|\chi_i|. 
\end{align*}
Also, since $1<R\th<2$, we have 
\begin{align}\label{bbb}
	|\<v\>^bM\mu^{-1/2}|_{L^2_v}\le C_b, 
\end{align} for any $b>0$. Then 
\begin{align*}
	|J_{2,3}| &\lesssim \sum_{\al'\le\al}\|\pa^{\al'}\na_x\phi\,|\f|_{L^2_D}\|_{L^2_x}\|w(\al)\pa^\al\g\|_{L^2_xL^2_D}
	\lesssim\sqrt{\E_k(t)}\D_k(t). 
%	 \left\{\begin{aligned}
%		&\sqrt{\E_k(t)}\D_k(t),\quad \text{ if }|\al|=1,\\
%		&\sqrt{\E_k(t)}\D_k(t),\quad \text{ if }|\al|=0. 
%	\end{aligned}\right. 
\end{align*}
Collecting the above estimates, we obtain 
\begin{align*}
	|J_2|\lesssim\sqrt{\E_k(t)}\D_k(t). 
%	 \left\{\begin{aligned}
%		&\sqrt{\E_k(t)}\D_k(t),\quad \text{ if }|\al|=1,\\
%		&\sqrt{\E_k(t)}\D_k(t),\quad \text{ if }|\al|=0. 
%	\end{aligned}\right. 
\end{align*}
For $J_3$, we have from \eqref{L1} that 
\begin{align*}
	J_3 &\le -\lam \|w(\al)\pa^\al\g\|^2_{L^2_xL^2_D} + C\|\pa^\al\g\|^2_{L^2_xL^2_{B_C}}. 
\end{align*}
For $J_4$, it follows from \eqref{220b} and \eqref{220} that 
\begin{align*}
	|J_4|\lesssim \delta(1+t)^{-2}+ \big(\eta_0+\delta^{\frac{1}{3}}+\sqrt{\E_k(t)}\big)\D_k(t).
%	\left\{\begin{aligned}
%		&\big(\eta_0+\delta^{\frac{1}{3}}+\sqrt{\E_k(t)}\big)\D_k(t),\quad \text{ if }|\al|=1,\\
%		&\big(\eta_0+\delta^{\frac{1}{3}}+\sqrt{\E_k(t)}\big)\D_k(t),\quad \text{ if }|\al|=0. 
%	\end{aligned}\right.
\end{align*}
For $J_5$, it follows from \eqref{344} and \eqref{bbb} that 
\begin{align*}
	|J_5|\lesssim C_\eta\|\pa^\al\na_x\g\|^2_{L^2_xL^2_D} + \eta\|\pa^\al\g\|^2_{L^2_xL^2_D} + \delta^{\frac{2}{3}}(1+t)^{-\frac{4}{3}}. 
\end{align*}
For $J_6$, similar to \eqref{olG3}, we have 
\begin{align*}
	P_1v\cdot\Big\{\frac{|v-u|^2\na_x\wt\th}{2R\th^2}+\frac{(v-u)\cdot\na_x\wt u}{R\th}\Big\}M
	=\frac{\sqrt R}{\sqrt\th}\sum_{j=1}^3\pa_{x_j}\wt\th \wh A_j\big(\frac{v-u}{\sqrt{R\th}}\big)M + \sum_{i,j=1}^3\pa_{x_i}\wt u_j\wh B_{ij}\big(\frac{v-u}{\sqrt{R\th}}\big)M. 
\end{align*}
Then by \eqref{bbb}, we have 
%we apply formula \eqref{328},  fast velocity decay on $M$ and integration by parts with respect to $\pa$ to deduce that 
\begin{align*}\notag
	J_7&= 
	\Big(\frac{1}{\sqrt\mu}\pa^\al\big(\frac{\sqrt R}{\sqrt\th}\sum_{j=1}^3\pa_{x_j}\wt\th \wh A_j\big(\frac{v-u}{\sqrt{R\th}}\big)M + \sum_{i,j=1}^3\pa_{x_i}\wt u_j\wh B_{ij}\big(\frac{v-u}{\sqrt{R\th}}\big)M\big),w^2(\al)\pa^\al\g\Big)_{L^2_{x,v}}\\
	&\notag\lesssim \|\pa^\al\na_x(\wt u,\wt\th)\|_{L^2_x}\|\pa^\al\g\|_{L^2_xL^2_D} + \sum_{|\al_1|\le 1}\|\pa^{\al_1}\na_x(\wt u,\wt\th)\|_{L^2_x}\sum_{|\al_1|\le 2}\|\pa^{\al_1}(u,\th)\|_{L^2_x}\|\pa^\al\g\|_{L^2_xL^2_D}\\
	&\lesssim \eta\|\pa^\al\g\|^2_{L^2_xL^2_D}+C_\eta\|\pa^\al\na_x(\wt u,\wt\th)\|^2_{L^2_x}+\big(\delta^{\frac{1}{3}}+\sqrt{\E_k(t)}\big)\D_k(t).
%	\left\{\begin{aligned}
%		&\big(\delta^{\frac{1}{3}}+\sqrt{\E_k(t)}\big)\D_k(t),\quad \text{ if }|\al|=1,\\
%		&\big(\delta^{\frac{1}{3}}+\sqrt{\E_k(t)}\big)\D_k(t),\quad \text{ if }|\al|=0. 
%	\end{aligned}\right. 
\end{align*}
In a summary, combining the above estimates and choosing $\eta>0$ sufficiently small, \eqref{327} yields 
\begin{multline}\label{442}
	\frac{1}{2}\pa_t\|w(\al)\pa^\al\g\|^2_{L^2_{x,v}}
	+\lam \|w(\al)\pa^\al\g\|^2_{L^2_xL^2_D} \lesssim\|\pa^\al\g\|^2_{L^2_{B_C}}+\|\pa^\al\na_x\g\|^2_{L^2_xL^2_D} + \|\pa^\al\na_x(\wt u,\wt\th)\|^2_{L^2_x}\\+ \delta^{\frac{2}{3}}(1+t)^{-\frac{4}{3}} +\big(\eta_0+\delta^{\frac{1}{3}}+\sqrt{\E_k(t)}\big)\D_k(t).
%	 \left\{\begin{aligned}
%		&\big(\eta_0+\delta^{\frac{1}{3}}+\sqrt{\E_k(t)}\big)\D_k(t),\quad \text{ if }|\al|=1,\\
%		&\big(\eta_0+\delta^{\frac{1}{3}}+\sqrt{\E_k(t)}\big)\D_k(t),\quad \text{ if }|\al|=0. 
%	\end{aligned}\right.
\end{multline}
Next we calculate the energy estimate without weight. For $|\al|\le 1$, applying $\pa^\al$ to \eqref{326} and taking inner product with $\pa^\al\g$ over $\Omega\times\R^3$, we have 
	\begin{multline*}
	\frac{1}{2}\pa_t\|\pa^\al\g\|^2_{L^2_{x,v}}+ \frac{1}{2}\int_{\pa\Omega}\int_{\R^3}v\cdot n|\pa^\al\g|^2\,dvdS(x)
	-\big( \pa^\al\big(\frac{P_1\na_x\phi\cdot\na_vF_2}{\sqrt{\mu}}\big),\pa^\al\g\big)_{L^2_{x,v}}\\
	\lesssim \big(\pa^\al \L\g,\pa^\al\g\big)_{L^2_{x,v}}
	+\big(\Gamma\big(\frac{M-\mu}{\sqrt\mu},\g\big)+ \Gamma\big(\g,\frac{M-\mu}{\sqrt\mu}\big)+\Gamma\big(\frac{G}{\sqrt\mu},\frac{G}{\sqrt\mu}\big),\pa^\al\g\big)_{L^2_{x,v}}\\
	+\big(\pa^\al (\frac{P_0(v\sqrt\mu\cdot\na_x\g)}{\sqrt\mu})-\pa^\al (\frac{P_1(v\cdot\na_x\ol G)}{\sqrt\mu})-\frac{\pa_t\ol G}{\sqrt\mu},\pa^\al\g\big)_{L^2_{x,v}}\\
	-\big(\pa^\al \frac{1}{\sqrt\mu}P_1v\cdot\Big\{\frac{|v-u|^2\na_x\wt\th}{2R\th^2}+\frac{(v-u)\cdot\na_x\wt u}{R\th}\Big\}M,\pa^\al\g\big)_{L^2_{x,v}},
\end{multline*}
The calculations are similar to the weighted case and we calculate the different terms, i.e. the forth term. If follows from \eqref{Lg} and $\g=\{I-P_{\mu}\}\g$ that 
\begin{align*}
\big(\pa^\al \L\g,\pa^\al\g\big)_{L^2_{x,v}}\le -\lam \|\pa^\al\g\|_{L^2_xL^2_D}^2. 
\end{align*}
Then similar to the calculations for deriving \eqref{442}, we have 
\begin{multline}\label{442a}
	\frac{1}{2}\pa_t\|\pa^\al\g\|^2_{L^2_{x,v}}
	+\lam \|\pa^\al\g\|^2_{L^2_xL^2_D} \lesssim\|\pa^\al\na_x\g\|^2_{L^2_xL^2_D} + \|\pa^\al\na_x(\wt u,\wt\th)\|^2_{L^2_x}\\+ \delta^{\frac{2}{3}}(1+t)^{-\frac{4}{3}} +\big(\eta_0+\delta^{\frac{1}{3}}+\sqrt{\E_k(t)}\big)\D_k(t).
%	 \left\{\begin{aligned}
%		&\big(\eta_0+\delta^{\frac{1}{3}}+\sqrt{\E_k(t)}\big)\D_k(t),\quad \text{ if }|\al|=1,\\
%		&\big(\eta_0+\delta^{\frac{1}{3}}+\sqrt{\E_k(t)}\big)\D_k(t),\quad \text{ if }|\al|=0. 
%	\end{aligned}\right.
\end{multline}
Taking combination $\kappa\times\eqref{442}+\eqref{442a}$ with small enough $\kappa>0$, we obtain \eqref{esg}. 
%\begin{multline*}
%	\frac{1}{2}\pa_t\|\pa^\al\g\|^2_{L^2_{x,v}}+\frac{\kappa}{2}\pa_t\|w(\al)\pa^\al\g\|^2_{L^2_{x,v}}
%	+\lam \|w(\al)\pa^\al\g\|^2_{L^2_xL^2_D}
%	 \lesssim\\|\pa^\al\na_x\g\|^2_{L^2_xL^2_D} + \|\pa^\al\na_x(\wt u,\wt\th)\|^2_{L^2_x}\\+ \min\{\delta^{\frac{2}{3}}(1+t)^{-\frac{4}{3}},(1+t)^{-2}\} + \left\{\begin{aligned}
%		&\big(\eta_0+\delta^{\frac{1}{3}}+\sqrt{\E_k(t)}\big)\D^h_k(t),\quad \text{ if }|\al|=1,\\
%		&\big(\eta_0+\delta^{\frac{1}{3}}+\sqrt{\E_k(t)}\big)\D_k(t),\quad \text{ if }|\al|=0. 
%	\end{aligned}\right.
%\end{multline*}
This completes the proof of Lemma \ref{Lem33}.

\end{proof}

\subsection{Estimates on $F_1$}

In this section, we will derive the second-order time-spatial derivative estimates on $F_1$ combining with the high-order macroscopic estimate \eqref{highmacro}, which implies the dissipation rate of $\pa^\al\g$ for the highest order $|\al|=2$.

	\begin{Lem}\label{Lem42} 
		Assume $\gamma>\max\{-3,-2s-\frac{3}{2}\}$ for Boltzmann case and $\gamma\ge 0$, $\frac{1}{2}\le s<1$ for VPB case. 
		Let $(F_1,\phi)$ be the solution to \eqref{F1} (satisfying \eqref{specular} and \eqref{Neumann} for the case of rectangular duct) and $\ve>0$ in \eqref{priori2} is small enough. Then 
	\begin{multline}\label{esF}
		\kappa\pa_t\E_{int}(t)  +
		\frac{\kappa^2}{2}\pa_t\sum_{|\al|=3}\big\|w(\al)\frac{\pa^\al F_1}{\sqrt\mu}\big\|^2_{L^2_xL^2_v}
		+\frac{1}{2}\pa_t\sum_{|\al|=3}\big\|\frac{\pa^\al F_1}{\sqrt\mu}\big\|^2_{L^2_xL^2_v}\\
		+\lam\kappa\sum_{2\le|\al|\le3}\|\pa^\al(\wt\rho,\wt u,\wt\th)\|_{L^2_x}^2
		+\lam\kappa^2\sum_{|\al|=3}\|w(\al)\pa^\al\g\|^2_{L^2_xL^2_D}
		+\lam\sum_{|\al|=3}\|\pa^\al\g\|^2_{L^2_xL^2_D}\\
		\lesssim \big(\eta_0+\delta^{\frac{1}{3}}+\sqrt{\E_k(t)}\big)\D_k(t)+\delta^{\frac{2}{3}}(1+t)^{-\frac{4}{3}}. 
	\end{multline}
	for some generic constant $\lam>0$, any $\eta>0$ and sufficiently small $\kappa>0$. Here, $\E_{int}(t)$ is given in \eqref{Eint}. 
\end{Lem}
\begin{proof}
%	For $|\al|=2$, applying $\pa^\al$ to 
	We can obtain from the first equation of \eqref{F1} and \eqref{olG} that  
	\begin{multline}\label{41}
		\pa_t\big(\frac{F_1}{\sqrt\mu}\big)+v\cdot\na_x\big(\frac{F_1}{\sqrt\mu}\big) = \L\g+\Gamma\big(\g,\frac{M-\mu}{\sqrt\mu}\big)+\Gamma\big(\frac{M-\mu}{\sqrt\mu},\g\big)+\Gamma\big(\frac{G}{\sqrt\mu},\frac{G}{\sqrt\mu}\big)\\
		+\frac{1}{\sqrt\mu}P_1v_1M\Big\{	\frac{|v-u|^2\bar\th_{x_1}}{2R\th^2}+\frac{(v-u)\cdot \bar{u}_{x_1}}{R\th}\Big\}
	\end{multline}
	For $|\al|=3$, applying derivative $\pa^\al$ to \eqref{41} and taking inner product with $w^2(\al)\frac{\pa^\al F_1}{\sqrt\mu}$ over $\Omega\times\R^3$, we have 
	\begin{multline}\label{44}
		\frac{1}{2}\pa_t\big\|w(\al)\frac{\pa^\al F_1}{\sqrt\mu}\big\|^2_{L^2_xL^2_v} + \frac{1}{2}\int_{\pa\Omega}\int_{\R^3}v\cdot n\big|w(\al)\frac{\pa^\al F_1}{\sqrt\mu}\big|^2\,dvdS(x) = \big(\pa^\al \L\g,w^2(\al)\frac{\pa^\al F_1}{\sqrt\mu}\big)_{L^2_{x,v}}\\
		+\big(\pa^\al\Gamma\big(\g,\frac{M-\mu}{\sqrt\mu}\big)+\pa^\al\Gamma\big(\frac{M-\mu}{\sqrt\mu},\g\big)+\pa^\al \Gamma\big(\frac{G}{\sqrt\mu},\frac{G}{\sqrt\mu}\big),w^2(\al)\frac{\pa^\al F_1}{\sqrt\mu}\big)_{L^2_{x,v}}\\ + \Big(\frac{1}{\sqrt\mu}\pa^\al P_1v_1M\Big\{	\frac{|v-u|^2\bar\th_{x_1}}{2R\th^2}+\frac{(v-u)\cdot \bar{u}_{x_1}}{R\th}\Big\},w^2(\al)\frac{\pa^\al F_1}{\sqrt\mu}\Big)_{L^2_{x,v}}. 
	\end{multline}
We denote the second to fifth terms in \eqref{44} by $K_i$ $(1\le i\le 4)$. 	
	For the boundary term $K_1$, we have from Lemma \ref{Lemspecular} that $\big|\frac{\pa^\al F_1}{\sqrt\mu}(R_xv)\big|=\big|\frac{\pa^\al F_1}{\sqrt\mu}(v)\big|$ on $x\in\pa\Omega$. Noticing $R_xv\cdot n=-v\cdot n$ and using change of variable $v\mapsto R_xv$, we know that 
	\begin{align*}
		K_1= \int_{\pa\Omega}\int_{\R^3}v\cdot n\big|w(\al)\frac{\pa^\al F_1}{\sqrt\mu}(v)\big|^2\,dvdS(x) &= \int_{\pa\Omega}\int_{\R^3}R_xv\cdot n\big|w(\al)\frac{\pa^\al F_1}{\sqrt\mu}(R_xv)\big|^2\,dvdS(x)\\
		&=-\int_{\pa\Omega}\int_{\R^3}v\cdot n\big|w(\al)\frac{\pa^\al F_1}{\sqrt\mu}(v)\big|^2\,dvdS(x)=0. 
	\end{align*}
	For $K_2$, we split $F_1=M+\ol G+\sqrt\mu \g$. For the term with $\g$, using \eqref{L1}, we have 
	\begin{align*}
		\big(\pa^\al \L\g,w^2(\al){\pa^\al \g}\big)_{L^2_{x,v}}\le -\lam \|w(\al)\pa^\al\g\|^2_{L^2_xL^2_D}+C\|\pa^\al\g\|^2_{L^2_xL^2_{B_C}}.
	\end{align*}
	For the term with $M+\ol G$, using \eqref{Gamma}, \eqref{Gamma1}, \eqref{333} and \eqref{344}, we have 
	\begin{multline*}
		|(\L\pa^\al\g,w^2(\al)\frac{\pa^\al(M+\ol G)}{\sqrt\mu})_{L^2_{x,v}}|\lesssim \|\pa^\al\g\|_{L^2_xL^2_D}\|w^2(\al)\frac{\pa^\al(M+\ol G)}{\sqrt\mu}\|_{L^2_xL^2_D}
		\\\lesssim \|\pa^\al\g\|^2_{L^2_xL^2_D}+\|\pa^{\al}(\wt\rho,\wt u,\wt\th)\|^2_{L^2_x} + \E_k(t)\D_k(t)+\delta^{\frac{2}{3}}(1+t)^{-\frac{4}{3}}.  
	\end{multline*}
	Note that $|\<v\>^b\frac{M}{\sqrt\mu}|_{L^2_v}\le C_b$ for any $b>0$ due to $1<R\th<2$, which is from \eqref{eta}. 
	Collecting the above estimates for $K_2$ and, we have 
	\begin{multline*}
		K_2\le -\lam \|w(\al)\pa^\al\g\|^2_{L^2_xL^2_D}+C\|\pa^\al\g\|^2_{L^2_xL^2_D}
		+C\|\pa^{\al}(\wt\rho,\wt u,\wt\th)\|^2_{L^2_x} + C\E_k(t)\D_k(t)\\+C\delta^{\frac{2}{3}}(1+t)^{-\frac{4}{3}}, 
	\end{multline*}
	for some $\lam>0$. 
	For $K_3$, it follows from \eqref{220b}, \eqref{333}, \eqref{332a} and \eqref{344} that 
	\begin{align*}
		|K_3|\lesssim \big(\eta_0+\delta^{\frac{1}{3}}+\sqrt{\E_k(t)}\big)\D_k(t)+\delta^{\frac{2}{3}}(1+t)^{-\frac{4}{3}}. 
	\end{align*}
Note that $\|\pa^\al(\wt\rho,\wt u,\wt\th)\|_{L^2_x}\lesssim \sqrt{\D_k(t)}$. 
For $K_4$, 
	using \eqref{olG3}, \eqref{333}, 
%	\eqref{pat}, 
	\eqref{pat1} and Lemma \ref{Lem21}, we have 
	\begin{align*}
		&\quad\,\Big(\frac{1}{\sqrt\mu}\pa^\al P_1v_1M\Big\{	\frac{|v-u|^2\bar\th_{x_1}}{2R\th^2}+\frac{(v-u)\cdot \bar{u}_{x_1}}{R\th}\Big\},w^2(\al)\frac{\pa^\al M}{\sqrt\mu}\Big)_{L^2_{x,v}}\\
		&\notag\lesssim \Big(\|\pa^\al\na_x(\bar\rho,\bar u,\bar\th)\|_{L^2_x}+\sum_{|\al'|\le2}\|\pa^{\al'}\na_x(\bar\rho,\bar u,\bar\th)\|_{L^3_x}\sum_{1\le|\al'|\le2}\|\pa^{\al'}(\rho,u,\th)\|_{L^6_x}\\&\qquad+\|\na_x(\bar\rho,\bar u,\bar\th)\|_{L^\infty_x}\|\pa^\al(\rho,u,\th)\|_{L^2_x}\Big)\big\|w^2(\al)\frac{\pa^\al M}{\sqrt\mu}\big\|_{L^2_{x,v}}\\
		&\lesssim \Big(\delta^{\frac{1}{3}}\sum_{1\le|\al'|\le3}\|\pa^{\al'}(\wt\rho,\wt u,\wt\th)\|_{L^2_x}+\delta^{\frac{1}{3}}(1+t)^{-\frac{2}{3}}\Big)\\
		&\qquad\times
		\Big(\|\pa^{\al}(\wt\rho,\wt u,\wt\th)\|_{L^2_x} + \sqrt{\E_k(t)\D_k(t)}+\delta^{\frac{1}{3}}(1+t)^{-\frac{2}{3}}\Big)\\
%		&\notag\lesssim \|\pa^\al\na_x(\bar\rho,\bar u,\bar\th)\|_{L^2_x}\|\pa^\al(\wt\rho,\wt u,\wt\th)\|_{L^2_x}
%		+ \|\pa^\al\na_x(\bar\rho,\bar u,\bar\th)\|_{L^2_x}\|\pa^\al(\bar\rho,\bar u,\bar\th)\|_{L^2_x}\\
%		&\notag\qquad
%		+ \|\na_x(\bar\rho,\bar u,\bar\th)\|_{L^3_x}\|\pa^\al(\wt\rho,\wt u,\wt\th)\|_{L^2_x}\|\pa^\al(\wt\rho,\wt u,\wt\th)\|_{L^6_x}\\
%		&\notag\qquad
%		+ \|\na_x(\bar\rho,\bar u,\bar\th)\|_{L^2_x}\|\pa^\al(\bar\rho,\bar u,\bar\th)\|_{L^\infty_x}\|\pa^\al(\wt\rho,\wt u,\wt\th)\|_{L^2_x}\\
		&\lesssim \eta\|\pa^\al(\wt\rho,\wt u,\wt\th)\|_{L^2_x}^2+ \big(\delta^{\frac{1}{3}}+\sqrt{\E_k(t)}\big)\D_k(t)+C_\eta\delta^{\frac{2}{3}}(1+t)^{-\frac{4}{3}}, 
	\end{align*}
for $\eta>0$.  Similarly, we deduce from \eqref{344} that 
	\begin{align*}
		&\notag\quad\,\Big(\frac{1}{\sqrt\mu}\pa^\al P_1v_1M\Big\{	\frac{|v-u|^2\bar\th_{x_1}}{2R\th^2}+\frac{(v-u)\cdot \bar{u}_{x_1}}{R\th}\Big\},w^2(\al)\frac{\pa^\al (\ol G+\sqrt\mu\g)}{\sqrt\mu}\Big)_{L^2_{x,v}}\\
		&\notag\lesssim \Big(\delta^{\frac{1}{3}}\sum_{1\le|\al'|\le3}\|\pa^{\al'}(\wt\rho,\wt u,\wt\th)\|_{L^2_x}+\delta^{\frac{1}{3}}(1+t)^{-\frac{2}{3}}\Big)
		\big(\delta^{\frac{1}{3}}(1+t)^{-\frac{2}{3}}+\|\pa^\al\g\|_{L^2_xL^2_D}\big)\\
%		\int_{\Omega}\Big(|\pa^\al\na_x(\bar\rho,\bar u,\bar\th)|+|\na_x(\bar\rho,\bar u,\bar\th)||\pa^\al(\rho,u,\th)|\Big)\\
%		&\notag\qquad\qquad\times\Big(|\pa^\al(\rho, u,\th)||\na_x(\bar\rho,\bar u,\bar\th)|+|\pa^\al\na_x(\bar\rho,\bar u,\bar\th)|+|\pa^\al\g|_{L^2_D}\Big)\,dx\\
		&\lesssim \eta\|\pa^\al\g\|_{L^2_xL^2_D}^2+\big(\delta^{\frac{1}{3}}+\sqrt{\E_k(t)}\big)\D_k(t)+C_\eta\delta^{\frac{2}{3}}(1+t)^{-\frac{4}{3}},
	\end{align*}
	for any $\eta>0$. 
	Combining the above estimates, we have 
%		Substituting  estimates \eqref{45}, \eqref{45aa}, \eqref{45a}, \eqref{414} and \eqref{415} into \eqref{44}, taking summation over $|\al|=1$ and choosing $\eta>0$ sufficiently small, we have
	\begin{multline}\label{522}
		\frac{1}{2}\pa_t\big\|w(\al)\frac{\pa^\al F_1}{\sqrt\mu}\big\|^2_{L^2_xL^2_v}
		+\lam\|w(\al)\pa^\al\g\|^2_{L^2_xL^2_D} \lesssim \|\pa^\al\g\|^2_{L^2_xL^2_D}
		+\|\pa^{\al}(\wt\rho,\wt u,\wt\th)\|^2_{L^2_x}\\ +\big(\eta_0+\delta^{\frac{1}{3}}+\sqrt{\E_k(t)}\big)\D_k(t)+C_\eta\delta^{\frac{2}{3}}(1+t)^{-\frac{4}{3}}, 
	\end{multline}
	for some generic constant $\lam>0$. 
	
	Next we calculate the estimate without weight. Applying derivative $\pa^\al$ to \eqref{41} and taking inner product with $w^2(\al)\frac{\pa^\al F_1}{\sqrt\mu}$ over $\Omega\times\R^3$, we apply similar calculations as above and calculate the different term $$\big(\pa^\al \L\g,\frac{\pa^\al F_1}{\sqrt\mu}\big)_{L^2_{x,v}}.$$ 
	Splitting $F_1=M+\ol G+\sqrt\mu\g$, we begin with calculating the term $\pa^\al M$. For $|\al|=3$, we write 
%	$\al=\al_1+\al_2$ with $|\al_1|=|\al_2|=1$. Then 
	\begin{align*}
%		\label{xjxjM}\notag
		\pa^\al M &= \big(\mu+(M-\mu)\big) \Big(\frac{\pa^\al\rho}{\rho}-\frac{3\pa^\al\th}{2\th}+\frac{(v-u)\cdot\pa^\al u}{R\th}+\frac{\pa^\al\th|v-u|^2}{2R\th^2}\Big)+M U\\
%		&\notag\qquad
%		+M\Big(\big(\frac{\pa^{\al_1}\rho}{\rho}-\frac{3\pa^{\al_1}\th}{2\th}+\frac{(v-u)\cdot\pa^{\al_1}u}{R\th}+\frac{\pa^{\al_1}\th|v-u|^2}{2R\th^2}\big)^2\\&\notag\qquad\qquad -\frac{(\pa^{\al_2}\rho)^2}{\rho^2}+\frac{3(\pa^{\al_2}\th)^2}{2\th^2}+\frac{|\pa^{\al_2}u|^2}{R\th}-\frac{2(v-u)\cdot\pa^{\al_2}u\pa^{\al_2}\th}{R\th^2}-\frac{(\pa^{\al_2}\th)^2|v-u|^2}{R\th^3} \Big)\\
		&:= I_1+I_2+I_3,
	\end{align*}
where $U$ is some function satisfying  
\begin{align*}
	|U|\lesssim \sum_{1\le|\al_1|\le 2}|\pa^{\al_1}(\rho,u,\th)|^2+\sum_{|\al_1|=1}|\pa^{\al_1}(\rho,u,\th)|^3. 
\end{align*}
	Here $I_1$ and $I_2$ contain the high order derivatives of $(\rho,u,\th)$ with $\mu$ and $M-\mu$ respectively. $I_3$ contains the low order derivatives with $M$. Then $(\L\g,\frac{I_1}{\sqrt\mu})_{L^2_{x,v}}=0$ due to $\frac{I_1}{\sqrt\mu}\in\ker \L$. For the term $I_2$, we use \eqref{Gamma} and \eqref{Gamma1} to control the upper bound of $\L$ and then use \eqref{33} to obtain that 
	\begin{align*}
		|(\L\pa^\al\g,\frac{I_2}{\sqrt\mu})_{L^2_{x,v}}|&\lesssim \|\pa^\al\g\|_{L^2_xL^2_D}\|\frac{I_2}{\sqrt\mu}\|_{L^2_xL^2_D}\\
		&\lesssim \eta_0\|\pa^\al\g\|_{L^2_xL^2_D}\big(\|\pa^\al(\wt\rho,\wt u,\wt\th)\|_{L^2_x} + \|\pa^\al(\bar\rho,\bar u,\bar\th)\|_{L^2_x}\big)\\
		&\lesssim \eta_0\D_k(t) + \delta^{\frac{2}{3}}(1+t)^{-\frac{4}{3}},
	\end{align*}
	where $\eta_0$ is given in \eqref{eta}. Note that $|\<v\>^b\frac{M}{\sqrt\mu}|_{L^2_v}\le C_b$ for any $b>0$ due to \eqref{eta}. 
	Similarly, 
	\begin{align*}
		|(\L\pa^\al\g,\frac{I_3}{\sqrt\mu})_{L^2_{x,v}}|
%		&\lesssim \|\pa^\al\g\|_{L^2_xL^2_D}\|\frac{I_3}{\sqrt\mu}\|_{L^2_xL^2_D}\\
		&\lesssim \|\pa^\al\g\|_{L^2_xL^2_D}\sum_{1\le|\al_1|\le 2}\|\pa^{\al_1}(\rho, u,\th)\|_{L^3_x}\sum_{1\le|\al_1|\le 2}\|\pa^{\al_1}(\rho, u,\th)\|_{L^6_x}\\
		&\lesssim \|\pa^\al\g\|_{L^2_xL^2_D}\sum_{1\le|\al_1|\le 2}\big(\|\pa^{\al_1}(\wt\rho,\wt u,\wt\th)\|_{H^1_x}+\delta^{\frac{1}{3}}(1+t)^{-\frac{2}{3}}\big)^2\\
%		\\&\qquad\times\big(\|\pa^{\al_2}(\wt\rho,\wt u,\wt\th)\|_{H^1_x}+\delta^{\frac{1}{3}}(1+t)^{-\frac{2}{3}}\big)\\
		&\lesssim \big(\delta^{\frac{1}{3}}+\sqrt{\E_k(t)}\big)\D_k(t) + \delta^{\frac{2}{3}}(1+t)^{-\frac{4}{3}}, 
	\end{align*}
	for $\eta>0$.
	Collecting the above estimates for $\L$, we have 
	\begin{align*}
		\big(\pa^\al \L\g,\frac{\pa^\al F_1}{\sqrt\mu}\big)_{L^2_{x,v}}\le -\lam \|\pa^\al\g\|^2_{L^2_xL^2_D}+\big(\eta_0+\delta^{\frac{1}{3}}+\sqrt{\E_k(t)}\big)\D_k(t) + \delta^{\frac{2}{3}}(1+t)^{-\frac{4}{3}}, 
	\end{align*}
	for some $\lam>0$. 
	Using this as the estimate of $\L$ and applying similar calculations in \eqref{522} for the other terms, we have the energy estimate without weight:
	\begin{multline}\label{522a}
		\frac{1}{2}\pa_t\big\|\frac{\pa^\al F_1}{\sqrt\mu}\big\|^2_{L^2_xL^2_v}
		+\lam\|\pa^\al\g\|^2_{L^2_xL^2_D} \lesssim\eta\|\pa^{\al}(\wt\rho,\wt u,\wt\th)\|^2_{L^2_x}+ \big(\eta_0+\delta^{\frac{1}{3}}+\sqrt{\E_k(t)}\big)\D_k(t)\\+C_\eta\delta^{\frac{2}{3}}(1+t)^{-\frac{4}{3}}.  
	\end{multline}
	Now we combine these estimates with macroscopic estimate \eqref{highmacro}. Taking linear combination $\kappa\times\eqref{highmacro}+\sum_{|\al|=3}\big(\kappa^2\times\eqref{522}+\eqref{522a}\big)$, we obtain 
	\begin{multline*}
		\kappa\pa_t\E_{int}(t)  +
		\frac{\kappa^2}{2}\pa_t\sum_{|\al|=3}\big\|w(\al)\frac{\pa^\al F_1}{\sqrt\mu}\big\|^2_{L^2_xL^2_v}
		+\frac{1}{2}\pa_t\sum_{|\al|=3}\big\|\frac{\pa^\al F_1}{\sqrt\mu}\big\|^2_{L^2_xL^2_v}\\
		+\lam\kappa\sum_{2\le|\al|\le3}\|\pa^\al(\wt\rho,\wt u,\wt\th)\|_{L^2_x}^2
		+\lam\kappa^2\sum_{|\al|=3}\|w(\al)\pa^\al\g\|^2_{L^2_xL^2_D}
		+\lam\sum_{|\al|=3}\|\pa^\al\g\|^2_{L^2_xL^2_D}\\
		 \lesssim \big(\eta_0+\delta^{\frac{1}{3}}+\sqrt{\E_k(t)}\big)\D_k(t)+\delta^{\frac{2}{3}}(1+t)^{-\frac{4}{3}},  
	\end{multline*}
where we choose $\kappa>0$ small enough and then $\eta>0$ sufficiently small. 
	Then we conclude Lemma \ref{Lem42}. 
\end{proof}

\section{Global Existence}\label{Sec6}
In this section, we will prove main theorem \ref{Main1}. 
Before we go into the global existence, we first combine the energy estimate together and deduce the following Lemma. 

%{\red check no time derivative in $\E_k(t)$?? }
\begin{proof}[Proof of Theorem \ref{Main1}]
We first assume the {\it a priori} assumption as \eqref{priori2}:
\begin{align}\label{priori1}
	\sup_{0\le t\le T}{\E_k}(t)\le \ve, 
\end{align}
for any $T>0$ and some small $\ve>0$.

For the global
%low-order
 energy estimate, we take linear combination $\kappa^3\times\eqref{32a}+\sum_{|\al|=0}\kappa^4\times\eqref{esg}+\sum_{1\le|\al|\le2}\kappa^2\times\eqref{esg}+\eqref{f1}+\eqref{esF}$ with small enough $\kappa>0$ to deduce 
\begin{multline*}
	\pa_t\ol\E_k(t) + \lam\sum_{|\al|\le 3}\|w(\al)\pa^\al\f\|^2_{L^2_xL^2_D}+
	\lam\sum_{|\al|\le 3}\big(\|\pa^\al a\|_{L^2_x}^2+\|\pa^\al\na_x\phi\|_{L^2_x}^2\big)\\
	+\lam\kappa^4\|w(0)\g\|^2_{L^2_xL^2_D}
	+\lam\kappa^2\sum_{1\le|\al|\le 2}\|\pa^\al\g\|^2_{L^2_xL^2_D}
	+\lam\kappa^3\sum_{1\le|\al|\le 2}\|w(\al)\pa^\al\g\|^2_{L^2_xL^2_D}\\
	+\lam\sum_{|\al|=3}\|\pa^\al\g\|^2_{L^2_xL^2_D}+\lam\kappa^2\sum_{|\al|=3}\|w(\al)\pa^\al\g\|^2_{L^2_xL^2_D}\\
	+ \lam\kappa^3\sum_{|\al|=1}\|\pa^\al(\wt\rho,\wt u,\wt \th)\|_{L^2_x}^2
	+\lam\kappa\sum_{2\le|\al|\le 3}\|\pa^\al(\wt\rho,\wt u,\wt\th)\|_{L^2_x}^2
	+ \lam\kappa^3\|\sqrt{\bar u_{1x_1}}(\wt\rho,\wt u_1,\wt\th)\|^2_{L^2_x}\\
	\lesssim 
	\delta^{\frac{1}{6}}(1+t)^{-\frac{7}{6}}
%	 +\kappa^3\sum_{|\al|=1}\|\pa^{\al}\g\|^2_{L^2_xL^2_{\gamma/2}}\\
%	\lesssim\kappa^4\|\pa^\al\na_x\g\|^2_{L^2_xL^2_D} + \kappa^4\|\pa^\al\na_x(\wt u,\wt\th)\|^2_{L^2_x}\\
+\big(\eta_0+\delta^{\frac{1}{3}}+\sqrt{\E_k(t)}\big)\D_k(t),
\end{multline*}
for some $\lam>0$, 
where $\ol\E_k(t)$ is given by 
\begin{multline}\label{olE}
	\ol\E_k(t) = \E_{k,3} + \kappa^3\int_\Omega\eta(t)\,dx + \kappa^4(\wt u,\na_x\wt\rho)_{L^2_x}
	+\frac{\kappa^4}{2}\|\g\|^2_{L^2_{x,v}}+\frac{\kappa^5}{2}\|w(0)\g\|^2_{L^2_{x,v}}\\
	+\frac{\kappa^2}{2}\sum_{1\le|\al|\le2}\|\pa^\al\g\|^2_{L^2_{x,v}}+\frac{\kappa^3}{2}\sum_{1\le|\al|\le2}\|w(\al)\pa^\al\g\|^2_{L^2_{x,v}} + \kappa\E_{int}(t) \\
	+\frac{\kappa^2}{2}\sum_{|\al|=3}\big\|w(\al)\frac{\pa^\al F_1}{\sqrt\mu}\big\|^2_{L^2_xL^2_v}
	+\frac{1}{2}\sum_{|\al|=3}\big\|\frac{\pa^\al F_1}{\sqrt\mu}\big\|^2_{L^2_xL^2_v}
\end{multline}
It follows that 
\begin{align}\label{glo2}
	\pa_t\ol\E_k(t) + \lam\D_k(t) \lesssim\delta^{\frac{1}{6}}(1+t)^{-\frac{7}{6}}+\big(\eta_0+\delta^{\frac{1}{3}}+\sqrt{\E_k(t)}\big)\D_k(t). 
\end{align}

Next we claim that 
\begin{equation}
	\begin{aligned}\label{Claim1}
		\E_k(t) -\delta^{\frac{2}{3}} \lesssim \ol\E_k(t) \lesssim \E_k(t) +\delta^{\frac{2}{3}},
%		\\
%		\E_k(t) -\delta^{\frac{2}{3}} \lesssim \ol\E_k(t) \lesssim \E_k(t) +\delta^{\frac{2}{3}}, 
	\end{aligned}
\end{equation}
where $\E_k(t)$ is given in \eqref{E}. 
Indeed, we first calculate $\sum_{|\al|=3}\big\|w(\al)\frac{\pa^\al F_1}{\sqrt\mu}\big\|^2_{L^2_xL^2_v}$. Applying \eqref{333} and \eqref{344}, we have 
\begin{align*}
	\sum_{|\al|=3}\big\|w(\al)\frac{\pa^\al F_1}{\sqrt\mu}\big\|^2_{L^2_xL^2_v}
	&\lesssim \sum_{|\al|=3}\|\pa^{\al}(\wt\rho,\wt u,\wt\th)\|_{L^2_x}^2 +\sum_{|\al|=3}\|w(\al)\pa^\al\g\|_{L^2_xL^2_v}\\&\qquad+ \E_k(t)+\delta^{\frac{2}{3}}(1+t)^{-\frac{4}{3}}\\
	&\lesssim \E_k(t) + \delta^{\frac{2}{3}}(1+t)^{-\frac{4}{3}}. 
\end{align*}
Using this, we obtain the upper bound for $\ol\E_k(t)$ in the second estimate of \eqref{Claim1}. 
%
%		Before we give the global estimate, we first claim that for $|\al|=4,5$,
For the lower bound, similar to \eqref{333}, we have for $|\al|=3$ that 
\begin{align*}
	\|w(\al)\frac{\pa^\al M}{\sqrt\mu}\|^2_{L^2_xL^2_v}&\gtrsim\|\pa^{\al}(\rho,u,\th)\|^2_{L^2_x}
	- \sum_{\substack{\al_1+\al_2=\al\\|\al_1|\ge 1,\,|\al_2|\ge 1}}\|\pa^{\al_1}(\rho,u,\th)\|_{L^6_x}^2\|\pa^{\al_2}(\rho,u,\th)\|^2_{L^3_x}\\
%	&\gtrsim\|\pa^{\al}(\rho,u,\th)\|_{L^2_x} - \sum_{1\le|\al_1|\le 4}\|\pa^{\al_1}(\rho,u,\th)\|_{L^3_x}\sum_{1\le|\al_1|\le 4}\|\pa^{\al_1}(\rho,u,\th)\|_{L^6_x}\\
	&\gtrsim\|\pa^{\al}(\wt\rho,\wt u,\wt\th)\|^2_{L^2_x} -C\E_k(t)\min\{\D_k(t),\E_k(t)\} - C\delta^{\frac{2}{3}}(1+t)^{-\frac{4}{3}}.  
\end{align*}
%Here, we apply \eqref{pat} to estimate the time derivatives as 
%\begin{align}\label{62}\notag
%	&\quad\,\sum_{1\le|\al_1|\le 4}\|\pa^{\al_1}(\rho,u,\th)\|_{L^3_x}\sum_{1\le|\al_1|\le 4}\|\pa^{\al_1}(\rho,u,\th)\|_{L^6_x}\\
%	&\notag\lesssim \Big(\sum_{1\le|\al_1|\le 5}\|\pa^{\al_1}(\wt\rho,\wt u,\wt\th)\|_{L^2_x}+\delta^{\frac{1}{3}}(1+t)^{-\frac{2}{3}}\Big)\Big(\sum_{2\le|\al_1|\le 5}\|\pa^{\al_1}(\wt\rho,\wt u,\wt\th)\|_{L^2_x}+\delta^{\frac{1}{3}}(1+t)^{-\frac{2}{3}}\Big)\\
%	&\notag\lesssim \sum_{m=1}^5\|\na_x^m(\wt\rho,\wt u,\wt\th)\|^2_{L^2_x} + \sum_{|\al|=1}^5\|\pa^\al\na_x\g\|_{L^2_xL^2_5}^2+\delta^{\frac{2}{3}}(1+t)^{-\frac{4}{3}}\\
%	&\lesssim \E_k(t)\min\{\D_k(t),\E_k(t)\}+\delta^{\frac{2}{3}}(1+t)^{-\frac{4}{3}}. 
%\end{align}
%\begin{align}\label{firstM1}
%\|\pa^\al M\|_{L^2_xL^2_v}\gtrsim \|\pa^\al(\wt\rho,\wt u,\wt\th)\|_{L^2_x}^2 -C\E_k(t)\min\{\D_k(t),\E_k(t)\} - C\delta^{\frac{1}{3}}(1+t)^{-\frac{2}{3}}. 
%\end{align}
Then applying $M=P_0F$ and $\E_k(t)\le \ve<1$, we have 
\begin{align}\label{E1}\notag
	\sum_{|\al|=3}\|\pa^\al(\wt\rho,\wt u,\wt\th)\|_{L^2_x}^2 &\lesssim \sum_{|\al|=3}\|w(\al)\frac{\pa^\al M}{\sqrt\mu}\|_{L^2_{x,v}}^2 + \E_k(t)+\delta^{\frac{2}{3}}\\
	&\notag\lesssim \sum_{|\al|=3}\|w(\al)\frac{\pa^\al P_0F_1}{\sqrt\mu}\|_{L^2_{x,v}}^2+ \E_k(t)+\delta^{\frac{2}{3}}\\
	&\lesssim \sum_{|\al|=3}\|\frac{\pa^\al F_1}{\sqrt\mu}\|_{L^2_{x,v}}^2+ \E_k(t)+\delta^{\frac{2}{3}}. 
\end{align}	
Note that $w(\al)M\mu^{-1/2}\le C$. 
Then we have from \eqref{344} and $\sqrt\mu\g=F_1-M-\ol G$ that 
	\begin{align}\label{E01}\notag
		\sum_{|\al|=3}\|w(\al)\pa^\al \g\|_{L^2_{x,v}}^2
		&\lesssim \sum_{|\al|=3}\Big(\|w(\al)\frac{\pa^\al F_1}{\sqrt\mu}\|_{L^2_{x,v}}^2 + \|w(\al)\frac{\pa^\al M}{\sqrt\mu}\|_{L^2_{x,v}}^2 + \|w(\al)\frac{\pa^\al \ol G}{\sqrt\mu}\|_{L^2_{x,v}}^2\Big)\\
		&\lesssim \sum_{|\al|=3}\|w(\al)\frac{\pa^\al F_1}{\sqrt\mu}\|_{L^2_{x,v}}^2 + \E_k(t)+\delta^{\frac{2}{3}}, 
	\end{align}
For the terms in $\E_{int}(t)$ given in \eqref{Eint}, we have 
\begin{align}\notag
	\label{E2}
	\kappa\sum_{1\le|\al|\le2,\,\al_0=0}(\pa^\al \wt u,\pa^\al \na_x\wt\rho)_{L^2_x}
	&\le \|\na_x\wt u\|_{L^2_x}\|\na_x^{2}\wt\rho\|_{L^2_x}\\
	&\le \eta\|\na_x\wt u\|^2_{L^2_x}+C_{\eta}\|\na_x^{2}\wt\rho\|_{L^2_x}^2, 
\end{align}
for any $\eta>0$. 
Choosing $\eta>0$ small enough, the right hand side of \eqref{E2} can be absorbed by the first term in \eqref{Eint} and the left hand side of \eqref{E1}. 
Similarly, we have 
\begin{align}\label{EEE}
	\kappa^4(\wt u,\na_x\wt\rho)_{L^2_x}\lesssim \kappa^4\|\wt u\|_{L^2_x}^2 + \kappa^4\|\na_x\wt\rho\|_{L^2_x}^2. 
\end{align}
It follows from \eqref{eta1} that 
\begin{align}\label{E0}
	\int_{\Omega}\eta(x)\,dx \approx \|(\wt\rho,\wt u,\wt\th)\|_{L^2_x}^2. 
\end{align}
%Combining \eqref{EEE} and \eqref{E0} with the estimate for $\ol\E_k(t)$ in \eqref{Claim1}, we obtain 
%\begin{align*}
%	\E_k(t) -\delta^{\frac{2}{3}} \lesssim \ol\E_k(t) \lesssim \E_k(t) +\delta^{\frac{2}{3}}.  
%\end{align*}
For the macroscopic components, we have from \eqref{Eint} that 
%%from \eqref{pat} that
%for $1\le m\le 4$ that 
\begin{align}\label{E3}
	%	\sum_{|\al|=m}\|\pa^\al(\wt\rho,\wt u,\wt\th)\|_{L^2_x}^2&\lesssim
	\sum_{\substack{1\le|\al|\le2\\\al_0=0}}\|\pa^\al(\wt\rho,\wt u,\wt\th)\|_{L^2_x}^2
	%	 +\sum_{|\al|=m}\|\pa^\al\g\|_{L^2_xL^2_5}^2+\delta^{\frac{1}{2}}\\
	&\approx \sum_{\substack{1\le|\al|\le2\\\al_0=0}}\int_{\Omega}\Big(\frac{R\th}{2\rho^2}|\pa^\al\wt\rho|^2+\frac{1}{2}|\pa^\al\wt u|^2+\th^{-1}|\pa^\al\wt\th|^2\Big)\,dx. 
\end{align} 
For the time derivative, we deduce from \eqref{pat} that for $|\al|\le 1$, 
\begin{align}\label{E3a}\notag
	\|\pa_t\pa^\al(\wt\rho,\wt u,\wt\th)\|_{L^2_x}^2&\lesssim \|\pa^\al\na_x(\wt\rho,\wt u,\wt\th)\|_{L^2_x}^2 + \|\pa^\al\na_x\g\|_{L^2_xL^2_{\gamma/2}}^2 + \delta^{\frac{2}{3}}+ \E_k(t)^2\\
	&\lesssim \ol\E_k(t) + \E_k(t)^2 + \delta^{\frac{2}{3}}. 
\end{align}
Combining \eqref{E1}, \eqref{E01}, \eqref{E2}, \eqref{EEE}, \eqref{E0}, \eqref{E3} and \eqref{E3a} with small enough $\eta>0$ in \eqref{E2}, we obtain 
\begin{align*}
	\E_k(t) \lesssim \ol\E_k(t) + \E_k(t)^2 + \delta^{\frac{2}{3}}. 
\end{align*}
Choosing $\ve>0$ in the {\it a priori} assumption \eqref{priori1} sufficiently small, we obtain 
\begin{align*}
	\E_k(t) \lesssim \ol\E_k(t) + \delta^{\frac{2}{3}}, 
\end{align*}
which gives the lower bound in \eqref{Claim1}. 
%For the estimate on $\E_k(t)$, 

%Choosing $\ve>0$ in \eqref{priori2} sufficiently small and $k\ge k_0$, we obtain the lower bound of $\ol\E_k(t)$ in \eqref{Claim1}, which conclude the claim \eqref{Claim1}. 

	\medskip
	
%	Recall the {\em a priori} assumption \eqref{priori2}: 
%	\begin{align*}
%		\sup_{0\le t\le T}{\E}_{k_0}(t)\le \ve, 
%	\end{align*}
%	for some small $\ve>0$. 
	%we have from \eqref{51} that $\sup_{0\le t\le T}\E_k(t)\le \ve_1+\delta^{\frac{1}{2}}$. 
	Next, for the global existence, choosing $\eta_0>0$ in \eqref{eta}, $\ve>0$ in \eqref{priori1} and $\delta>0$ in \eqref{delta} sufficiently small, we deduce from \eqref{glo2} that 
%	 and \eqref{glo1} that 
	\begin{align}\label{gloa}
		\pa_t\ol\E_k(t) + \lam\D_k(t) \lesssim \delta^{\frac{1}{6}}(1+t)^{-\frac{7}{6}},
	\end{align}
%and 
%\begin{align}\label{glob}
%	\pa_t\ol\E_k(t) + \lam\D_k(t) \lesssim (1+t)^{-2}.
%\end{align}
	for some small constant $\lam>0$. 
	Solving ODE \eqref{gloa} and combining \eqref{Claim1}, we obtain 
	\begin{align*}
%		\sup_{0\le t\le T}\E_k(t)+\lam\int^T_0\D_k(t)\,dt\lesssim 
		\sup_{0\le t\le T}\ol{\E}_k(t)+\lam\int^T_0\D_k(t)\,dt \lesssim \ol{\E}_k(0) + \delta^{\frac{1}{6}}\lesssim \E_k(0)+ \delta^{\frac{1}{6}}. 
	\end{align*}
%Recall that we require $k\ge 1-\gamma$ in Lemma \ref{Lem41}. 
	Choosing $\delta$ in \eqref{delta} and $\ve_0$ in \eqref{small2} sufficiently small, we close the {\it a priori} assumption \eqref{priori1} and obtain \eqref{es1}. 
%	The local existence of the solutions to the non-cutoff Boltzmann equation in cube with specular boundary condition was proved in \cite{Deng2021}. By a straightforward modification of the argument there, we can obtain the local existence of the solutions to non-cutoff Boltzmann equation \eqref{F1} in rectangular duct with specular reflection boundary condition \eqref{specular} and $F(t,x,v)\ge 0$ under the assumption of Theorem \ref{Main1}. For brevity, we omit the proof; see also \cite[Appendix]{Wang2019} for the method of modification. 
%By the {\it a priori} estimate \eqref{es1} and the local existence of solution,
Then it's standard to apply the continuity argument to obtain the global existence and uniqueness of solution to non-cutoff VPB system \eqref{1} or Boltzmann equation \eqref{FB} (satisfying \eqref{specular} and \eqref{Neumann} for the case of rectangular duct). 
	The estimate \eqref{es2} follows from Lemma \ref{Lem21} and estimate \eqref{es1}.
%	 a \eqref{olG2} and \eqref{fast}. 
	
	\smallskip
	
	Next we prove the time-asymptotic stability of rarefaction wave in \eqref{es3}. 
	From dissipation rate in \eqref{gloa}, if we assume $\E_{k-\gamma/2}(0)\le A$ for some $A>0$, then 
	\begin{equation*}
		\left\{
		\begin{aligned}
			&\int^\infty_0\big(\|\na_x(\wt{\rho},\wt u,\wt\th)(t)\|^2_{H^1_x} + \|\<v\>^{k-\gamma/2}\na_x (\g,\f)(t)\|^2_{H^1_xH^s_{\gamma/2}}\big)dt\le C,\\
			&\int^\infty_0\Big|\pa_t\big(\|\na_x(\wt{\rho},\wt u,\wt\th)(t)\|^2_{H^1_x} + \|\<v\>^{k-\gamma/2}\na_x (\g,\f)(t)\|^2_{H^1_xH^s_{\gamma/2}}\big)\Big|\,dt\\
			&\qquad\le \int^\infty_0\big(\|\pa_t\na_x(\wt{\rho},\wt u,\wt\th)(t)\|^2_{H^1_x} + \|\<v\>^{k-\gamma/2}\pa_t\na_x (\g,\f)(t)\|^2_{H^1_xH^s_{\gamma/2}}\\
			&\qquad\qquad\qquad+\|\na_x(\wt{\rho},\wt u,\wt\th)(t)\|^2_{H^1_x} + \|\<v\>^{k-\gamma/2}\na_x (\g,\f)(t)\|^2_{H^1_xH^s_{\gamma/2}}\big)\,dt\le C,
		\end{aligned}
		\right. 
	\end{equation*}
	and hence, 
	\begin{align}\label{tendt}
		\lim_{t\to\infty}\|\na_x(\wt{\rho},\wt u,\wt\th)(t)\|^2_{H^1_x} + \|\<v\>^{k-\gamma/2}\na_x (\g,\f)(t)\|^2_{H^1_xH^s_{\gamma/2}} = 0. 
	\end{align}
	By Sobolev embedding \eqref{Gag}, we know that 
	\begin{multline*}
		\|(\wt{\rho},\wt u,\wt\th)\|^2_{L^\infty_x} + \|\<v\>^k(\g,\f)\|^2_{L^\infty_xL^2_v}
		\lesssim \|(\wt{\rho},\wt u,\wt\th)\|^{17/9}_{H^2_x}\|\na_x(\wt{\rho},\wt u,\wt\th)\|^{1/9}_{L^2_x}\\ + \|\<v\>^k(\g,\f)\|^{17/9}_{H^2_xL^2_v}\|\<v\>^k\na_x(\g,\f)\|^{1/9}_{L^2_xL^2_v}.
	\end{multline*}
Then we have from \eqref{tendt} that 
\begin{align*}
	\lim_{t\to\infty}\big(\|(\wt{\rho},\wt u,\wt\th)\|^2_{L^\infty_x} + \|\<v\>^k(\g,\f)\|^2_{L^\infty_xL^2_v}\big)=0. 
\end{align*}
Also, it follows from Lemma \ref{Lem21} that 
	\begin{align*}
		\lim_{t\to\infty}\big\|\<v\>^k\big(M_{[\bar\rho,\bar u,\bar\th]}(x_1,t) -M_{[\rho^r,u^r,\th^r]}(\frac{x_1}{1+t})\big)\big\|^2_{L^\infty_xL^2_v}=0.
	\end{align*}
%	Combining the above three estimates, 
Therefore, we can obtain the time asymptotic behavior of $F_\pm=M+\ol G+\sqrt\mu\g\pm\sqrt\mu\f$:
	\begin{align*}
	\lim_{t\to\infty}\|\<v\>^{k}\mu^{-1/2}\big(F_\pm-M_{[\rho^r,u^r,\th^r]}(\frac{x_1}{1+t})\big)\|^2_{L^\infty_xL^2_v}=0.
	\end{align*}
%	Thanks to \eqref{time}, the first and second order derivative estimate has time decay $(1+t)^{-\frac{1}{2}}$ and thus, 
%%	$\sup_{t\ge 0}\E_k(t)\le C$ and hence, the second-derivative terms in \eqref{se} are uniformly bounded in $t$. The first-derivative terms in \eqref{se} tends to $0$ as $t\to\infty$ by \eqref{tendt} since $k+\gamma/2\ge 0$. Thus, 
%	\begin{align*}
%		\|(\wt{\rho},\wt u,\wt\th)\|^2_{L^\infty_x} + \|\<v\>^k\g\|^2_{L^\infty_xL^2_v}+ \|\<v\>^k\f\|^2_{L^\infty_xL^2_v}\lesssim (1+A)(1+t)^{-1}.
%	\end{align*}
%Applying \eqref{344}, we have time decay for $\ol G$:
%\begin{align*}
%	\|\ol G\|^2_{L^\infty_xL^2_v}\lesssim \|\na_x\ol G\|_{L^2_xL^2_v}\|\na_x^2\ol G\|_{L^2_xL^2_v}
%	\lesssim (1+t)^{-2}. 
%\end{align*}
%	Using \eqref{time2} to estimate $M_{[\bar\rho,\bar u,\bar\th]} -M_{[\rho^r,u^r,\th^r]}$, we deduce that 
%	\begin{align*}
%		\big\|\<v\>^k\big(M_{[\bar\rho,\bar u,\bar\th]}(x_1,t) -M_{[\rho^r,u^r,\th^r]}(\frac{x_1}{1+t})\big)\big\|^2_{L^\infty_xL^2_v}\lesssim (1+t)^{-1}.
%	\end{align*}
%	Combining the above three estimates, we can obtain the time asymptotic behavior of $F_\pm=M+\ol G+\sqrt\mu\g\pm\sqrt\mu\f$:
%	\begin{align*}
%	\|\<v\>^{k}\mu^{-1/2}\big(F_\pm-M_{[\rho^r,u^r,\th^r]}(\frac{x_1}{1+t})\big)\|^2_{L^\infty_xL^2_v}\lesssim (1+A)(1+t)^{-1}.
%	\end{align*}
	This completes the proof of Theorem \ref{Main1}.

\end{proof}

\medskip
\noindent {\bf Acknowledgements.} 
D.-Q Deng was partially supported by Direct Grant from BIMSA and International Postdoctoral Exchange Fellowship Program (Talent-Introduction Program) (YJ20220056).

\providecommand{\bysame}{\leavevmode\hbox to3em{\hrulefill}\thinspace}
\providecommand{\MR}{\relax\ifhmode\unskip\space\fi MR }
\providecommand{\MRhref}[2]{%
	\href{http://www.ams.org/mathscinet-getitem?mr=#1}{#2}
}
\providecommand{\href}[2]{#2}

\end{document}